\documentclass[letterpaper,11pt]{article}
\usepackage{ece-paper}
\graphicspath{ {./figures/} }

\usepackage{changes}
\definechangesauthor[name={Pengkun Yang}, color=blue]{PY}

\newcommand{\bfH}{{\mathbf{H}}}
\newcommand{\bfM}{{\mathbf{M}}}
\newcommand{\supp}{{\mathrm{supp}}}

\usepackage{algorithm,algorithmic}

\usetikzlibrary{positioning}

\title{Optimal estimation of Gaussian mixtures via denoised method of moments}
\author{Yihong Wu and Pengkun Yang\thanks{Y.~Wu is with the Department of Statistics and Data Science, Yale University, New Haven, CT, \texttt{yihong.wu@yale.edu}.
P.~Yang is with the Department of Electrical Engineering, Princeton University, Princeton, NJ, \texttt{pengkuny@princeton.edu}. 
This work is supported in part by the NSF Grant CCF-1527105, an NSF CAREER award CCF-1651588, and an Alfred Sloan fellowship.
}}

\date{\today}

\begin{document}

\maketitle

\begin{abstract}
The Method of Moments  \cite{Pearson1894} is one of the most widely used methods in statistics for parameter estimation, by means of solving the system of equations that match the population and estimated moments. However, in practice and especially for the important case of mixture models, one frequently needs to contend with the difficulties of non-existence or non-uniqueness of statistically meaningful solutions, as well as the high computational cost of solving large polynomial systems. Moreover, theoretical analyses of the method of moments are mainly confined to asymptotic normality style of results established under strong assumptions.

This paper considers estimating a $k$-component Gaussian location mixture with a common (possibly unknown) variance parameter. To overcome the aforementioned theoretic and algorithmic hurdles, a crucial step is to denoise the moment estimates by projecting to the truncated moment space (via semidefinite programming) before solving the method of moments equations. Not only does this regularization ensures existence and uniqueness of solutions, it also yields fast solvers by means of Gauss quadrature. Furthermore, by proving new moment comparison theorems in the Wasserstein distance via polynomial interpolation and majorization techniques, we establish the statistical guarantees and adaptive optimality of the proposed procedure, as well as oracle inequality in misspecified models. These results can also be viewed as provable algorithms for Generalized Method of Moments \cite{Hansen1982} which involves non-convex optimization and lacks theoretical guarantees. 
\end{abstract}

\tableofcontents

\section{Introduction}
\label{sec:intro}

\subsection{Gaussian mixture model}
Consider a $k$-component Gaussian location mixture model, where each observation is distributed as
\begin{equation}
    \label{eq:model}
    X\sim \sum_{i=1}^k w_i N(\mu_i, \sigma^2).
\end{equation}
Here $w_i$ is the mixing weight such that $w_i\ge 0$ and $\sum_i w_i = 1$, $\mu_i$ is the mean (center) of the $i\Th$ component, and $\sigma$ is the common standard deviation. Equivalently, we can write the distribution of an observation $ X $ as a convolution
\begin{equation}
    \label{eq:model-conv}
    X\sim\nu*N(0,\sigma^2),
\end{equation}
where $\nu = \sum_{i=1}^k w_i \delta_{\mu_i}$ denotes the \emph{mixing distribution}. Thus, we can write $X=U+\sigma Z$, where $U\sim\nu$ is referred to as the latent variable, and $Z$ is standard normal and independent of $U$.

Generally speaking, there are three formulations of learning mixture models:
\begin{itemize}
    \item \textbf{Parameter estimation}:
    estimate the means $\mu_i$'s and the weights $w_i$'s up to a global permutation, and possibly also $\sigma^2$.
    \item \textbf{Density estimation}:
    estimate the probability density function of the Gaussian mixture under certain loss such as $L_2$ or Hellinger distance. This task is further divided into the cases of \emph{proper} and \emph{improper} learning, depending on whether the estimator is required to be a $k$-Gaussian mixture or not; in the latter case, there is more flexibility in designing the estimator but less interpretability.
    \item \textbf{Clustering}: estimate the latent variable of each sample (i.e.~$U_i$, if the $i$th sample is represented as $X_i=U_i+\sigma Z_i$) with a small misclassification rate.
\end{itemize}
It is clear that clustering necessarily relies on the separation between the clusters; 
however, as far as estimation is concerned, both parametric and non-parametric, no separation condition should be needed and one can obtain accurate estimates of the parameters even when clustering is impossible. Furthermore, one should be able to learn from the data the order of the mixture model, that is, the number of components. However, in the present literature, most of the estimation procedures with finite sample guarantees are either clustering-based, or rely on separation conditions in the analysis (e.g.~\cite{balakrishnan2017statistical,lu2016statistical,hopkins2018mixture}). Bridging this conceptual divide is one of the main motivations of the present paper.

Existing methodologies for mixture models are largely divided into likelihood-based and moment-based methods; see \prettyref{sec:related} for a detailed review. Among likelihood-based methods, the \emph{Maximum Likelihood Estimate} (MLE) is not efficiently computable due to the non-convexity of the likelihood function. The most popular heuristic procedure to approximate the MLE is the \emph{Expectation-Maximization} (EM) algorithm \cite{DLR1977}; however, absent separation conditions, no theoretical guarantee is known in general. Moment-based methods include the classical \emph{method of moments} \cite{Pearson1894} and many extensions \cite{Hansen1982,AGHKT2014}; however, the usual method of moments suffers from many issues as elaborated in the next subsection. In the theoretical computer science literature, \cite{KMV2010,MV2010,HP15} proposed moment-based polynomial-time algorithms with provable guarantees; 
however, these methods are typically based on grid search and far from being practical. 
Finding theoretically sound, numerically stable, and computationally efficient version of the method of moments is a major objective of this paper.

\subsection{Failure of the classical method of moments}
\label{sec:mmfail}
The method of moments, commonly attributed to Pearson \cite{Pearson1894}, produces an estimator by equating the population moments to the sample moments.
While conceptually simple, this method suffers from the following problems, especially in the context of mixture models:
\begin{itemize}
    \item \emph{Solvability}: 
    the method of moments entails solving a multivariate polynomial system, in which one frequently encounters non-existence or non-uniqueness of statistically meaningful solutions.
    \item \emph{Computation}: solving moment equations can be computationally intensive. For instance, for $k$-component Gaussian mixture models, the system of moment equations consist of $2k-1$ polynomial equations with $2k-1$ variables.
    \item \emph{Accuracy}: 
    existing statistical literature on the method of moments \cite{VdV00,Hansen1982} either shows mere consistency under weak assumptions, or proves asymptotic normality assuming very strong regularity conditions (so that the delta method works), which generally do not hold in mixture models since the convergence rates can be slower than parametric.
    Some results on nonparametric rates are known (cf.~\cite[Theorem 5.52]{VdV00} and \cite[Theorem 14.4]{Kosorok}) but the conditions are extremely hard to verify.
\end{itemize}

To explain the failure of the vanilla method of moments in Gaussian mixture models, we analyze the following simple two-component example:
\begin{example}
    Consider a Gaussian mixture model with two unit variance components: $X\sim w_1 N(\mu_1, 1)+w_2 N(\mu_2, 1)$. Since there are three parameters $\mu_1,\mu_2$ and $w_1=1-w_2$, we use the first three moments and solve the following system of equations:
    \begin{equation}
        \begin{aligned}
            \Expect_n[X]&=\Expect[X]=w_1\mu_1+w_2\mu_2,\\
            \Expect_n[X^2]&=\Expect[X^2]=w_1\mu_1^2+w_2\mu_2^2+1,\\
            \Expect_n[X^3]&=\Expect[X^3]=w_1\mu_1^3+w_2\mu_2^3+3(w_1\mu_1+w_2\mu_2),
        \end{aligned}    
        \label{eq:MM-X}
    \end{equation}
    where $\Expect_n[X^i] \triangleq \frac{1}{n} \sum_{j=1}^n X_j^i$ denotes the $i^{\rm th}$ moment of the empirical distribution from $n$ \iid~samples. The right-hand sides of \prettyref{eq:MM-X} are related to the moments of the mixing distribution by a linear transformation, which allow us to equivalently rewrite the moment equations \prettyref{eq:MM-X} as:
    \begin{equation}
        \begin{aligned}
            \Expect_n[X]&=\Expect[U]=w_1\mu_1+w_2\mu_2,\\
            \Expect_n[X^2-1]&=\Expect[U^2]=w_1\mu_1^2+w_2\mu_2^2,\\
            \Expect_n[X^3-3X]&=\Expect[U^3]=w_1\mu_1^3+w_2\mu_2^3,
        \end{aligned}
        \label{eq:MM-U}
    \end{equation}
    where $U\sim w_1\delta_{\mu_1}+w_1\delta_{\mu_2}$. 
		It turns out that with finitely many samples, there is always a non-zero chance that \prettyref{eq:MM-U} has no solution; even with infinite samples, it is possible that the solution does not exist with constant probability. To see this, note that, from the first two equations of \prettyref{eq:MM-U}, the solution does not exist whenever
    \begin{equation}
        \label{eq:CS-fail}
        \Expect_n[X^2]-1<\Expect_n^2[X],
    \end{equation}
    that is, the Cauchy-Schwarz inequality fails. Consider the case $\mu_1=\mu_2=0$, i.e., $X\sim N(0,1)$. Then \prettyref{eq:CS-fail} is equivalent to
    \[
        n(\Expect_n[X^2]-\Expect_n^2[X])\le n,
    \]
    where the left-hand side follows the $\chi^2$-distribution with $n-1$ degrees of freedom. Thus, \prettyref{eq:CS-fail} occurs with probability  approaching $\frac{1}{2}$ as $n$ diverges, according to the central limit theorem.
\end{example}

In view of the above example, we note that the main issue with the classical method of moments is the following: although individually each moment estimate is accurate ($\sqrt{n}$-consistent), jointly they do not correspond to the moments of any distribution. Moment vectors satisfy many geometric constraints, \eg, the Cauchy-Schwarz and \Holder{} inequalities, and lie in a convex set known as the \emph{moment space}. Thus for any model parameters, with finitely many samples the method of moments fails with nonzero probability whenever the noisy estimates escape the moment space; even with infinitely many samples, it also provably happens with constant probability when the order of the mixture model is strictly less than $k$, or equivalently, the population moments lie on the boundary of the moment space (see \prettyref{lmm:boundary} for a justification).

\subsection{Main results}
\label{sec:main}
In this paper, we propose the \emph{denoised method of moments} (DMM), which consists of
three main steps: (1) compute noisy estimates of moments, e.g., the unbiased estimates; (2) jointly denoise the moment estimates by projecting them onto the moment space; (3) execute the usual method of moments. 
It turns out that the extra step of projection resolves the three issues of the vanilla version of the method of moments identified in \prettyref{sec:mmfail} simultaneously:
\begin{itemize}
    \item \emph{Solvability}: a unique statistically meaningful solution is guaranteed to exist by the classical theory of moments;
    \item \emph{Computation}: the solution can be found through an efficient algorithm (Gauss quadrature) instead of invoking generic solvers of polynomial systems;
    \item \emph{Accuracy}: the solution provably achieves the optimal rate of convergence, and automatically adaptive to the clustering structure of the population.
\end{itemize}
We emphasize that the denoising (projection) step is explicitly carried out via a convex optimization in \prettyref{sec:known}, and implicitly used in analyzing Lindsay's algorithm \cite{Lindsay1989} in \prettyref{sec:unknown}, when the variance parameter is known and unknown, respectively. 


Following the framework proposed in \cite{Chen95,HK2015}, in this paper we consider the estimation of the mixing distribution, rather than estimating the parameters of each component. The main benefits of this formulation include the following:
\begin{itemize}
    \item Assumption-free: to recover individual components it is necessary to impose certain assumptions to ensure identifiability, such as lower bounds on the mixing weights and separations between components, none of which is needed for estimating the mixing distribution.
    Furthermore, under the usual assumption such as separation conditions, statistical guarantees on estimating the mixing distribution can be naturally translated to those for estimating the individual parameters.
    \item Inference on the number of components: this formulation allows us to deal with misspecified models and estimate the order of the mixture model.
\end{itemize}
Equivalently, estimating the mixing distribution can be viewed as a deconvolution problem, where the goal is to recover the distribution $\nu$ based on observations drawn from the convolution \prettyref{eq:model-conv}.

In this framework, a meaningful and flexible loss function for estimating the mixing distribution is the \emph{$1$-Wasserstein distance} (see \prettyref{sec:why-wass} for a justification in the context of mixture models), defined by 
\begin{equation}
    \label{eq:w1}
    W_1(\nu,\nu')\triangleq \inf \{\Expect[\Norm{X-Y}]: X\sim \nu, Y\sim\nu'\},
\end{equation}
where the infimum is taken over all couplings, i.e., joint distributions of $X$ and $Y$ which are marginally distributed as $\nu$ and $\nu'$ respectively.
In one dimension, the $W_1$ distance coincides with the $L_1$-distance between the cumulative distribution functions (CDFs) \cite{villani.topics}.

Next we present the theoretical results, which can be classified into two categories:
\begin{itemize}
\item To estimate the mixing distribution, our methodology produces moment-based estimators that are optimal in both worst-case (\prettyref{thm:main-W1})
 and adaptive sense (\prettyref{thm:main-W1-adaptive}), for both known and unknown $\sigma$.

\item To estimate the mixture density, the same procedure produces a \emph{proper} estimate that attains the optimal parametric rate (\prettyref{thm:main-density}), despite the fact that the mixing distribution can only be estimated at a non-parametric rate. Moreover, the procedure is robust to model misspecification (\prettyref{thm:oracle}).

\end{itemize}

Throughout the paper, we assume that the number of components satisfies
\begin{equation}
k = O\pth{\frac{\log n}{\log\log n}}.
\label{eq:k}
\end{equation}
If the order of mixture is large, namely, $k\ge \Omega(\frac{\log n}{\log\log n})$, including continuous mixtures, then one can approximate it by a finite mixture with $O(\frac{\log n}{\log\log n})$ components and estimate the mixing distribution using the DMM estimator. Furthermore, this method is optimal (see \prettyref{thm:large-k} at the end of this subsection).
Our main result is the following theorem: 
\begin{theorem}[Optimal rates]
    \label{thm:main-W1}
    Suppose that $|\mu_i|\le M$ for $M\ge 1$ and $\sigma$ is bounded by a constant, and both $k$ and $M$ are given.
    \begin{itemize}
        \item If $\sigma$ is known, then there exists an estimator $\hat \nu$ computable in 
				$O(kn)$ time
				such that, with probability at least $1-\delta$,
        \begin{equation}
            \label{eq:known-variance-W1}
            W_1(\nu,\hat \nu)\le O\pth{Mk^{1.5}\pth{\frac{n}{\log(1/\delta)}}^{-\frac{1}{4k-2}}}.
        \end{equation}
        
        \item If $\sigma$ is unknown, then there exists an estimator $(\hat \nu,\hat \sigma)$ computable in 
				$O(kn)$ time such that, with probability at least $1-\delta$,
        \begin{equation}
            \label{eq:unknown-variance-W1}
            W_1(\nu,\hat \nu)\le O\pth{Mk^{2}\pth{\frac{n}{\log(1/\delta)}}^{-\frac{1}{4k}}},
        \end{equation}
        and 
        \begin{equation}
            \label{eq:unknown-variance-sigma}        
            |\sigma^2-\hat \sigma^2| \leq  O\pth{M^2k\pth{\frac{n}{\log(1/\delta)}}^{-\frac{1}{2k}}}.
        \end{equation}
    \end{itemize}           
\end{theorem}

For fixed for constant $k$, the above convergence rates are minimax optimal as shown in \prettyref{sec:lb}; in the case of known $\sigma$, the optimality of \prettyref{eq:known-variance-W1} has been previously shown in \cite{HK2015}, while the matching lower bounds for \prettyref{eq:unknown-variance-W1}--\prettyref{eq:unknown-variance-sigma} are new. 

Note that the results in \prettyref{thm:main-W1} are proved under the worst-case scenario where the centers can be arbitrarily close, \eg, components completely overlap. It is reasonable to expect a faster convergence rate when the components are better separated, and, in fact, a parametric rate in the best-case scenario
where the components are fully separated and weights are bounded away from zero. 
To capture the clustering structure of the mixture model, we introduce the following definition:
\begin{definition}
    \label{def:sep}
    The Gaussian mixture \prettyref{eq:model} has $k_0$ $(\gamma,\omega)$-separated clusters if there exists a partition $S_1,\dots,S_{k_0}$ of $[k]$ such that
    \begin{itemize}
        \item $|\mu_i - \mu_{i'}| \ge \gamma$ for any $i\in S_\ell$ and $i'\in S_{\ell'}$ such that $\ell\ne \ell'$;
        \item $\sum_{i\in S_\ell}w_i\ge \omega$ for each $\ell$.
    \end{itemize}
    In the absence of the minimal weight condition (i.e.~$\omega=0$), we say the Gaussian mixture has $k_0$ $\gamma$-separated clusters. 	
\end{definition}

The next result shows that the DMM estimators attain the following adaptive rates: 
\begin{theorem}[Adaptive rate]
    \label{thm:main-W1-adaptive}
    Under the conditions of \prettyref{thm:main-W1}, suppose there are $k_0$ $(\gamma,\omega)$-separated clusters such that $\gamma\omega\ge C\epsilon$ for some absolute constant $C>2$, 
		where $\epsilon$ denotes the right-hand side of \prettyref{eq:known-variance-W1} and \prettyref{eq:unknown-variance-W1} when $\sigma$ is known and unknown, respectively.
    \begin{itemize}
        \item If $\sigma$ is known, then, with probability at least $1-\delta$,\footnote{Here $O_k(\cdot)$ denotes a constant factor that depends on $k$ only.}
        \begin{equation}
            \label{eq:known-variance-W1-adaptive}
            W_1(\nu,\hat \nu)\le  O_k\pth{M\gamma^{-\frac{2k_0-2}{2(k-k_0)+1}}\pth{\frac{n}{\log(k/\delta)}}^{-\frac{1}{4(k-k_0)+2}}}.
        \end{equation}
        \item If $\sigma$ is unknown, then, with probability at least $1-\delta$,\footnote{Note that the estimation rate for the mean part $\nu$ is the square root of the rate for estimating the variance parameter $\sigma^2$. Intuitively, this phenomenon is due to the infinite divisibility of the Gaussian distribution: note that for the location mixture model $\nu * N(0,\sigma^2)$ with $\nu\sim N(0,\epsilon^2)$ and $\sigma^2=1$ has the same distribution as that of $\nu\sim \delta_0$ and $\sigma^2=1+\epsilon^2$.}
        \begin{equation}
            \label{eq:unknown-variance-W1-adaptive}
            \sqrt{|\sigma^2-\hat \sigma^2|},~W_1(\nu,\hat \nu)\le  O_k\pth{M\gamma^{-\frac{k_0-1}{k-k_0+1}}\pth{\frac{n}{\log(k/\delta)}}^{-\frac{1}{4(k-k_0+1)}}}.
        \end{equation}
    \end{itemize}
\end{theorem}
For fixed $k,k_0$ and $\gamma$, the rate in \prettyref{eq:known-variance-W1-adaptive} is minimax optimal 
 in view of the lower bounds in \cite{HK2015}; we also provide a simple proof in \prettyref{rmk:lb-adaptive-known} by extending the lower bound argument in \prettyref{sec:lb}. 
If $\sigma$ is unknown, we do not have a matching lower bound for \prettyref{eq:unknown-variance-W1-adaptive}. In fact, in the fully-separated case ($k_0=k$), \prettyref{eq:unknown-variance-W1-adaptive} reduces to $n^{-\frac{1}{4}}$ while the parametric rate is clearly achievable. 
Let us emphasize that, for known $\sigma$, the rates \prettyref{eq:known-variance-W1} and \prettyref{eq:known-variance-W1-adaptive} for fixed $k,k_0$ and $\gamma$ have been previously obtained in \cite{HK2015} by means of the computationally expensive minimum distance estimator; for unknown $\sigma$, the results in \prettyref{eq:unknown-variance-W1}, \prettyref{eq:unknown-variance-sigma}, and \prettyref{eq:unknown-variance-W1-adaptive} are new.

Next we discuss the implication on density estimation (\emph{proper} learning), where the goal is to estimate the density function of the Gaussian mixture by another $k$-Gaussian mixture density. 
Given that the estimated mixing distribution $\hat \nu$ from \prettyref{thm:main-W1}, a natural density estimate is the convolution $\hat f =\hat \nu * N(0,\sigma^2)$. \prettyref{thm:main-density} below shows that the density estimate $\hat f$ is  $O(\frac{1}{\sqrt{n}})$-close to the true density $f$ in the total variation distance $\TV(f,g) \triangleq \frac{1}{2} \|f-g\|_1$.
\begin{theorem}[Density estimation]
    \label{thm:main-density}
    Under the conditions of \prettyref{thm:main-W1}, denote the density of the underlying model by $f = \nu * N(0,\sigma^2)$. If $\sigma$ is given, then there exists an estimate $\hat f$ such that
    \[
        \TV(\hat f,f)\le O_k(\sqrt{\log(1/\delta)/n}),
    \]
    with probability $1-\delta$.
\end{theorem}

So far we have been focusing on well-specified models.
In the case of misspecified models, the data need not be generated from a $k$-Gaussian mixture. In this case, the DMM procedure still reports a meaningful estimate that is close to the best $k$-Gaussian mixture fit of the unknown distribution. This is 
made precise by the next result of oracle inequality type. Analogous results hold for $\chi^2$-divergence, Kullback-Leibler divergence, and Hellinger distance as well.

\begin{theorem}[Misspecified model]
    \label{thm:oracle}
    Assume that $X_1,\ldots,X_n$ is independently drawn from a density $f$ which is 1-subgaussian. Suppose there exists a $k$-component Gaussian location mixture $g$ with a given variance $\sigma^2$ such that $\TV(f,g) \le \epsilon$. Then, there exists an estimate $\hat f$ such that
    \[
        \TV(\hat f,f)\le O_k\pth{\epsilon\sqrt{\log(1/\epsilon)}+\sqrt{\log(1/\delta)/n}},
    \]
    with probability $1-\delta$.
\end{theorem}

 

To conclude this subsection, we present a result for estimating mixtures of an arbitrarily large order, including continuous mixtures, in the case of known variance. In this situation we apply the DMM method to produce a mixture of order $\min\{k,O(\frac{\log n}{\log\log n})\}$. The convergence rate is minimax optimal in view of the matching lower bound in \prettyref{prop:large-k-lb}.
\begin{theorem}[Higher-order mixture]
\label{thm:large-k}
Suppose $|\mu_i|\le M$ for $M\ge 1$ and $\sigma$ is a bounded constant, where $M,\sigma$ are given. Then there exists an estimate $\hat\nu$ such that, with probability at least $1-\delta$,
\[
W_1(\nu,\hat\nu)\le O\pth{M\pth{\frac{\log\log n}{\log n}+\sqrt{\frac{\log(1/\delta)}{n^{1-c}}}}},
\]
for some constant $c<1$.
\end{theorem}

\subsection{Why Wasserstein distance?}
\label{sec:why-wass}
Throughout the paper we consider estimating the mixing distribution $\nu$ with respect to the Wasserstein distance. This is a natural criterion, which is not too stringent to yield trivial result (such as the Kolmogorov-Smirnov (KS) distance\footnote{
    Consider two mixing distributions $\delta_0$ and $\delta_{\epsilon}$ with arbitrarily small $\epsilon$, whose KS distance is always one.
    })
and, at the same time, strong enough to provide meaningful guarantees on the means and weights. In fact, the commonly used criterion $\min_\Pi\sum_i|\mu_i-\hat\mu_{\Pi(i)}|$ over all permutations $\Pi$ is precisely ($k$ times) the Wasserstein distance between two equally weighted distributions \cite{villani.topics}.

Furthermore, we can obtain statistical guarantees on the support sets and weights of the estimated mixing distribution under the usual assumptions in literature \cite{Dasgupta1999,KMV2010,HP15} that include separation between the means and lower bound on the weights. See \prettyref{sec:wass} for a detailed discussion. We highlight the following result, phrased in terms of the parameter estimation error up to a permutation:
\begin{lemma}
    \label{lmm:w1-parameters}
    Let
    \[
        \nu=\sum_{i=1}^k w_i \delta_{\mu_i}, \quad \hat \nu=\sum_{i=1}^k \hat w_i \delta_{\hat \mu_i}.
    \]
    Suppose that $W_1(\nu,\hat\nu)<\epsilon$. Let
    \[
        \epsilon_1=\min\{|\mu_i-\mu_j|,|\hat\mu_i-\hat\mu_j|:1\le i< j\le k\},\quad
        \epsilon_2=\min\{w_i,\hat w_i:i\in [k]\}.
    \]
    If $\epsilon<\epsilon_1\epsilon_2/4$, then, there exists a permutation $\Pi$ such that
    \[
        \Norm{\mu-\Pi\hat \mu}_\infty < \epsilon/\epsilon_2,\quad \Norm{w-\Pi \hat w}_\infty < 2\epsilon/\epsilon_1,
    \]
    where $\mu=(\mu_1,\ldots,\mu_k)$, $w=(w_1,\ldots,w_k)$ denote the atoms and weights of $\nu$, respectively, and $\hat\mu, \hat w$ denote those of $\hat\nu$, 
\end{lemma}

\subsection{Related work}
\label{sec:related}

There exist a vast literature on mixture models, in particular Gaussian mixtures, and the method of moments. For a comprehensive review see \cite{Lindsay1995,Fruhwirth2006}. Below we highlight a few existing results that are related to the present paper.

\paragraph{Likelihood-based methods.} Maximum likelihood estimation (MLE) is one of the most useful method for parameter estimation. Under strong separation assumptions, MLE is consistent and asymptotically normal \cite{RW1984}; however, those assumptions are difficult to verify, and it is computationally hard to obtain the global maximizer due to the non-convexity of the likelihood function in the location parameters.

Expectation-Maximization (EM) \cite{DLR1977} is an iterative algorithm that aims to approximate the MLE. It has been widely applied in Gaussian mixture models \cite{RW1984,XJ1996} and more recently in high-dimensional settings \cite{balakrishnan2017statistical}. In general, this method is only guaranteed to converge to a local maximizer of the likelihood function rather than the global MLE. In practice we need to employ heuristic choices of the initialization \cite{KX2003} and stopping criteria \cite{SMA2000}, as well as possibly data augmentation techniques \cite{meng1997algorithm,PL2001}. Furthermore, its slow convergence rate is widely observed in practice \cite{RW1984,KX2003}. Global convergence of the EM algorithm is recently analyzed by \cite{XHM16,DTZ17} but only in the special case of two equally weighted components. Additionally, the EM algorithm accesses the entire data set in each iteration, which is particularly expensive for large sample size and high dimensions.

Lastly, we mention the nonparametric maximum likelihood estimation (NPMLE) in mixture models proposed by \cite{KW56}, where the maximization is taken over all mixing distributions which need not be $k$-atomic. This is an infinite-dimensional convex optimization problem, which has been studied in \cite{Laird1978,Lindsay1981,Lindsay1995} and more recently in \cite{koenker2014convex} on its computation based on discretization. One of the drawbacks of NPMLE is its lack of interpretability since the solution is a discrete distribution with at most $n$ atoms cf.~\cite[Theorem 2]{koenker2014convex}. Furthermore, few statistical guarantees in terms of convergence rate are available.

\paragraph{Moment-based methods.} The simplest moment-based method is the method of moments (MM) introduced by Pearson \cite{Pearson1894}. The failure of the vanilla MM described in \prettyref{sec:mmfail} has motivated various modifications including, notably, 
the \emph{Generalized Method of Moments} (GMM) introduced by Hansen \cite{Hansen1982}. GMM is a widely used methodology for analyzing economic and financial data (cf.~\cite{Hall2005} for a thorough review). Instead of exactly solving the MM equations, GMM aims to minimize the sum of squared differences between the sample moments and the fitted moments. Despite its nice asymptotic properties \cite{Hansen1982}, GMM involves a non-convex optimization problem which is computationally challenging to solve. 
In practice, heuristics such as gradient descent are used \cite{Chausse2010} which converge slowly and lack theoretical guarantees.

For Gaussian mixture models (and more generally finite mixture models), our results can be viewed as a solver for GMM which is provably exact and computationally efficient, improving over existing heuristic methods in terms of both speed and accuracy significantly; this is another algorithmic contribution of the present paper. 
The key is to switch the view from optimizing over $k$-atomic mixing distributions (which is non-convex) to moment space (which is convex and efficiently optimizable via SDP).
We also note that minimizing the sum of squares in GMM is not crucial and minimizing any distance yields the same theoretical guarantee. We discuss the connections to GMM in details in \prettyref{sec:known}.

There are a number of recent work in the theoretical computer science literature on provable results for moment-based estimators in Gaussian location-scale mixture models, see, e.g., \cite{MV2010,KMV2010,BS10,HP15,li2017robust}.
For instance, the algorithm \cite{MV2010} is based on exhaustive search over the discretized parameter space such that the population moments is close to the empirical moments. In addition to being computationally expensive, this method achieves the estimation accuracy $n^{-C/k}$ for some constant $C$, which is suboptimal in view of \prettyref{thm:main-W1}.
By carefully analyzing Pearson's method of moments equations \cite{Pearson1894},  \cite{HP15} showed that the optimal rate for two-component location-scale mixtures is $\Theta(n^{-1/12})$; however, this approach is difficult to generalize to higher order mixtures.
Finally, for moment-based methods in multiple dimensions, such as spectral and tensor decomposition, we defer the discussion to \prettyref{sec:multiple}.

\paragraph{Minimum distance estimators.} In the case of known variance, the minimum distance estimator is studied by \cite{DK68,Chen95,HK2015}. Specifically, the estimator is a $k$-atomic distribution $\hat\nu$ such that $\hat\nu*N(0,\sigma^2)$ is the closest to the empirical distribution of the samples in certain distance. The minimax optimal rate $O(n^{-\frac{1}{4k-2}})$ for estimating the mixing distribution under the Wasserstein distance is shown in \cite{HK2015} (which corrects the previous result in \cite{Chen95}), by bounding the $W_1$ distance between the mixing distributions in terms of the KS distance of the Gaussian mixtures \cite[Lemma 4.5]{HK2015}. 
However, the minimum distance estimator is in general computationally expensive and suffers from the same non-convexity issue of the MLE. 
In contrast, denoised method of moments is efficiently computable and adaptively achieves the optimal rate of accuracy as given in \prettyref{thm:main-W1-adaptive}.
For arbitrary Gaussian location mixtures in one dimension, the minimum distance estimator was considered in \cite{edelman1988estimation} in the context of empirical Bayes. Under the assumptions of bounded first moment, it is shown in \cite[Corollary 2]{edelman1988estimation} that the mixing distribution can be estimated at rate $O((\log n)^{-1/4})$ under the $L_2$-distance between the CDFs; this loss is, however, weaker than the $W_1$-distance (i.e.~$L_1$ distance between the CDFs).

\paragraph{Density estimation}

If the estimator is allowed to be any density (\emph{improper} learning), it is known that as long as the mixing distribution has a bounded support, the rate of convergence is close to parametric regardless of the number of components. Specifically, the optimal squared $L_2$-risk is found to be $\Theta(\frac{\sqrt{\log n}}{n})$ \cite{kim2014minimax}, achieved by the kernel density estimator designed for analytic densities \cite{ibragimov2001estimation}. 
As mentioned before, \emph{proper} density estimate (which is required to be a $k$-Gaussian mixture) is more desirable for the sake of interpretability; however, 
finding the $k$-Gaussian mixture that best approximates a given function such as a kernel density estimate can be computationally challenging due to, again, the non-convexity in the location parameters. 
In this regard, another contribution of Theorems \ref{thm:main-density} and \ref{thm:oracle} is the observation that proper and near optimal estimates/approximates can be found efficiently via the method of moments.
Finally, we note that MLE for estimating the density of general Gaussian mixtures has been studied in \cite{genovese.wasserman,GV2001}.

\subsection{Notations}
A discrete distribution supported on $k$ atoms is called a $k$-atomic distribution. 
The expectation of a given function $f$ under a distribution $\mu$ is denoted by $\Expect_{\mu}f=\Expect_\mu [f(X)]=\int f(x)\mu(\diff x)$, and the subscript $\mu$ may be omitted if it is specified from the context. 
The empirical mean of $f$ from $n$ samples is denoted as $\Expect_n[f(X)] = \frac{1}{n} \sum_{i=1}^n f(X_i)$, where $X_1,\ldots,X_n$ are \iid~copies of $X$.
The $ r\Th $ moment of a distribution $ \mu $ is denoted by $ m_r(\mu) \triangleq \Expect_{\mu}X^r$.
The moment matrix associated with $ m_0,m_1,\dots,m_{2r} $ is a Hankel matrix of order $ r+1 $:
\begin{equation}
    \label{eq:moment-matrix}
    \bfM_r=
    \begin{bmatrix}
        m_0 & m_1 & \cdots & m_r \\
        m_1 & m_2 & \cdots & m_{r+1} \\
        \vdots  & \vdots  & \ddots & \vdots  \\
        m_r & m_{r+1} & \cdots & m_{2r} 
    \end{bmatrix}.
\end{equation}
For matrices $A \succeq B$ stands for $A-B$ being positive semidefinite.
The interval $[x-a,x+a]$ is abbreviated as $[x\pm a]$.
For any $x,y\in\reals$, $x\wedge y \triangleq \min\{x,y\}$ and $(x)_+ \triangleq \max\{x,y\}$.
For two vectors $x=(x_1,\dots,x_n)$ and $y=(y_1,\dots,y_n)$, let $\inner{x,y}\triangleq \sum_i x_iy_i$.
 A distribution $\pi$ is called $\sigma$-subgaussian if $\Expect_{\pi}[e^{tX}]\le \exp(t^2\sigma^2/2)$ for all $t\in\reals$.
We use standard big-$O$ notations, \eg, for two positive sequence $\{a_n\}$ and $\{b_n\}$, $a_n=O(b_n)$ if $a_n\le Cb_n$ for some constant $C>0$; $a_n=\Omega(b_n)$ if $b_n=O(a_n)$; $a_n=\Theta(b_n)$ if $a_n=O(b_n)$ and $a_n=\Omega(b_n)$. We write $a_n=O_\beta(b_n)$ if $C$ depends on another parameter $\beta$.

\subsection{Organization}
The paper is organized as follows. In \prettyref{sec:prelim} we provide some basic results of the theory of moments and the Wasserstein distance. In \prettyref{sec:w1-moments} we introduce the moment comparison theorems, which bound the Wasserstein distance between two discrete distributions in terms of the discrepancy of their moments. These are key results to prove the main theorems. In \prettyref{sec:estimate}, we propose estimation algorithms and provide their statistical guarantees. Matching minimax lower bounds are given in \prettyref{sec:lb}. \prettyref{sec:exp} contains numerical experiments and comparison with other methods such as the EM algorithm. \prettyref{sec:discuss} discusses extensions and open problems including location-scale mixtures and the multivariate case. Proofs are given in \prettyref{sec:proof}; in particular, \prettyref{sec:poly} contains a brief discussion on polynomial interpolation and majorization, which play a crucial role in the proof. 
Auxiliary results are deferred to \prettyref{app:aux}.

\section{Preliminaries}
\label{sec:prelim}

\subsection{Moment space, SDP characterization, and Gauss quadrature}
\label{sec:moment}
The theory of moments plays a key role in the developments of analysis, probability, statistics, and optimization.  See the classics \cite{ST1943,KS1953} and the recent monographs \cite{Lasserre2009,Schmudgen17} for a detailed treatment. Below, we briefly review a few basic facts that are related to this paper.

The $r\Th$ moment vector of a distribution $\pi$ is a $r$-tuple $\bfm_r(\pi)=(m_1(\pi),\dots,m_r(\pi))$. The $r\Th$ moment space on $K\subseteq \reals$ is defined as 
\[
    \calM_r(K)=\{\bfm_r(\pi):\pi \text{ is supported on }K\},
\]
which is the convex hull of $\{(x,x^2,\dots,x^r):x\in K\}$. A valid moment vector satisfies many geometric constraints such as the Cauchy-Schwarz and \Holder{} inequalities. When $K=[a,b]$ is a compact interval, $\calM_r([a,b])$ is completely described by (see \cite[Theorem 3.1]{ST1943}, and also \cite{KS1953,Lasserre2009}) the following condition:
\begin{equation}
    \label{eq:moment-psd}
    \begin{cases}
        ~\bfM_{0,r}\succeq 0,\quad (a+b)\bfM_{1,r-1}\succeq ab\bfM_{0,r-2}+\bfM_{2,r}, & \text{ $r$ even},\\
        ~b\bfM_{0,r-1}\succeq \bfM_{1,r} \succeq a\bfM_{0,r-1}, & \text{ $r$ odd},
    \end{cases}
\end{equation}
where $\bfM_{i,j}$ denotes the Hankel matrix with entries $m_i,m_{i+1},\dots,m_j$:
\[
    \bfM_{i,j}=
    \begin{bmatrix}
        m_i & m_{i+1} & \cdots & m_{\frac{i+j}{2}} \\
        m_{i+1} & m_{i+2} & \cdots & m_{\frac{i+j}{2}+1} \\
        \vdots  & \vdots  & \ddots & \vdots  \\
        m_{\frac{i+j}{2}} & m_{\frac{i+j}{2}+1} & \cdots & m_{j} 
    \end{bmatrix}.
\]

\begin{example}[Moment spaces on {$[0,1]$}]
For the first two moments, $\calM_2([0,1])$ is simply described by $m_1\ge m_2\ge 0$ and $m_2\ge m_1^2 $. 
For $r=3$, according to \prettyref{eq:moment-psd}, $\calM_3([0,1])$ is described by
    \[
        \begin{bmatrix}
            1   & m_1   \\
            m_1 & m_2  
        \end{bmatrix}
        \succeq 
        \begin{bmatrix}
            m_1 & m_2   \\
            m_2 & m_3  
        \end{bmatrix}
        \succeq 0.
    \]
    Using Sylvester's criterion (see \cite[Theorem 7.2.5]{horn-2nd}), they are equivalent to 
    \begin{align*}
        &0\le m_1\le 1, \quad m_2\ge m_3\ge 0,\\
        &m_1m_3\ge m_2^2, \quad (1-m_1)(m_2-m_3)\ge (m_1-m_2)^2.
    \end{align*}
    The necessity of the above inequalities is apparent: the first two follow from the support being $[0,1]$, and the last two follow from the Cauchy-Schwarz inequality. It turns out that they are also sufficient.
\end{example}

Moment matrices of discrete distributions satisfy more structural properties. For instances, the moment matrix of a $k$-atomic distribution of any order is of rank at most $k$, and is a deterministic function of ${\bf m}_{2k-1}$; the number of atoms can be characterized using the determinants of moment matrices (see \cite[p.~362]{Uspensky37} or \cite[Theorem 2A]{Lindsay1989}) as follows:   
\begin{theorem}
    \label{thm:supp-detM}
		$(m_1,\dots,m_{2r})$ are the first $2r$ moments
		of a distribution with exactly $ r $ points of support if and only if 
    $ \det(\bfM_{r-1})>0 $          and $ \det(\bfM_r)=0 $.
\end{theorem}

\begin{algorithm}[t]
    \caption{Quadrature rule}
    \label{algo:quadrature}
    \begin{algorithmic}
        \REQUIRE a valid moment vector $ (m_1,\dots,m_{2k-1}) $.
        \ENSURE nodes $ x=(x_1,\dots,x_k) $ and weights $ w=(w_1,\dots,w_k) $.
        \STATE Define the following degree-$k$ polynomial $ P $
        \begin{equation*}
            P(x)=\det\begin{bmatrix}
                1 & m_1 & \cdots & m_k \\
                \vdots  & \vdots  & \ddots & \vdots  \\
                m_{k-1} & m_{k} & \cdots & m_{2k-1} \\
                1 & x & \cdots & x^{k} 
            \end{bmatrix}.
        \end{equation*}
        \STATE Let the nodes $ (x_1,\dots,x_k) $ be the roots of the polynomial $ P $.
        \STATE Let the weights $ w=(w_1,\dots,w_k)  $ be
        \begin{equation*}
            w=\begin{bmatrix}
                1 & 1 & \cdots & 1 \\
                x_1 & x_2 & \cdots & x_{k} \\
                \vdots  & \vdots  & \ddots & \vdots  \\
                x_1^{k-1} & x_2^{k-1} & \cdots & x_k^{k-1} 
            \end{bmatrix}^{-1}
            \begin{bmatrix}
                1 \\
                m_1 \\
                \vdots  \\
                m_{k-1} 
            \end{bmatrix}.
        \end{equation*}
    \end{algorithmic}
\end{algorithm}
Next we discuss the closely related notion of \emph{Gauss quadrature}, which is a discrete approximation for a given distribution in the sense of moments and plays an important role in the execution of the DMM estimator. Given $\pi$ supported on an interval $[a,b]\subseteq \reals$, a $k$-point Gauss quadrature is a $k$-atomic distribution $\pi_k=\sum_{i=1}^kw_i\delta_{x_i}$, also supported on $[a,b]$, such that, for any polynomial $P$ of degree at most $2k-1$,
\begin{equation}
    \label{eq:quadrature-goal}
    \Expect_\pi P = \Expect_{\pi_k}P=\sum_{i=1}^{k}w_iP(x_i).
\end{equation}
Gauss quadrature is known to always exist and is uniquely determined by ${\bf m}_{2k-1}(\pi)$ (cf.~e.g.~\cite[Section 3.6]{stoer.2002}), which shows that any valid moment vector of order $2k-1$ can be realized by a unique $k$-atomic distribution.
A basic algorithm to compute Gauss quadrature is \prettyref{algo:quadrature} \cite{GW1969} and many algorithms with improved computational efficiency and numerical stability have been proposed; cf.~\cite[Chapter 3]{Gautschi2004}.

\subsection{Wasserstein distance}
\label{sec:wass}
A central quantity in the theory of optimal transportation, the Wasserstein distance is the minimum cost of mapping one distribution to another.
In this paper, we will be mainly concerned with the 1-Wasserstein distance defined in \prettyref{eq:w1}, which can be equivalently expressed, through the Kantorovich duality \cite{villani.topics}, as
\begin{equation}
    \label{eq:W1-dual}
    W_1(\nu,\nu')=\sup \{\Expect_\nu[\varphi]-\Expect_{\nu'}[\varphi]:\varphi\text{ is 1-Lipschitz}\}.
\end{equation}
The optimal coupling in \prettyref{eq:w1} has many equivalent characterization \cite{villani.topics} but is often difficult to compute analytically in general. Nevertheless, the situation is especially simple for distributions on the real line, where the quantile coupling is known to be optimal and hence
\begin{equation}
    \label{eq:W1-CDF}
    W_1(\nu,\nu')=\int |F_\nu(t)-F_{\nu'}(t)|\diff t,
\end{equation}
where $F_\nu$ and $F_{\nu'}$ denote the CDFs of $\nu$ and $\nu'$, respectively. Both \prettyref{eq:W1-dual} and \prettyref{eq:W1-CDF} provide convenient characterizations to bound the Wasserstein distance in \prettyref{sec:w1-moments}.

As previously mentioned in \prettyref{sec:why-wass}, two discrete distributions close in the Wasserstein distance have similar support sets and weights. This is made precise by \prettyref{lmm:w1-hausdorff} and \ref{lmm:w1-prokhorov} next:

\begin{lemma}
    \label{lmm:w1-hausdorff}
    Suppose $\nu$ and $\nu'$ are discrete distributions supported on $S$ and $S'$, respectively. Let 
    $\epsilon=\min\{\nu(x):x\in S\}\wedge\min\{\nu'(x):x\in S'\}$. Then,
    \[
        d_H(S,S')\le W_1(\nu,\nu')/\epsilon ,
    \]
    where $d_H$ denotes the Hausdorff distance defined as
    \begin{equation}
        \label{eq:hausdorff}
        d_H(S,S')=\max\sth{\sup_{x\in S}\inf_{x'\in S'}|x-x'|,\sup_{x'\in S'}\inf_{x\in S}|x-x'|}.
    \end{equation}
\end{lemma}
\begin{proof}
    For any coupling $P_{XY}$ such that $X\sim\nu$ be $Y\sim\nu'$,
    \begin{align*}
        \Expect |X-Y| = \sum_x \Prob[X=x] \Expect[|X-Y| | X=x]
        \ge \sum_x \epsilon\cdot \inf_{x'\in S'} |x-x'| \ge \epsilon \cdot\sup_{x\in S} \inf_{x'\in S'} d(x,x').
    \end{align*}
    Interchanging $X$ and $Y$ completes the proof.
\end{proof}

\begin{lemma}
    \label{lmm:w1-prokhorov}
    For any $\delta>0$,
    \begin{gather*}
        \nu(x)- \nu'([x\pm\delta])\le W_1(\nu,\nu')/\delta,\quad 
        \nu'(x)- \nu([x\pm\delta])\le W_1(\nu,\nu')/\delta. 
    \end{gather*}
\end{lemma}
\begin{proof}
    Using the optimal coupling $P_{XY}^*$ such that $X\sim\nu$ be $Y\sim\nu'$, applying Markov inequality yields that
    \[
        \Prob[|X-Y|>\delta]\le \Expect|X-Y|/\delta=W_1(\nu,\nu')/\delta.
    \]
    By Strassen's theorem (see \cite[Corollary 1.28]{villani.topics}), for any Borel set $B$, we have $\nu(B)\le \nu'(B^\delta)+W_1(\nu,\nu')/\delta$ and $\nu'(B)\le \nu(B^\delta)+W_1(\nu,\nu')/\delta$, where $B^\delta\triangleq \{x:\inf_{y\in B}|x-y|\le \delta\}$ denotes the $\delta$-fattening of $B$. The conclusion follows by considering a singleton $B=\{x\}$.
\end{proof}

\prettyref{lmm:w1-hausdorff} and \ref{lmm:w1-prokhorov} together yield a bound on the parameter estimation error (up to a permutation) in terms of the Wasserstein distance, which was previously given in \prettyref{lmm:w1-parameters}: 
\begin{proof}
    Denote the support sets of $\nu$ and $\nu'$ by $S=\{\mu_1,\ldots,\mu_k\}$ and $S'=\{\hat\mu_1,\ldots,\hat\mu_k\}$, respectively. Applying \prettyref{lmm:w1-hausdorff} yields that $d_H(S,S')<\epsilon/\epsilon_2$, which is less than $\epsilon_1/4$ by the assumption $\epsilon<\epsilon_1\epsilon_2/4$. Since $|\mu_i-\mu_j|\ge \epsilon$ for every $i\ne j$, then there exists a permutation $\Pi$ such that 
    \[
        \Norm{\mu-\Pi\hat \mu}_\infty<\epsilon/\epsilon_2.
    \]
    Applying \prettyref{lmm:w1-prokhorov} twice with $\delta=\epsilon/2$, $x=\mu_i$ and $x=(\Pi\hat\mu)_i$, respectively, we obtain the desired
    \[
        w_i-(\Pi \hat w)_i\le 2\epsilon/\epsilon_1,\quad (\Pi \hat w)_i-w_i\le 2\epsilon/\epsilon_1. \qedhere
    \]
\end{proof}

\section{Optimal transport and moment comparison theorems}
\label{sec:w1-moments}
A discrete distribution with $k$ atoms has $2k-1$ free parameters. Therefore it is reasonable to expect that it can be uniquely determined by its first $2k-1$ moments. Indeed, we have the following simple identifiability results for discrete distributions: 
\begin{lemma}
    \label{lmm:identify} 
	Let $\nu$ and $\nu'$ be distributions on the real line.
    \begin{enumerate}
        \item If $\nu$ and $\nu'$ are both $k$-atomic, then $\nu=\nu'$ if and only if $\bfm_{2k-1}(\nu)=\bfm_{2k-1}(\nu')$.
        \item If $\nu$ is $k$-atomic, then $\nu=\nu'$ if and only if $\bfm_{2k}(\nu)=\bfm_{2k}(\nu')$.
    \end{enumerate}    
\end{lemma}

In the context of statistical estimation, we only have access to samples and noisy estimates of moments. To solve the inverse problems from moments to distributions, our theory relies on the following stable version of the identifiability in \prettyref{lmm:identify}, which show that closeness of moments implies closeness of distributions in Wasserstein distance.
In the sequel we refer to Propositions \ref{prop:stable1} and \ref{prop:stable2} as moment comparison theorems.
\begin{prop}
    \label{prop:stable1}
    Let $\nu$ and $\nu'$ be $k$-atomic distributions supported on $[-1,1]$. If $|m_i(\nu)-m_i(\nu')|\le \delta$ for $i=1,\dots,2k-1$, then
    \[
        W_1(\nu,\nu')\le O\pth{k\delta^{\frac{1}{2k-1}}}.
    \]
\end{prop}
\begin{prop}
    \label{prop:stable2}        
    Let $\nu$ be a $k$-atomic distribution supported on $[-1,1]$. If $|m_i(\nu)-m_i(\nu')|\le \delta$ for $i=1,\dots,2k$, then
    \[
        W_1(\nu,\nu')\le O\pth{k\delta^{\frac{1}{2k}}}.
    \]
\end{prop}

\begin{remark}
    \label{rmk:stable}
    The exponents in \prettyref{prop:stable1} and \ref{prop:stable2} are optimal. To see this, we first note that the number of moments needed for identifiability in \prettyref{lmm:identify} cannot be reduced:
    \begin{enumerate}
        \item Given any $2k$ distinct points, there exist two $k$-atomic distributions with disjoint support sets but identical first $2k-2$ moments (see \prettyref{lmm:match2k-2}).
        \item Given any continuous distribution, its $k$-point Gauss quadrature is $k$-atomic and have identical first $2k-1$ moments (see \prettyref{sec:moment}).
    \end{enumerate}
    By the first observation, there exist two $k$-atomic distributions $\nu$ and $\nu'$ such that 
    \[
        m_i(\nu)=m_i(\nu'),~i=1,\ldots,2k-2,\quad |m_{2k-1}(\nu)-m_{2k-1}(\nu')|= c_k,
        \quad W_1(\nu,\nu')=d_k,
    \]
    where $c_k$ and $d_k$ are strictly positive constants that depend on $k$. Let $\tilde\nu$ and $\tilde\nu'$ denote the distributions of $\epsilon X$ and $\epsilon X'$ such that $X\sim\nu$ and $X'\sim\nu'$, respectively. Then, we have
    \[
        \max_{i\in [2k-1]}|m_i(\tilde\nu)-m_i(\tilde\nu)|= \epsilon^{2k-1}c_k,
        \quad W_1(\tilde\nu,\tilde\nu')=\epsilon d_k.
    \]
    This concludes the tightness of the exponent in \prettyref{prop:stable1}. Similarly, the exponent in \prettyref{prop:stable2} is also tight using the second observation.
\end{remark}

\begin{remark}
    \label{rmk:moments-comparison}
    Classical moments comparison theorems aim to show convergence of distributions by comparing a \emph{growing} number of moments. For example, Chebyshev's theorem (see \cite[Theorem 2]{Diaconis1987}) states if ${\bf m}_r(\pi)={\bf m}_r(N(0,1))$, then 
    \[
        \sup_{x\in\reals}|F_{\pi}(x)-\Phi(x)|\le \sqrt{\frac{\pi}{2r}},
    \]
    where $F_{\pi}$ and $\Phi$ denote the CDFs of $\pi$ and $N(0,1)$, respectively. 
    For two compactly supported distributions, the above estimate can be sharpened to $O(\frac{\log r}{r})$ \cite{Krawtchouk1932}. In contrast, in the context of estimating finite mixtures we are dealing with discrete mixing distributions, which can be identified by a \emph{fixed} number of moments. However, with finitely many samples, it is impossible to exactly determine the moments, and measuring the error in the KS distance leads to triviality (see \prettyref{sec:why-wass}). It turns out that $W_1$-distance is a suitable metric for this purpose, and the closeness of moments does imply the closeness of distribution in the $W_1$ distance, which is the integrated difference ($L_1$-distance) between CDFs as opposed the uniform error ($L_\infty$-distance). An upper bound on the $W_1$ distance is obtained in \cite{KV17} (see also \prettyref{lmm:w1-compare}) involving the differences  of the first $k$ moments and a $\Theta(\frac{1}{k})$ term that does not vanish for fixed $k$. The discrepancy between parameters of two Gaussian mixtures is obtained by comparing moments in \cite{KMV2010,MV2010}, which is not applicable for estimating the mixing distribution.
\end{remark}

\section{Estimators and statistical guarantees}
\label{sec:estimate}
In this section we introduce the DMM estimators and prove the statistical bounds announced in \prettyref{sec:intro}. To keep the presentation simple, we focus on estimators with expected risk guarantees.
To obtain a high-probability bound, one can employ the usual technique of dividing the samples into batches, applying the unbiased moment estimator to each batch and taking the median, then finally executing the DMM method to estimate the mixing distribution. 
\subsection{Known variance}
\label{sec:known}
The denoised method of moments for estimating Gaussian location mixture models \prettyref{eq:model-conv} with known variance parameter $\sigma^2$ consists of three main steps:
\begin{enumerate}
    \item estimate ${\bf m}_{2k-1}(\nu)$ by $\tilde m = (\tilde m_1,\dots,\tilde m_{2k-1})$ (using Hermite polynomials);
    \item denoise $\tilde m$ by its projection $\hat m$ onto the moment space (semidefinite programming);
    \item find a $k$-atomic distribution $\hat \nu$ such that ${\bf m}_{2k-1}(\hat \nu)=\hat m$ (Gauss quadrature).
\end{enumerate}
The complete algorithm is summarized in \prettyref{algo:known}.

\begin{algorithm}[ht]
    \caption{Denoised method of moments (DMM) with known variance.}
    \label{algo:known}
    \begin{algorithmic}[1]
        \REQUIRE $n$ independent samples $X_1,\dots,X_n$, order $k$, variance $\sigma^2$, interval $I=[a,b]$. 
        \ENSURE estimated mixing distribution.
        \FOR{$r=1$ \TO $2k-1$}\label{line:for}
        \STATE{$\hat{\gamma}_r=\frac{1}{n}\sum_i X_i^r$}
        \STATE{$\tilde{m}_r=r!\sum_{i=0}^{\Floor{r/2}}\frac{(-1/2)^i}{i!(r-2i)!}\hat{\gamma}_{r-2i}\sigma^{2i}$}\label{line:tildem}
        \ENDFOR\label{line:endfor}
        \STATE{Let $\hat m$ be the optimal solution of the following:}
        \begin{equation}
            \label{eq:project}
            \min\{\Norm{\tilde m - \hat m}: \hat m~\text{satisfies \prettyref{eq:moment-psd}}\},
        \end{equation} 
        where $\tilde m=(\tilde m_1,\dots,\tilde m_{2k-1})$.
        \STATE{Report the outcome of the Gauss quadrature (\prettyref{algo:quadrature}) with input $\hat m$.}
    \end{algorithmic}
\end{algorithm}

We estimate the moments of the mixing distribution in lines~\ref{line:for} to \ref{line:endfor}. The unique unbiased estimators for the polynomials of the mean parameter in a Gaussian location model are Hermite polynomials 
\begin{equation}
    \label{eq:Hermite1}
    H_r(x)=r!\sum_{j=0}^{\floor{r/2}} \frac{(-1/2)^j }{j!(r-2j)!}x^{r-2j},
\end{equation}
such that $\Expect H_r(X)=\mu^r$ when $X\sim N(\mu,1)$. Thus, if we define
\begin{equation}
    \label{eq:Hermite}
    \gamma_r(x,\sigma)=\sigma^rH_r(x/\sigma)=r!\sum_{j=0}^{\floor{r/2}} \frac{(-1/2)^j }{j!(r-2j)!}\sigma^{2j}x^{r-2j},
\end{equation}
then $\Expect \gamma_r(X,\sigma)=\mu^r$ when $X\sim N(\mu,\sigma^2)$. Hence, by linearity, $\tilde m_r$ is an unbiased estimate of $m_r(\nu)$. The variance of $\tilde m_r$ is bounded by the following lemma: 
\begin{lemma}
    \label{lmm:var-tildem}
    If $X_1,\dots,X_n\iiddistr \nu*N(0,\sigma^2)$ and $\nu$ is supported on $[-M,M]$, then
    \[
        \var[\tilde m_r]\le \frac{1}{n}(O(M+\sigma\sqrt{r}))^{2r}.
    \]
\end{lemma}

As observed in \prettyref{sec:mmfail}, the major reason for the failure of the usual method of moments is that the unbiased estimate $\tilde m$ needs not constitute a legitimate moment sequence, despite the consistency of each individual $\tilde m_i$. To resolve this issue, we project $\tilde m$ to the moment space using \prettyref{eq:project}. As explained in \prettyref{sec:moment}, \prettyref{eq:moment-psd} consists of positive semidefinite constraints, and thus the optimal solution of \prettyref{eq:project} can be obtained by semidefinite programming (SDP).\footnote{
    The formulation \prettyref{eq:project} with Euclidean norm can already be implemented in popular modeling languages for convex optimization problem such as {\sf CVXPY} \cite{cvxpy}. A standard form of SDP is given in \prettyref{app:sdp}.
}
 In fact, it suffices to solve a \emph{feasibility} program and find any valid moment vector $\hat m$ that is within the desired $\frac{1}{\sqrt{n}}$ statistical accuracy.

Now that $\hat m$ is indeed a valid moment sequence, we use the Gauss quadrature introduced in \prettyref{sec:moment} (see \prettyref{algo:quadrature} in \prettyref{sec:moment}) to find the unique $k$-atomic distribution $\hat\nu$ such that ${\bf m}_{2k-1}(\hat \nu)=\hat m$. Using \prettyref{algo:known}, $\tilde m$ is computed in $O(kn)$ time, the semidefinite programming is solvable in $O(k^{6.5})$ time using the interior-point method (see \cite{WSV2012}), and the Gauss quadrature can be evaluated in $O(k^3)$ time \cite{GW1969}. 
In view of the global assumption \prettyref{eq:k}, \prettyref{algo:known} can be executed in $O(kn)$ time.

We now prove the statistical guarantee \prettyref{eq:known-variance-W1} for the DMM estimator previously announced in \prettyref{thm:main-W1}:
\begin{proof}
    By scaling it suffices consider $M=1$. We use \prettyref{algo:known} with Euclidean norm in \prettyref{eq:project}. Using the variance of $\tilde m$ in \prettyref{lmm:var-tildem} and Chebyshev inequality yield that, for each $r=1,\ldots,2k-1$, with probability $1-\frac{1}{8k}$, 
    \begin{equation}
        \label{eq:tilde-mr-rate}
        |\tilde m_r-m_r(\nu)| \le \sqrt{k/n}(c\sqrt{r})^r,
    \end{equation}
    for some absolute constant $c$. By the union bound, with probability $3/4$, \prettyref{eq:tilde-mr-rate} holds simultaneously for every $r=1,\dots,2k-1$, and thus
    \[
        \Norm{\tilde m-{\bf m}_{2k-1}(\nu)}_2\le \epsilon,\quad \epsilon\triangleq \frac{(\sqrt{ck})^{2k+1}}{\sqrt{n}}.
    \]
    Since ${\bf m}_{2k-1}(\nu)$ satisfies \prettyref{eq:moment-psd} and thus is one feasible solution for \prettyref{eq:project}, we have $\Norm{\tilde m-\hat m}_2\le \epsilon$. Note that $\hat m={\bf m}_{2k-1}(\hat \nu)$. Hence, by triangle inequality, we obtain the following statistical accuracy:
    \begin{equation}
        \label{eq:bfm-example}
        \Norm{{\bf m}_{2k-1}(\hat \nu)-{\bf m}_{2k-1}(\nu)}_2\le \epsilon,
    \end{equation}
    Applying \prettyref{prop:stable1} yields that, with probability $3/4$,
    \[
        W_1(\hat\nu,\nu)\le O\pth{k^{1.5}n^{-\frac{1}{4k-2}}}.
    \]
    The confidence $1-\delta$ in \prettyref{eq:known-variance-W1} can be obtained by the usual ``median trick'': divide the samples into $T=\log\frac{2k}{\delta}$ batches, apply \prettyref{algo:known} to each batch of $n/T$ samples, and take $\tilde m_r$ to be the median of these estimates. Then Hoeffding's inequality and the union bound imply that, with probability $1-\delta$,
    \begin{equation}
        \label{eq:Hoeffding-example}
        |\tilde m_r-m_r(\nu)| \le \sqrt{\frac{\log(2k/\delta)}{n}}(c\sqrt{r})^r,\quad \forall~r=1,\ldots,2k-1,
    \end{equation}
    and the desired \prettyref{eq:known-variance-W1} follows. 
\end{proof}


To conclude this subsection, we discuss the connection to the Generalized Method of Moments (GMM). Instead of solving the moment equations, GMM aims to minimize the difference between estimated and fitted moments:
\begin{equation}
    \label{eq:GMM-opt}
    Q(\theta)=(\hat m- m(\theta))^\top W(\hat m - m(\theta)),
\end{equation}
where $\hat m$ is the estimated moment, $\theta$ is the model parameter, and $W$ is a positive semidefinite weighting matrix. The minimizer of $Q(\theta)$ serves as the GMM estimate for the unknown model parameter $\theta_0$. In general the objective function $Q$ is non-convex in $\theta$, notably under the Gaussian mixture model with $\theta$ corresponding to the unknown means and weights, which is hard to optimize. Note that \prettyref{eq:project} with the Euclidean norm is \emph{equivalent} to GMM with the identity weighting matrix. Therefore \prettyref{algo:known} is an exact solver for GMM in the Gaussian location mixture model. 

In theory, the optimal weighting matrix $W^*$ that minimizes the asymptotic variance is the inverse of $\lim_{n\diverge}\cov[\sqrt{n}(\hat m- m(\theta_0))]$, which depends the unknown model parameters $\theta_0$. Thus, a popular approach is a two-step estimator \cite{Hall2005}:  
\begin{enumerate}
    \item a suboptimal weighting matrix, \eg, identify matrix, is used in the GMM to obtain a consistent estimate of $\theta_0$ and hence a consistent estimate $\hat W$ for $W^*$;
    \item $\theta_0$ is re-estimated using the weighting matrix $\hat W$. 
\end{enumerate}
The above two-step approach can be similarly implemented in the denoised method of moments.

\subsection{Unknown variance}
\label{sec:unknown}
When the variance parameter $\sigma^2$ is unknown, unbiased estimator for the moments of the mixing distribution no longer exists (see \prettyref{lmm:unbiased-not-exist}). It is not difficult to consistently estimate the variance,\footnote{
    For instance, the simple estimator $\hat\sigma=\frac{\max_i X_i}{\sqrt{2\log n}}$ satisfies $|\sigma-\hat \sigma|=O_P(\log n)^{-\frac{1}{2}}$.
    }
then plug into the DMM estimator in \prettyref{sec:known} to obtain a consistent estimate of the mixing distribution $\nu$; however, the convergence rate is far from optimal. In fact, to achieve the optimal rate in \prettyref{thm:main-W1}, 
 it is crucial to simultaneously estimate both the means and the variance parameters. To this end, again we take a moment-based approach. The following result provides a guarantee for any joint estimate of both the mixing distribution and the variance parameter in terms of the moments accuracy.
\begin{prop}
    \label{prop:accuracy2}
    Let
    \[
        \pi=\nu*N(0,\sigma^2),\quad \hat\pi=\hat\nu*N(0,\hat\sigma^2),
    \]
    where $\nu,\hat\nu$ are $k$-atomic distributions supported on $[-M,M]$, and $\sigma,\hat\sigma$ are bounded by a constant. If $|m_r(\pi)-m_r(\hat\pi)|\le \epsilon$ for $r=1,\dots,2k$, then
    \[
        |\sigma^2-\hat\sigma^2|\le O(M^2\epsilon^{\frac{1}{k}}),\quad W_1(\nu,\hat\nu)\le O(Mk^{1.5}\epsilon^{\frac{1}{2k}}).
    \]
\end{prop}
To apply \prettyref{prop:accuracy2}, we can solve the method of moments equations, namely, find a $k$-atomic distribution $\hat\nu$ and $\hat\sigma^2$ such that
\begin{equation}
    \label{eq:MM-sigma}
    \Expect_n[X^r] = \Expect_{\hat\pi}[X^r], \qquad r=1,\ldots,2k
\end{equation}
where $\hat\pi=\hat\mu * N(0,\hat\sigma^2)$ is the fitted Gaussian mixture. Here both the number of equations and the number of variables are equal to $2k$. Suppose \prettyref{eq:MM-sigma} has a solution $(\hat\mu,\hat\sigma)$. Then applying \prettyref{prop:accuracy2}  with $\delta = O_k(\frac{1}{\sqrt{n}})$ achieves the rate $O_k(n^{-1/(4k)})$ in \prettyref{thm:main-W1}, which is minimax optimal (see \prettyref{sec:lb}). In sharp contrast to the case of known $\sigma$, where we have shown in \prettyref{sec:mmfail} that the vanilla method of moments equation can have no solution unless we denoise by projection to the moment space, here with one extra scale parameter $\sigma$, one can show that \prettyref{eq:MM-sigma} has a solution with probability one!\footnote{
    It is possible that the equation \prettyref{eq:MM-sigma} has no solution, for instance, when $k=2, n=7$ and the empirical distribution is $\pi_7=\frac{1}{7}\delta_{-\sqrt{7}}+\frac{1}{7}\delta_{\sqrt{7}}+\frac{5}{7}\delta_0$. The first four empirical moments are ${\bf m}_4(\pi_7)=(0,2,0,14)$, which cannot be realized by any two-component Gaussian mixture \prettyref{eq:model}. Indeed, suppose $\hat \pi=w_1N(\mu_1,\sigma^2)+(1-w_1)N(\mu_2,\sigma^2)$ is a solution to \prettyref{eq:MM-sigma}. Eliminating variables leads to the contradiction that $2\mu_1^4+2=0$. Assuringly, as we will show later in \prettyref{lmm:U-exist}, such cases occur with probability zero.
} 
Furthermore, an efficient method of finding \emph{a} solution to \prettyref{eq:MM-sigma} is due to Lindsay \cite{Lindsay1989} and summarized in \prettyref{algo:Lindsay}. Here, the sample moments can be computed in $O(kn)$ time, and the smallest non-negative root of the polynomial of degree $k(k+1)$ can be found in $O(k^2)$ time using Newton's method (see \cite{Atkinson2008}). So overall Lindsay's estimator can be evaluated in $O(kn)$ time.
\begin{algorithm}[ht]
    \caption{Lindsay's estimator for normal mixtures with an unknown common variance}
    \label{algo:Lindsay}
    \begin{algorithmic}[1]
        \REQUIRE $ n $ samples $ X_1,\dots,X_n $.
        \ENSURE estimated mixing distribution $\hat\nu$, and estimated variance $\hat\sigma^2$.
        \FOR{$ r=1 $ \TO $ 2k $}
        \STATE{$ \hat{\gamma}_r=\frac{1}{n}\sum_i X_i^r $}
        \STATE{$ \hat{m}_r(\sigma)=r!\sum_{i=0}^{\Floor{r/2}}\frac{(-1/2)^i}{i!(r-2i)!}\hat{\gamma}_{r-2i}\sigma^{2i} $}
        \ENDFOR
        \STATE{Let $ \hat{d}_k(\sigma) $ be the determinant of the matrix $\{\hat{m}_{i+j}(\sigma)\}_{i,j=0}^k$.}\label{line:hat-dk}
        \STATE{Let $ \hat{\sigma} $ be the smallest positive root of $\hat{d}_k(\sigma)=0$.}\label{line:hat-sigma}
        \FOR{$ r=1 $ \TO $ 2k $}
        \STATE{$ \hat m_r=\hat{m}_r(\hat{\sigma}) $} \label{line:hat-mr}
        \ENDFOR
        \STATE{Let $\hat\nu$ be the outcome of the Gauss quadrature (\prettyref{algo:quadrature}) with input $\hat m_1,\dots,\hat m_{2k-1}$}
        \STATE{Report $\hat\nu$ and $\hat\sigma^2$.}
    \end{algorithmic}
\end{algorithm}

In \cite{Lindsay1989} the consistency of this estimator was proved under the extra condition that $\hat \sigma$ (which is a random variable) as a root of $d_k$ has multiplicity one.
It is unclear whether this condition is guaranteed to hold.
We will show that, unconditionally, Lindsay's estimator is not only consistent, but in fact achieves the minimax optimal rate \prettyref{eq:unknown-variance-W1} and \prettyref{eq:unknown-variance-sigma} previously announced in \prettyref{thm:main-W1}. We start by proving that Lindsay's algorithm produces an estimator $\hat\sigma$ so that the corresponding the moment estimates lie in the moment space with probability one. In this sense, although no explicit projection is involved, the noisy estimates are \emph{implicitly} denoised.

We first describe the intuition of the choice of $\hat\sigma$ in Lindsay's algorithm, \ie, line \ref{line:hat-sigma} of \prettyref{algo:Lindsay}. Let $X\sim \nu*N(0,\sigma^2)$. For any $\sigma' \leq \sigma$, we have 
\[
    \Expect[\gamma_j(X,\sigma')]=m_j(\nu*N(0,\sigma^2-\sigma'^2)).
\]
Let $d_k(\sigma')$ denote the determinant of the moment matrix $\{\Expect[\gamma_{i+j}(X,\sigma')]\}_{i,j=0}^k$, which is an even polynomial in $\sigma'$ of degree $k(k+1)$. According to \prettyref{thm:supp-detM}, $d_k(\sigma')>0$ when $0\le \sigma'<\sigma$ and becomes zero at $\sigma'=\sigma$, and thus $\sigma$ is characterized by the smallest positive zero of $d_k$. In lines \ref{line:hat-dk} -- \ref{line:hat-sigma}, $d_k$ is estimated by $\hat d_k$ using the empirical moments, and $\sigma$ is estimated by the smallest positive zero of $\hat d_k$. We first note that $\hat d_k$ indeed has a positive zero:
\begin{lemma}
    \label{lmm:sigma-exist}
    Assume $ n>k $ and the mixture distribution has a density. Then, almost surely, $\hat d_k$ has a positive root within $(0,s]$, where $s^2 \triangleq \frac{1}{n}\sum_{i=1}^n(X_i-\Expect_n[X])^2$ denotes the sample variance. 
\end{lemma}

The next result shows that, with the above choice of $\hat\sigma$, the moment estimates $\hat{m}_j=\Expect_n[\gamma_j(X,\hat\sigma)]$ for $j=1,\dots,2k$ given in line \ref{line:hat-mr} are implicitly denoised and lie in the moment space with probability one. Thus \prettyref{eq:MM-sigma} has a solution, and the estimated mixing distribution $\hat\nu$ can be found by the Gauss quadrature. This result was previously shown in \cite{Lindsay1989} assuming that $\hat \sigma$ is of multiplicity one. 
In contrast, \prettyref{lmm:U-exist} only requires that $n\ge 2k-1$ and the mixture distribution has a density.
\begin{lemma}
    \label{lmm:U-exist}
    Assume $ n\ge 2k-1 $ and the mixture distribution has a density. Then, almost surely, there exists a $k$-atomic distribution $\hat\nu$ such that $m_j(\hat\nu)=\hat{m}_j$ for $j\le 2k$, where $\hat{m}_j$ is from \prettyref{algo:Lindsay}.
\end{lemma}

With the above analysis, we now prove the statistical guarantee \prettyref{eq:unknown-variance-W1} and \prettyref{eq:unknown-variance-sigma} for Lindsay's algorithm announced in \prettyref{thm:main-W1}:
\begin{proof}
    It suffices to consider $M=1$. Let $\hat\pi=\hat\nu*N(0,\hat\sigma^2)$ and $\pi=\nu*N(0,\sigma^2)$ denote the estimated mixture distribution and the ground truth, respectively. Let $\hat m_r=\Expect_n[X^r]$ and $m_r=m_r(\pi)$. The variance of $\hat m_r$ is upper bounded by
    \[
        \var[\hat m_r]=\frac{1}{n}\var[X_1^r]\le \frac{1}{n}\Expect[X^{2r}]
        \le \frac{(\sqrt{cr})^{2r}}{n},
    \]
    for some absolute constant $c$. Using Chebyshev inequality, for each $r=1,\dots,2k$, with probability $1-\frac{1}{8k}$, we have, 
    \begin{equation}
        \label{eq:tildemr-rate2}    
        |\hat m_r-m_r|\le (\sqrt{cr})^r\sqrt{k/n}.
    \end{equation}
    By the union bound, with probability 3/4, the above holds holds simultaneously for every $r=1,\dots,2k$. It follows from \prettyref{lmm:sigma-exist} and \ref{lmm:U-exist} that \prettyref{eq:MM-sigma} holds with probability one. Therefore,
    \[
        |m_r(\hat\pi)-m_r(\pi)|\le (\sqrt{cr})^r\sqrt{k/n},\quad r=1,\dots,2k.
    \]
    for some absolute constant $c$. In the following, the error of variance estimate is denoted by $\tau^2=|\sigma^2-\hat\sigma^2|$.

    \begin{itemize}
        \item If $\sigma\le \hat\sigma$, let $\nu'=\hat\nu*N(0,\tau^2)$. Using $\Expect_\pi[\gamma_r(X,\sigma)]=m_r(\nu)$ and $\Expect_{\hat\pi}[\gamma_r(X,\sigma)]=m_r(\nu')$, where $\gamma_r$ is the Hermite polynomial \prettyref{eq:Hermite}, we obtain that (see \prettyref{lmm:Hermite-moments})
        \begin{equation}
            \label{eq:mr-deconv-example}
            |m_r(\nu')-m_r(\nu)|\le (\sqrt{c'k})^{2k}\sqrt{k/n},\quad r=1,\dots,2k,
        \end{equation}
        for an absolute constant $c'$. Applying \prettyref{prop:accuracy2} yields that,
        \[
            |\sigma^2-\hat\sigma^2|\le O(kn^{-\frac{1}{2k}}),\quad W_1(\nu,\hat\nu)\le O(k^2n^{-\frac{1}{4k}}).
        \]
        \item If $\sigma\ge \hat\sigma$, let $\nu'=\nu*N(0,\tau^2)$. Similar to \prettyref{eq:mr-deconv-example}, we have 
        \[
            |m_r(\hat\nu)-m_r(\nu')|\le (\sqrt{c'k})^{2k}\sqrt{k/n}\triangleq \epsilon,\quad r=1,\dots,2k.
        \]
        To apply \prettyref{prop:accuracy2}, we also need to ensure that $\hat\nu$ has a bounded support, which is not obvious. To circumvent this issue, we apply a truncation argument thanks to the following tail probability bound for $\hat\nu$ (see \prettyref{lmm:tail-hatU}):
        \begin{equation}
            \label{eq:hatnu-tail}
            \Prob[|\hat U|\ge \sqrt{c_0k}]\le \epsilon(\sqrt{c_1k}/t)^{2k},\quad \hat U\sim \hat \nu,
        \end{equation}
        for absolute constants $c$ and $c'$. To this end, consider $\tilde U=\hat U\indc{|\hat U|\le \sqrt{c_0k}}\sim\tilde\nu$. Note that $\tilde U$ is $k$-atomic supported on $[-\sqrt{c_0k},\sqrt{c_0k}]$, we have $W_1(\nu,\hat\nu)\le \epsilon e^{O(k)}$ and $|m_r(\tilde\nu)-m_r(\hat\nu)|\le k\epsilon(c_1k)^k$ for $r=1,\dots,2k$. Using the triangle inequality yields that 
        \[
            |m_r(\tilde \nu)-m_r(\nu')|\le \epsilon+k\epsilon(c_1k)^k.
        \]
        Now we apply \prettyref{prop:accuracy2} with $\tilde\nu$ and $\nu*N(0,\tau^2)$ where both $\tilde \nu$ and $\nu$ are $k$-atomic supported on $[-\sqrt{c_0k},\sqrt{c_0k}]$. In the case $\tilde\nu$ is discrete, the dependence on $k$ in \prettyref{prop:accuracy2} can be improved (by improving \prettyref{eq:mr-deconv} in the proof) and we obtain that
        \[
            |\sigma^2-\hat\sigma^2|\le O(kn^{-\frac{1}{2k}}),\quad W_1(\nu,\tilde\nu)\le O(k^2n^{-\frac{1}{4k}}).
        \]
        Using $k\le O(\frac{\log n}{\log\log n})$, we also obtain $W_1(\nu,\hat\nu)\le O(k^2n^{-\frac{1}{2k}})$ by the triangle inequality.
    \end{itemize}
    To obtain a confidence $1-\delta$ in \prettyref{eq:unknown-variance-W1} and \prettyref{eq:unknown-variance-sigma}, we can replace the empirical moments $\hat m_r$ by the median of $T=\log\frac{2k}{\delta}$ independent estimates similar to \prettyref{eq:Hoeffding-example}. 
\end{proof}

\subsection{Adaptive rates}
\label{sec:adaptive}
In sections \ref{sec:known} and \ref{sec:unknown}, we proved the statistical guarantees of our estimators under the worst-case scenario where the means can be arbitrarily close. Under separation conditions on the means (see \prettyref{def:sep}), our estimators automatically achieve a strictly better accuracy than the one claimed in \prettyref{thm:main-W1}. The goal in this subsection is to show those adaptive results. The key is the following adaptive version of the moment comparison theorems (cf.~Propositions \ref{prop:stable1} and \ref{prop:stable2}):
\begin{prop}
    \label{prop:stable1-separation}
    Suppose both $\nu$ and $\nu'$ are supported on a set of $\ell$ atoms in $[-1,1]$, and each atom is at least $\gamma$ away from all but at most $\ell'$ other atoms. Let $\delta=\max_{i\in [\ell-1]}|m_i(\nu)-m_i(\nu')|$. Then,
    \[
        W_1(\nu,\nu')\le \ell \pth{\frac{\ell 4^{\ell-1}\delta}{\gamma^{\ell-\ell'-1}}}^{\frac{1}{\ell'}}.
    \]
\end{prop}
\begin{prop}
    \label{prop:stable2-separation}
    Suppose $\nu$ is supported on $k$ atoms in $[-1,1]$ and any $t\in\reals$ is at least $\gamma$ away from all but $k'$ atoms. Let $\delta=\max_{i\in [2k]}|m_i(\nu)-m_i(\nu')|$. Then,
    \[
        W_1(\nu,\nu')\le 8k\pth{\frac{k 4^{2k} \delta}{\gamma^{2(k-k')}}}^{\frac{1}{2k'}}.
    \]
\end{prop}

The adaptive result \prettyref{eq:known-variance-W1-adaptive} in the known variance parameter case is obtained using \prettyref{prop:stable1-separation} in place of \prettyref{prop:stable1}. To deal with unknown variance parameter case, using \prettyref{prop:stable2-separation}, we first show the following adaptive version of \prettyref{prop:accuracy2}:
\begin{prop}
    \label{prop:accuracy2-separation}
    Under the conditions of \prettyref{prop:accuracy2}, if both Gaussian mixtures both have $k_0$ $\gamma$-separated clusters in the sense of \prettyref{def:sep}, then, 
    \[
        \sqrt{|\sigma^2-\hat\sigma^2|},~ W_1(\nu,\hat\nu)\le O_k\pth{\pth{\frac{\epsilon}{\gamma^{2(k_0-1)}}}^{\frac{1}{2(k-k_0+1)}}}.
    \]
\end{prop}

Using these propositions, we now prove the adaptive rate of the denoised method of moments previously announced in \prettyref{thm:main-W1-adaptive}:
\begin{proof}[Proof of \prettyref{thm:main-W1-adaptive}]
    By scaling it suffices to consider $M=1$. Recall that the Gaussian mixture is assumed to have $k_0$ $(\gamma,\omega)$-separated clusters in the sense of \prettyref{def:sep}, that is, there exists a partition $S_1,\dots,S_{k_0}$ of $[k]$ such that $|\mu_i - \mu_{i'}| \ge \gamma$ for any $i\in S_\ell$ and $i'\in S_{\ell'}$ such that $\ell\ne \ell'$, and $\sum_{i\in S_\ell}w_i\ge \omega$ for each $\ell$.

    Let $\hat\nu$ be the estimated mixing distribution which satisfies $W_1(\nu,\hat \nu)\le \epsilon$ by \prettyref{thm:main-W1}. Since $\gamma\omega\ge C\epsilon$ by assumption, for each $S_\ell$, there exists $i\in S_\ell$ such that $\mu_i$ is within distance $c\gamma$, where $c=1/C$, to some atom of $\hat \nu$. Therefore, the estimated mixing distribution $\hat \nu$ has $k_0$ $(1-2c)\gamma$-separated clusters. Denote the union of the support sets of $\nu$ and $\hat \nu$ by $\calS$.
    \begin{itemize}
        \item When $\sigma$ is known, each atom in $\calS$ is $\Omega(\gamma)$ away from at least $2(k_0-1)$ other atoms. Then \prettyref{eq:known-variance-W1-adaptive} follows from \prettyref{prop:stable1-separation} with $\ell=2k$ and $\ell'=(2k-1)-2(k_0-1)$. 
        \item When $\sigma$ is unknown, \prettyref{eq:unknown-variance-W1-adaptive} follows from a similar proof of \prettyref{eq:unknown-variance-W1} and \prettyref{eq:unknown-variance-sigma} with \prettyref{prop:accuracy2} replaced by \prettyref{prop:accuracy2-separation}. \qedhere
    \end{itemize}
\end{proof}

Finally, we note that if one only assumes the separation condition but not the lower bound on the weights, we can obtain an intermediate result that is stronger than \prettyref{eq:known-variance-W1} but weaker than \prettyref{eq:known-variance-W1-adaptive}. 
\begin{theorem}
    \label{thm:separation-W1}
    Under the conditions of \prettyref{thm:main-W1}, suppose $\sigma$ is known and the Gaussian mixture has $k_0$ $\gamma$-separated clusters. Then, with probability at least $1-\delta$,
    \begin{equation}
        \label{eq:known-variance-W1-separation}
        W_1(\nu,\hat \nu)\le O_k\pth{M\gamma^{-\frac{k_0-1}{2k-k_0}}\pth{\frac{n}{\log(k/\delta)}}^{-\frac{1}{4k-2k_0}}}.
    \end{equation}
\end{theorem}

\subsection{Unbounded means}
\label{sec:unbounded}
In the previous subsections, we assume that the means lie in a bounded interval. In the unbounded case, it is in fact impossible to estimate the mixing distribution under the Wasserstein distance\footnote{
    Let $\pi_\epsilon=\frac{1+\epsilon}{2}\delta_{0}+\frac{1-\epsilon}{2}\delta_{M}$. Then $W_1(\pi_0,\pi_\epsilon)=M\epsilon$, but $D(\pi_0\|\pi_\epsilon)\le O(\epsilon^2)$ independent of $M$. Choosing $\epsilon = o(1/\sqrt{n})$ and $M \gg 1/\epsilon$ leads to arbitrarily large estimation error.
}.
Nevertheless, provided that the weights are bounded away from zero, it is possible to estimate the support set of the mixing distribution with respect to the Hausdorff distance (cf.~\prettyref{eq:hausdorff}). This is the goal of this subsection.

In the unbounded case, blindly applying the previous moment-based methods does not work, because the estimated moments suffer from large variance due to the wide range of values of the means (cf.~\prettyref{lmm:var-tildem}). To resolve this issue, we shall apply the ``divide and conquer'' strategy as follows: partition the real line into intervals, estimate means in each interval separately, and report the union as the estimate of the set of centers. The complete procedure is given in \prettyref{algo:unbounded}.
\begin{algorithm}[ht]
    \caption{Estimate means of a Gaussian mixture model in the unbounded case.}
    \label{algo:unbounded}
    \begin{algorithmic}[1]
        \REQUIRE $n$ samples $X_1,\dots,X_{n}$, variance parameter $\sigma^2$ (optional), cluster parameter $L$, and weights threshold $\tau$, test sample size $n'$.
        \ENSURE a set of estimated means $\hat S$
        \STATE Merge overlapping intervals $[X_i\pm L]$ for $i\le n'$ into disjoint ones $I_1,\dots,I_s$.\label{line:cluster-I}
        \FOR{$j=1$ \TO $s$}
        \STATE Let $c_j,\ell_j$ be such that $I_j=[c_j\pm \ell_j]$.
        \STATE Let $C_j=\{X_i-c_j:X_i\in I_j,i>n'\}$.\label{line:cluster-C}
        \IF{$\sigma^2$ is specified}
        \STATE Let ${(\hat w,\hat \mu)}$ be the outcome of \prettyref{algo:known} with input $C_j$, $\sigma^2$, and $[-\ell_j,\ell_j]$.
        \ELSE 
        \STATE Let ${(\hat w,\hat \mu)}$ be the outcome of \prettyref{algo:Lindsay} with input $C_j$.
        \ENDIF
        \STATE Let $\hat S_j=\{\hat x_i+c_j: \hat w_i\ge \tau\}$.
        \ENDFOR
        \STATE Report $\hat S=\cup_j\hat S_j$.
    \end{algorithmic}
\end{algorithm}

The first step is to apply a clustering method that partitions the samples into a small number of groups. There are many clustering algorithms in practice such as the popular Lloyd's $k$-means clustering \cite{Lloyd1982}. In lines \ref{line:cluster-I} -- \ref{line:cluster-C}, we present a conservative yet simple clustering with the following guarantees (see \prettyref{lmm:cluster}):
\begin{itemize}
    \item each interval is of length at most $O(kL)$;
    \item a sample $X_i=U_i+\sigma Z_i$ is always in the same interval as the latent variable $U_i$. 
\end{itemize}
In the present clustering method, each cluster $C_j$ only contains samples that are not used line \ref{line:cluster-I} so that the intervals are independent of each $C_j$. This is a commonly used \emph{sample splitting} technique in statistics to simplify the analysis. Note that only a small number of samples are needed to determine the intervals (see \prettyref{thm:unbounded}). In the second step, we estimate means in each $I_j$ using samples $C_j$ and report the union of all means.

The statistical guarantee of \prettyref{algo:unbounded} is analyzed in \prettyref{thm:unbounded}. Note that \prettyref{thm:unbounded} holds in the worst-case, and can be improved in many situations: the number of samples in each $C_j$ increases proportionally to the total weights; the adaptive rate in \prettyref{thm:main-W1-adaptive} is applicable when separation is present within one interval; we can postulate fewer components in one interval based on information from other intervals.

\begin{theorem}
    \label{thm:unbounded}
    Assume in the Gaussian mixture \prettyref{eq:model} $w_i \geq \epsilon$, $\sigma$ is bounded. Let $S = \supp(\nu)$ be the set of means of the Gaussian mixture, and $\hat S$ be the output of \prettyref{algo:unbounded} with $L=\Theta(\sqrt{\log n})$ and $\tau=\epsilon/(2k)$. If $n\ge 2n'\ge \Omega(\frac{\log (k/\delta)}{\epsilon})$, then, with probability $1-\delta-n^{-\Omega(1)}$, we have 
    \[
        d_H(\hat S, S)\le  
        \begin{cases}
            ~O\pth{L k^{3.5} (\frac{\epsilon n}{\log(1/\delta)})^{-\frac{1}{4k-2}}/\epsilon},& \sigma\text{ is known},\\
            ~O\pth{L k^4 (\frac{\epsilon n}{\log(1/\delta)})^{-\frac{1}{4k}}/\epsilon},& \sigma\text{ is unknown},
        \end{cases}
    \]
    where $d_H$ denotes the Hausdorff distance (see \prettyref{eq:hausdorff}).
\end{theorem}

\section{Lower bounds}
\label{sec:lb}
This section introduces minimax lower bounds for estimating Gaussian location mixture models which certify the optimality of our estimators. We will apply Le Cam's two-point method, namely, find two Gaussian mixtures that are statistically close. Then any estimator suffers a loss at least proportional to the distance between these two mixing distributions.

To bound the statistical distance between two mixture models, one commonly used technique is \emph{moment matching}, \ie, $\nu*N(0,1)$ and $\nu*N(0,1)$ are statistically close if ${\bf m}_\ell(\nu)={\bf m}_\ell(\nu')$ for some large $\ell$. This is demonstrated in \prettyref{fig:tv-latent-mixture}, and is made precise in \prettyref{lmm:d-matching}. Statistical closeness via moment matching has been established, for instance, by orthogonal expansion \cite{WV2010,CL11}, by Taylor expansion \cite{HP15,WY14}, and by the best polynomial approximation \cite{WY15}. Similar results to this lemma were previously obtained in \cite{WV2010,CL11,HP15}. 
\begin{figure}[ht]
    \centering
    \subfigure[Mixing distributions]
    {\label{fig:tv-latent}\includegraphics[width=.4\linewidth]{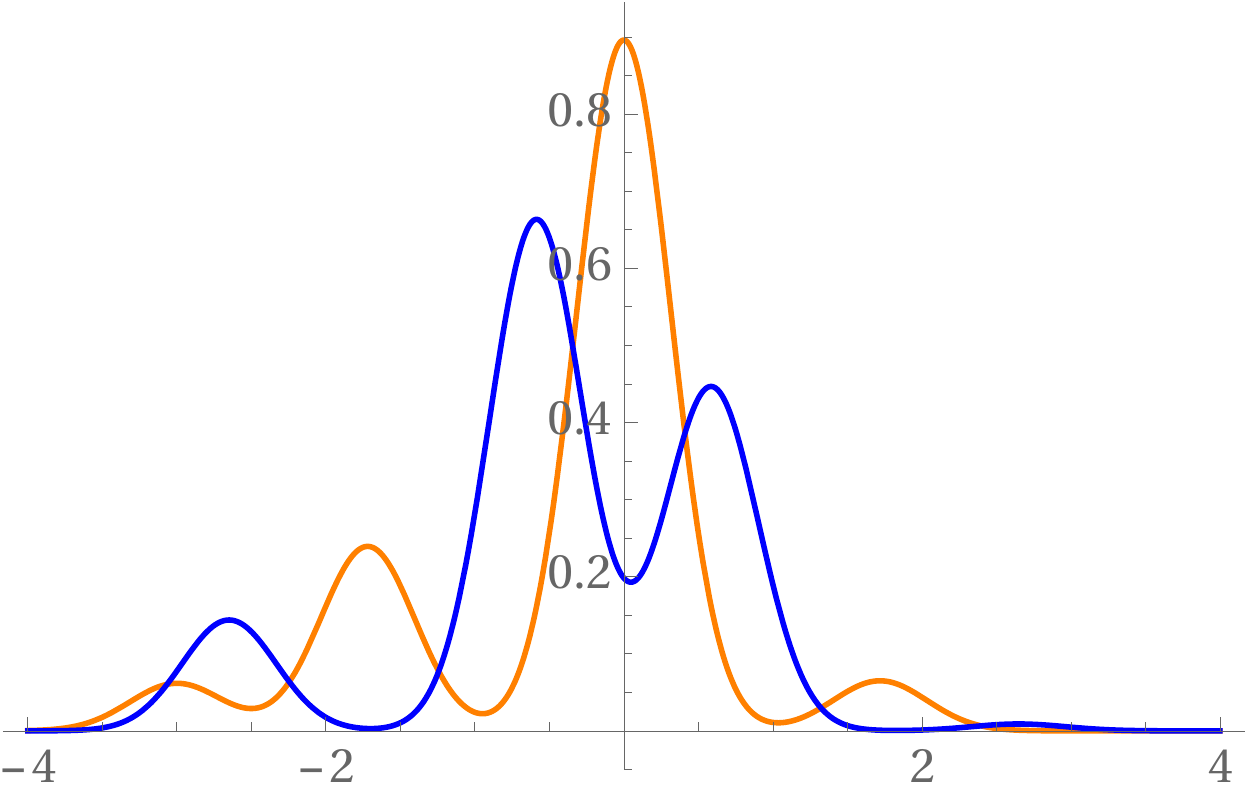}}
    \qquad
    \subfigure[Mixture distributions]
    {\label{fig:tv-mixture}\includegraphics[width=.4\linewidth]{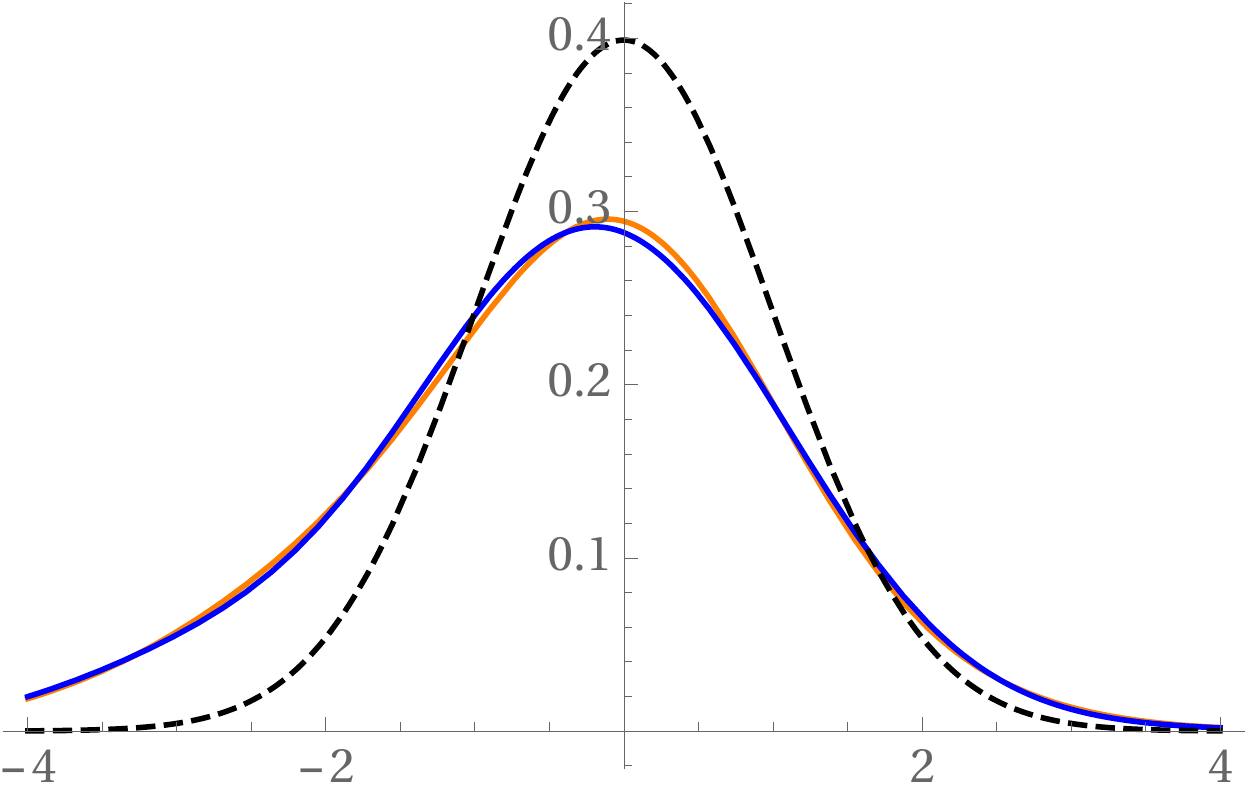}}
    \caption{Statistical closeness via moment matching. In \subref{fig:tv-latent}, two different mixing distributions coincide on their first six moments; in \subref{fig:tv-mixture}, the mixing distributions are convolved with the standard normal distribution (the black dashed line), and the Gaussian mixtures are statistically close. 
    }
    \label{fig:tv-latent-mixture}
\end{figure}
\begin{lemma}
    \label{lmm:d-matching}
    Suppose $\nu$ and $\nu'$ are centered distributions such that ${\bf m}_\ell(\nu)={\bf m}_\ell(\nu')$. 
    \begin{itemize}
        \item If $\nu$ and $\nu'$ are $\epsilon$-subgaussian for $\epsilon<1$, then
        \begin{equation}
            \label{eq:chi2-subg}
            \chi^2(\nu*N(0,1)\|\nu'*N(0,1))\le O\pth{\frac{1}{\sqrt{\ell}}\frac{\epsilon^{2\ell+2}}{1-\epsilon^2}}.
        \end{equation}
        \item If $\nu$ and $\nu'$ are supported on $[-\epsilon,\epsilon]$ for $\epsilon<1$, then
        \begin{equation}
            \label{eq:chi2-bdd}
            \chi^2(\nu*N(0,1)\|\nu'*N(0,1))\le O\pth{\pth{\frac{e\epsilon^2}{\ell+1}}^{\ell+1}}.
        \end{equation}
    \end{itemize}
\end{lemma}

\begin{remark}[Tightness of \prettyref{lmm:d-matching}]
    When $\ell$ is odd, there exists a pair of $\epsilon$-subgaussian distributions $\nu$ and $\nu'$ such that ${\bf m}_\ell(\nu)={\bf m}_\ell(\nu')$, while $\chi^2(\nu*N(0,1)\|\nu'*N(0,1))\ge \Omega_\ell(\epsilon^{2\ell+2})$. Such a pair can be constructed using Gauss quadrature introduced in \prettyref{sec:moment}. To this end, let $\ell=2k-1$ and we set $\nu=N(0,\epsilon^2)$ and $\tilde g_k$ to be its $k$-point Gauss quadrature. Then ${\bf m}_{2k-1}(\nu)={\bf m}_{2k-1}(\tilde g_k)$, and $\tilde g_k$ is also $\epsilon$-subgaussian (see \prettyref{lmm:quadrature-moments-error}). It is shown in \cite[Eq.~(54)]{WV2010} that
    \[
        \chi^2(\tilde g_k*N(0,1)\|\nu*N(0,1))=\sum_{j \ge 2k}\frac{1}{j!}\pth{\frac{\epsilon^2}{1+\epsilon^2}}^j|\Expect_{g_k}[H_j]|^2,
    \]
    where $g_k$ is the $k$-point Gauss quadrature of the standard normal distribution, and $H_k$ is the degree-$k$ Hermite polynomial defined in \prettyref{eq:Hermite1}. Since $\Expect_{g_k}[H_{2k}]=-k!$ (see \prettyref{lmm:quadrature-moments-Hermite}), 
	for any $\epsilon<1$, we have 
    \[
        \chi^2(\tilde g_k*N(0,1)\|\nu*N(0,1))\ge \frac{(k!)^2}{(2k)!}\pth{\frac{\epsilon^2}{1+\epsilon^2}}^{2k}\ge (\Omega(\epsilon))^{4k},
    \]
\end{remark}

In view of \prettyref{lmm:d-matching}, the best lower bound follows from two different mixing distributions $\nu$ and $\nu'$ such that ${\bf m}_\ell(\nu)={\bf m}_\ell(\nu')$ with the largest degree $\ell$, which is $2k-2$ when both distributions are $k$-atomic and $2k-1$ when one of them is $k$-atomic (see \prettyref{lmm:identify} and the following \prettyref{rmk:stable}).
Next we provide the precise minimax lower bounds for the case of known and unknown variance separately.

\paragraph{Known variance. } We shall assume $\sigma=1$. First, we define the space of all $k$ Gaussian location mixtures as 
\[
    \calP_k=\{\nu*N(0,1): \nu \text{ is $k$-atomic supported on }[-1,1]\},
\]
and we consider the worst-case risk over all mixture models in $\calP_k$. From the identifiability of discrete distributions in \prettyref{lmm:identify}, two different $k$-atomic distributions can match up to $2k-2$ moments. Therefore, using \prettyref{lmm:d-matching}, the best minimax lower bound using Le Cam's method is obtained from the optimal pair of distributions for the following:
\begin{equation}
    \label{eq:optimal-lb}
    \begin{aligned}
        \max & ~W_1(\nu,\nu')\\
        \mathrm{s.t.}&~{\bf m}_{2k-2}(\nu)={\bf m}_{2k-2}(\nu'),\\
        &~ \nu,\nu' \text{ are $k$-atomic on }[-\epsilon,\epsilon].
    \end{aligned}
\end{equation}
The value of the above optimization problem is $\Omega(\epsilon/k)$ (see \prettyref{lmm:max-W1}). Using $\epsilon=\sqrt{k}n^{-\frac{1}{4k-2}}$, we obtain the following minimax lower bound:
\begin{prop}
    \label{prop:lb-known-sigma}
    \[
        \inf_{\hat \nu}\sup_{P\in\calP_k}\Expect_P W_1(\nu,\hat\nu) \ge \Omega\pth{\frac{1}{\sqrt{k}}n^{-\frac{1}{4k-2}}}
    \]
    where $\hat\nu$ is an estimator measurable with respect to $X_1,\ldots,X_n\iiddistr P=\nu*N(0,1)$.
\end{prop}

\begin{remark}
    \label{rmk:lb-adaptive-known}
    The above lower bound argument can be easily extended to prove the optimality of \prettyref{eq:known-variance-W1-adaptive} in \prettyref{thm:main-W1-adaptive}, where the mixture satisfies further separation conditions in the sense of \prettyref{def:sep}. In this case, the main difficulty is to estimate parameters in the biggest cluster. When there are $k_0$ $\gamma$-separated clusters, the biggest cluster is of order at most $k'=k-k_0+1$. Similar to \prettyref{eq:optimal-lb}, let $\tilde\nu$ and $\tilde\nu'$ be two $k'$-atomic distributions on $[-\epsilon,\epsilon]$. Consider the following mixing distributions
    \[
        \nu=\frac{k_0-1}{k_0}\nu_0+\frac{1}{k_0}\tilde\nu,\quad \nu'=\frac{k_0-1}{k_0}\nu_0+\frac{1}{k_0}\tilde\nu',
    \]
    where $\nu_0$ is the uniform distribution over $\{\pm 2\gamma,\pm 3\gamma,\ldots\}$ of cardinality $k_0-1$. Then both mixture models have $k_0$ $(\gamma,\frac{1}{k_0})$-separated clusters. Thus the minimax lower bound $\Omega(\frac{1}{\sqrt{k'}}n^{-\frac{1}{4k'-2}})$ analogously follows from Le Cam's method.
\end{remark}

By similar argument, when the order of the mixture model is $\Omega(\frac{\log n}{\log \log n})$, we obtain from \prettyref{eq:optimal-lb} a pair $\nu$ and $\nu'$ supported on $[-1,1]$ with identical first $\frac{\log n}{\log\log n}$ moments. This leads to the following lower bound which matches the upper bound in \prettyref{thm:large-k}.
\begin{prop}
\label{prop:large-k-lb}
If $k=\Omega(\frac{\log n}{\log\log n})$, then
\[
\inf_{\hat \nu}\sup_{P\in\calP_k}\Expect W_1(\nu,\hat\nu)\ge \Omega\pth{\frac{\log\log n}{\log n}}.
\]
\end{prop}

\paragraph{Unknown variance. } In this case the collection of mixture models is defined as 
\[
    \calP_k'=\{\nu*N(0,\sigma^2): \nu \text{ is $k$-atomic supported on }[-1,1],~\sigma\le 1\}.
\]
In \prettyref{lmm:d-matching}, mixing distributions are not restricted to be $k$-atomic but can be Gaussian location mixtures themselves, thanks to the infinite divisibility of the Gaussian distributions, \eg, $N(0,\epsilon^2)*N(0,0.5)=N(0,0.5+\epsilon^2)$. Let $g_k$ be the $k$-point Gauss quadrature of $N(0,\epsilon^2)$. Then $g_k$ has the same first $2k-1$ moments as $N(0,\epsilon^2)$, and $g_k*N(0,0.5)$ is a $k$-Gaussian mixture. Applying \prettyref{eq:chi2-subg} yields that
\[
    \chi^2(g_k*N(0,0.5)\|N(0,0.5+\epsilon^2))\le O(\epsilon^{4k}).
\]
Using $W_1(g_k,\delta_0)\ge \Omega(\epsilon/\sqrt{k})$ (see \prettyref{lmm:quadrature-l1}), and choosing $\epsilon=n^{-\frac{1}{4k}}$, we obtain the following minimax lower bound:
\begin{prop}
    \label{prop:lb-unknown-sigma}
    For $k\ge 2$,
    \begin{gather*}
        \inf_{\hat \nu}\sup_{P\in\calP_k}\Expect_P W_1(\nu,\hat\nu) \ge \Omega\pth{\frac{1}{\sqrt{k}}n^{-\frac{1}{4k}}},\\
        \inf_{\hat \nu}\sup_{P\in\calP_k}\Expect_P |\sigma^2-\hat\sigma^2| \ge \Omega\pth{n^{-\frac{1}{2k}}},
    \end{gather*}
    where the infimum is taken over estimators $\hat\nu, \hat\sigma^2$ measurable with respect to $X_1,\ldots,X_n\iiddistr P=\nu*N(0,\sigma^2)$.
\end{prop}

\section{Numerical experiments}
\label{sec:exp}


The algorithms of this paper are implemented in {\sf Python}.\footnote{
    The implementations are available at \url{https://github.com/Albuso0/mixture}.
}
In \prettyref{algo:known}, the explicit denoising via semidefinite programming uses {\sf CVXPY} \cite{cvxpy} and {\sf CVXOPT} \cite{cvxopt}, and the Gauss quadrature is calculated based on \cite{GW1969}. 
In this section, we compare the performance of our algorithms with the EM algorithm, also implemented in {\sf Python}, and the GMM algorithm using the popular package {\sf gmm} \cite{Chausse2010} implemented in {\sf R}. We omit the comparison with the vanilla method of moments which constantly fails to output a meaningful solution (see \prettyref{sec:mmfail}). In all figures presented in this section, we omit the running time of {\sf gmm}, which is on the order of hours as compared to seconds using our algorithms; the slowness of of {\sf gmm} is mainly due to the heuristic solver of the non-convex optimization \prettyref{eq:GMM-opt}.

We first clarify the parameters used in the experiments. EM and the iterative solver for \prettyref{eq:GMM-opt} in {\sf gmm} both require an initialization and a stop criterion. We use the best over five random initializations: the means are drawn independently from a uniform distribution, and the weights are from a Dirichlet distribution; then we pick the estimate that maximizes the likelihood and the minimal moment discrepancy \prettyref{eq:GMM-opt} in EM and GMM, respectively. The EM algorithm terminates when log-likelihood increases less than $10^{-3}$ or 5,000 iterations are reached; we use the default stop criterion in {\sf gmm} \cite{Chausse2010}.

\paragraph{Known variance. } We generated a random instance of Gaussian mixture model with five components and an unit variance. The means are drawn uniformly from $[-1,1]$; the weights are drawn from the Dirichlet distribution with parameters $(1,1,1,1,1)$, i.e., uniform over the probability simplex. It has the following parameters:
\begin{center}
    \begin{tabular}{c|c  c  c  c  c}  
      \toprule
      Weights&  0.123 &  0.552 & 0.010 & 0.080 & 0.235\\
      \hline
      Centers& -0.236 & -0.168 & -0.987 & 0.299  & 0.150\\
      \bottomrule
    \end{tabular}
\end{center}
We repeat the experiments 20 times and plot and the average and the standard deviation of the errors in the Wasserstein distance. We also plot the running time at each sample size. The results are shown in \prettyref{fig:DMM}.
\begin{figure}[h]
    \centering
    \includegraphics[width=\linewidth]{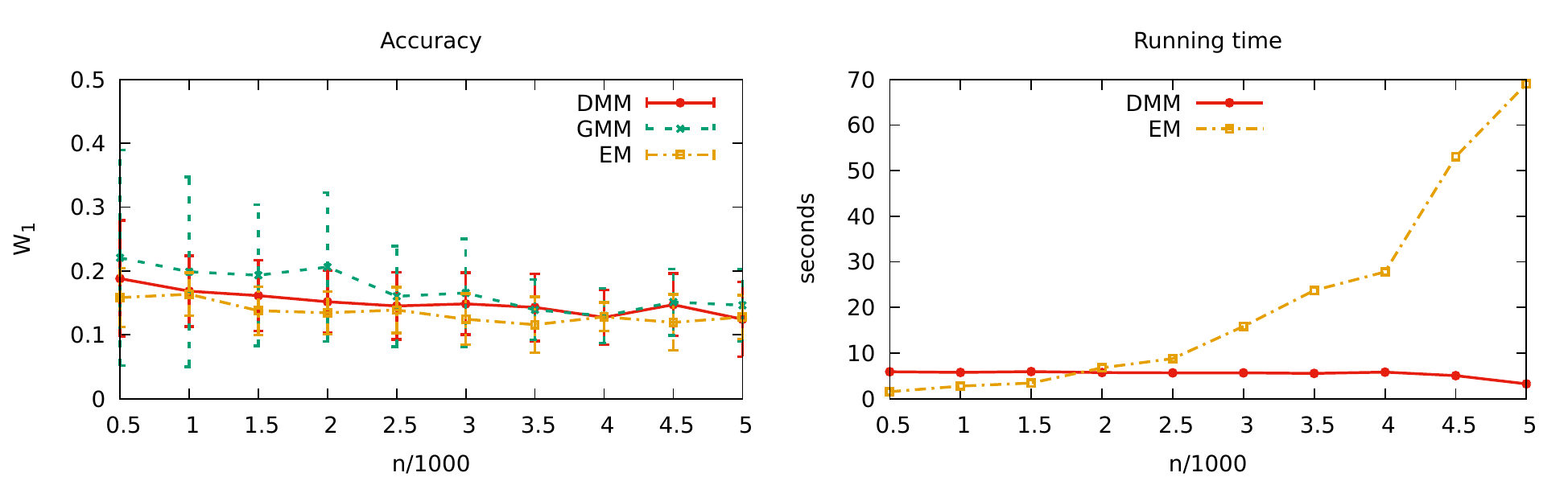}
    \caption{Comparison of different methods under a randomly generated five-component Gaussian mixture model.}
    \label{fig:DMM}
\end{figure}
These three algorithms have comparable accuracies, but EM is significantly slower than DMM: it is 15 times slower with 5,000 samples and increasingly slower as the sample size grows. This is because EM accesses all samples in each iteration, instead of first summarizing data into a few moments.

Furthermore, EM converges particularly slowly when components are poorly separated, since the likelihood function is very flat near its maximum \cite{RW1984,KX2003}. In this case, a loose stop criterion can terminate the algorithm prematurely, while a stringent one incurs substantially longer running time. 
To demonstrate this effect, in \prettyref{fig:DMM-k2} we consider the extreme case of a two-component Gaussian mixture with overlapping components, where the samples are drawn from $N(0,1)$. We run the EM algorithm that terminates when the log-likelihood increases less than $10^{-3}$ and $10^{-4}$, shown as EM and EM$^+$ in \prettyref{fig:DMM-k2}. 
\begin{figure}[h]
    \centering
    \includegraphics[width=\linewidth]{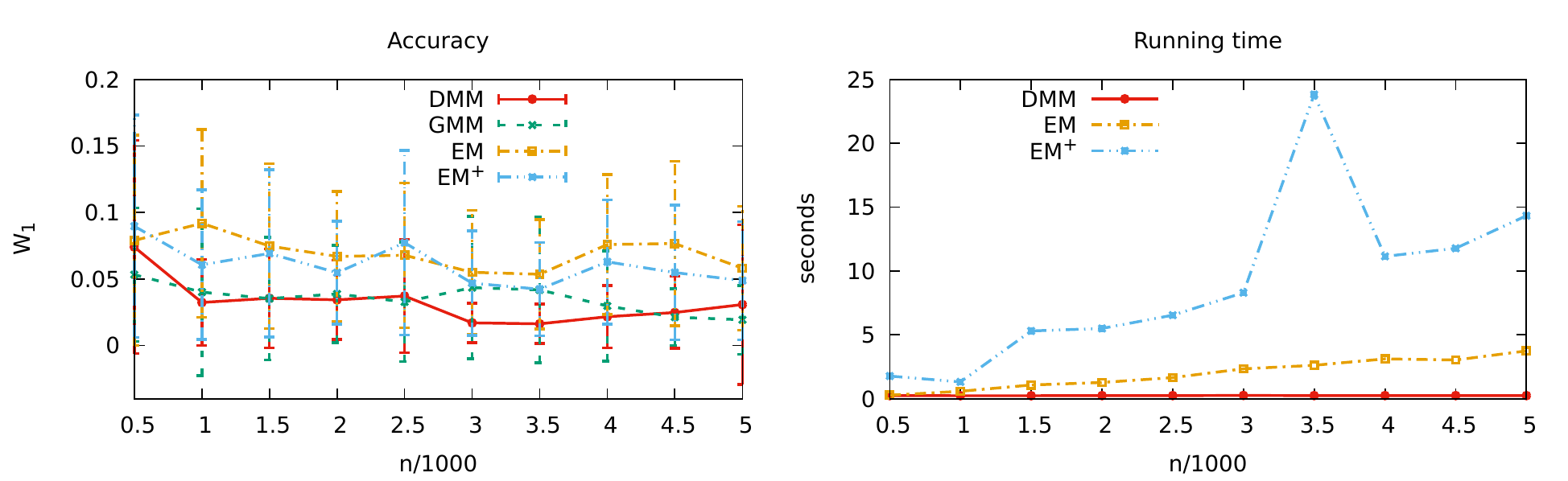}
    \caption{Comparison of different methods when components completely overlap.}
    \label{fig:DMM-k2}
\end{figure}
Again the estimation errors are similar, but EM$^+$ is much slower than EM without substantial gain in the accuracy. Specifically, at 5,000 samples, EM is still 15 times slower than DMM, but EM$^+$ is 60 times slower.

Lastly, we demonstrate a faster rate in the well-separated case as shown in \prettyref{thm:main-W1-adaptive}. In this experiment, the samples are drawn from $\frac{1}{2}N(1,1)+\frac{1}{2}N(1,-1)$. The results are shown in \prettyref{fig:DMM-separated}.
\begin{figure}[h]
    \centering
    \includegraphics[width=\linewidth]{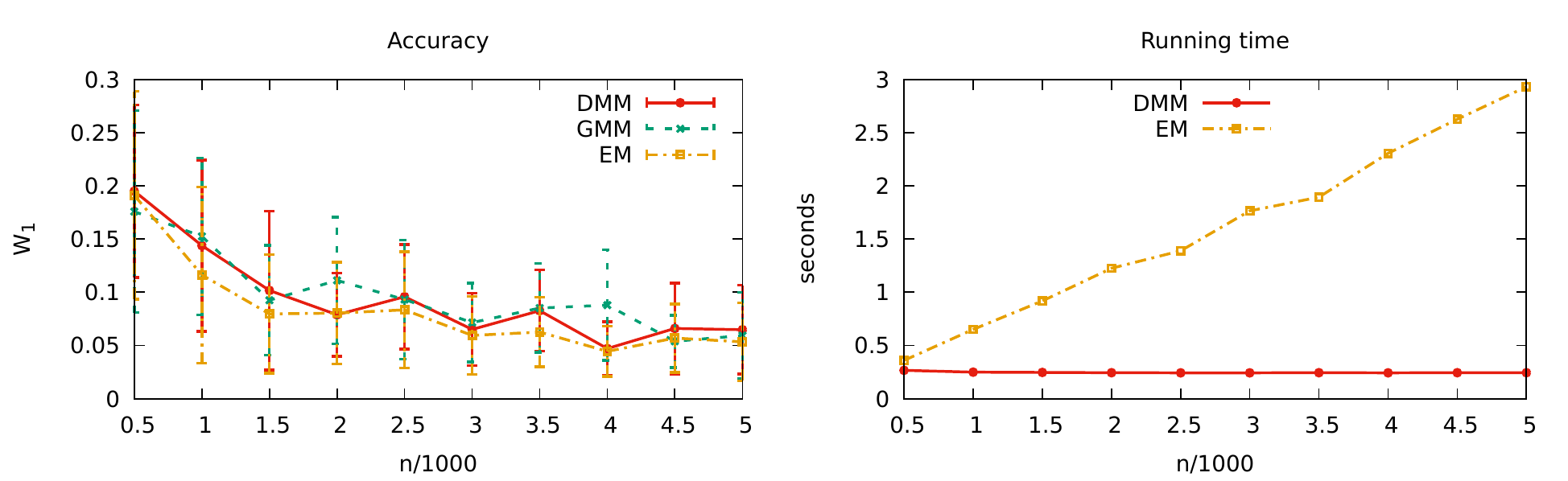}
    \caption{Comparison of different methods when components are separated.}
    \label{fig:DMM-separated}
\end{figure}
In this case, the estimation error decays faster than the one shown in \prettyref{fig:DMM-k2}. The larger absolute values of the Wasserstein distance is an artifact of the range of the means.

\paragraph{Unknown variance. }
We conduct an experiment under the same five-component Gaussian mixture as before, but now the estimators no longer have access to the true variance parameter. In this case, Lindsay's algorithm (see \prettyref{algo:Lindsay}) involves the empirical moments of degrees up to 10, among which higher order moments are hard to estimate with limited samples. Indeed, the standard deviation of $\Expect_n[X^{10}]$ is $\frac{1}{\sqrt{n}}\sqrt{\var[X^{10}]} \approx 473$ under this specified model with $n=5000$ samples. 
To resolve this issue, we introduce an extra step to determine whether an empirical moment is too noisy and accept the empirical moment of order $j$ only when its empirical variance satisfies
\begin{equation}
    \label{eq:screening}
    \frac{\Expect_n[X^{2j}]-(\Expect_n[X^j])^2}{n}\le \tau ,
\end{equation}
where the left-hand side of \prettyref{eq:screening} is an estimate of the variance of $\Expect_n[X^{j}]$ and $\tau$ is some threshold. The estimated mixture model has $\tilde k$ components for the largest $\tilde k$ such that the first $2\tilde k$ empirical moments are all accepted. 
In the experiment, we choose $\tau=0.5$. The results are shown in \prettyref{fig:Lindsay}.
\begin{figure}[h]
    \centering
    \includegraphics[width=\linewidth]{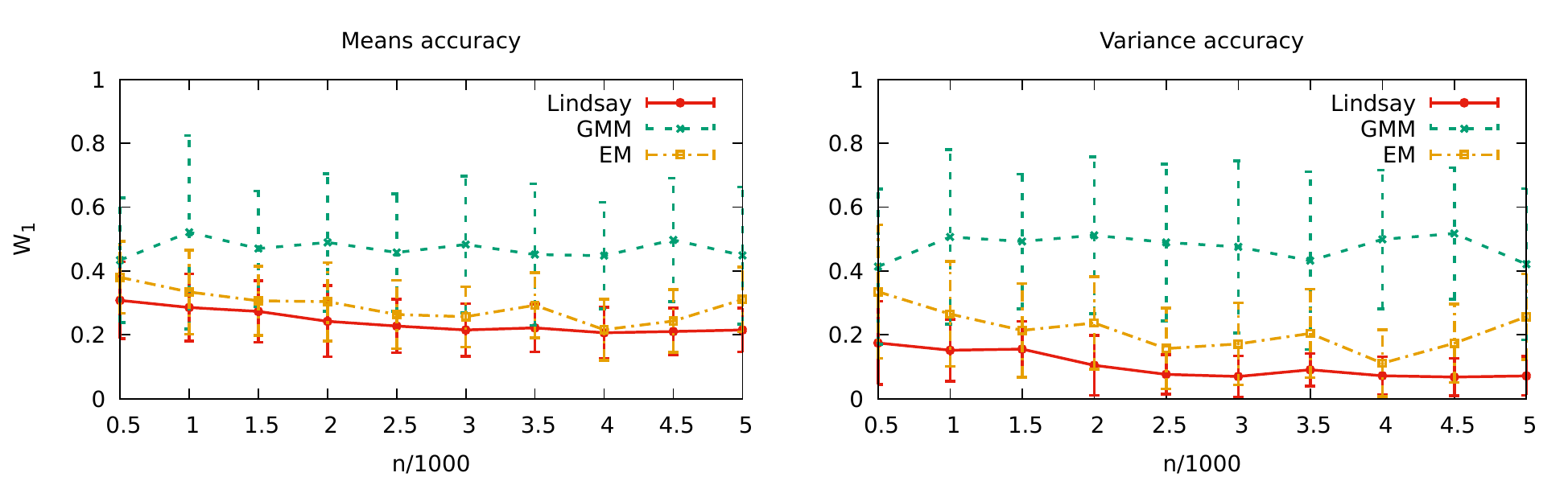}
    \caption{Comparison of different methods with unknown variance. }
    \label{fig:Lindsay}
\end{figure}
The performance of the Lindsay and EM estimators are similar and better than GMM, which is possibly due to the large variance of higher order empirical moments. The running time comparison are similar to before and thus are omitted. The experiments under the models of \prettyref{fig:DMM-k2} and \prettyref{fig:DMM-separated}
 also yield similar results.



\section{Extensions and discussions}
\label{sec:discuss}
\subsection{Gaussian location-scale mixtures}
In this paper we focus on the Gaussian location mixture model \prettyref{eq:model}, where all components share the same (possibly unknown) variance.
One immediate extension is the Gaussian location-scale mixture model with heteroscedastic components:
\begin{equation}
\sum_{i=1}^k w_i N(\mu_i, \sigma_i^2)
\label{eq:model-locationscale}
\end{equation}
Parameter estimation for this model turns out to be significantly more difficult than the location mixture model, in particular
\begin{itemize}
        \item The likelihood function is unbounded. In fact, it is well-known that the maximum likelihood estimator is ill-defined \cite[p.~905]{KW56}.
        For instance, 
consider $k=2$, for any sample size $n$, we have
\[
\sup_{p_1,p_2,\theta_1,\theta_2,\sigma} \prod_{i=1}^n \qth{\frac{p_1 }{\sigma_1}\varphi\pth{\frac{X_i-\theta_1}{\sigma_1} }+
\frac{p_2 }{\sigma_2}\varphi\pth{\frac{X_i-\theta_2}{\sigma_2} }} = \infty,
\]
achieved by, e.g., $\theta_1=X_1, p_1=1/2$, $\sigma_2=1$, and $\sigma_1\to0$.
        
        \item 
In this model, the identifiability result based on moments is not completely settled and we do not have a counterpart of \prettyref{lmm:identify}. Note that the model \prettyref{eq:model-locationscale} comprises $3k-1$ free parameters ($k$ means, $k$ variances, and $k$ weights normalized to one), so it is expected to be identified through its first $3k-1$ moments.
However, the intuition of equating the number of parameters and the number of equations is known to be wrong as pointed out by Pearson \cite{Pearson1894}, who showed that for $k=2$, five moments are insufficient and six moments are enough.
The recent result \cite{ARS2016} showed that, if the parameters are in general positions, then $3k-1$ moments can identify the Gaussian mixture distribution up to finitely many solutions (known as algebraic identifiability). Whether $3k$ moments can uniquely identify the model (known as rational identifiability) in general positions remains open, except for $k=2$.
In the worst case, we need at least $4k-2$ moments for identifiability since for scale-only Gaussian mixtures all odd moments are zero 
(see \prettyref{sec:general-mix} for details).

\end{itemize}

Besides the issue of identifiability, the optimal estimation rate under the Gaussian location-scale mixture model is resolved only in special cases. 
The sharp rate is only known in the case of two components to be $\Theta(n^{-1/12})$ for estimating means and $\Theta(n^{-1/6})$ for estimating variances \cite{HP15}, achieved by a robust variation of Pearson's method of moment equations \cite{Pearson1894}.
For $k$ components, the optimal rate is known to be $n^{-\Theta(1/k)}$ \cite{MV2010,KMV2010}, achieved by an exhaustive grid search on the parameter space.
See \cite[Table 1]{li2017robust} for a comprehensive review of the existing results in the univariate case.
In addition, the above results all aim to recover parameters of all components (up to a global permutation), which necessarily require extra assumptions including lower bounds on mixing weights and separation between components; recovering the mixing distribution with respect to, say, Wasserstein distance, remains open.

\subsection{Multiple dimensions}
\label{sec:multiple}
So far we have focused on Gaussian mixtures in one dimension. The multivariate version of this problem has been studied in the context of clustering, or classification, which typically usually requires non-overlapping components \cite{Dasgupta1999,VW2004}. 
One commonly used approach is dimensionality reduction: project samples onto some lower dimensional subspace and perform clustering, then map back to the original space. Common choices of the subspace include random subspaces and subspaces obtained from the singular value decomposition.
The approach using random subspace is analyzed in \cite{Dasgupta1999,AK2001}, and requires a pairwise separation polynomial in the dimensions;
the subspace from singular value decomposition is analyzed in \cite{VW2004,AM2005}, 
and requires a pairwise separation that grows polynomially in the number of components.
Tensor decomposition for spherical Gaussian mixtures has been studied in \cite{HK2013,AGHKT2014}, which requires the stronger assumption that
that means are linear independent and is inapplicable in lower dimensions, say, two or three dimensions. 

When components are allowed to overlap significantly, the random projection approach is also adopted by \cite{MV2010,KMV2010,HP15}, where the estimation  problem in high dimensions is reduced to that in one dimension, so that univariate methodologies can be invoked as a primitive.
We provide an algorithm (\prettyref{algo:d-dim}) using similar random projection ideas to estimate the parameters of a Gaussian mixture model in $d$ dimensions for known covariance matrices, using the univariate algorithm in \prettyref{sec:known} as a subroutine, and obtain the estimation guarantee in \prettyref{thm:d-dim}; the unknown covariance case can be handled analogously using the algorithm in \prettyref{sec:unknown} instead. 
However, the dependency of the performance guarantee on the dimension is highly suboptimal,\footnote{
    Specifically, in $d$ dimensions, estimating each coordinate independently incurs 
		an $\ell_2$-loss proportional to $\sqrt{d}$;
    however, it is possible to achieve $d^{1/4}$ by, e.g., spectral methods (see \prettyref{lmm:pca}).} which stems from the fact that the method based on random projections estimates each coordinate independently. Moreover, this method needs to match the Gaussian components of the estimated model in each random direction, which necessarily requires lower bounds on the mixing weights and separation between the means.
\begin{algorithm}[ht]
    \caption{Learning a $k$-component Gaussian mixture in $d$ dimensions.}
    \label{algo:d-dim}
    \begin{algorithmic}[1]
        \REQUIRE $ n $ samples $ X_1,X_2,\dots,X_n \in \reals^d$, common covariance matrix $\Sigma$, and separation parameter $\tau$, radius parameter $\rho$.
        \ENSURE estimated mixing distribution $\hat\pi$ with weights and means $(\hat w_j,\hat\mu_j)$ for $j=1,\ldots,k$ .
        \STATE Let $(b_1,\dots,b_d)$ be a set of random orthonormal basis in $\reals^d$, and $r=b_1$.
        \STATE Let $\{(w_j,\mu_j)\}$ be the outcome of \prettyref{algo:known} using $n$ projected samples $\inner{X_1, r}, \dots, \inner{X_n,r}$, variance $r^\top\Sigma r$, and interval $[-\rho,\rho]$.
        \STATE Reordering the indices such that $\mu_1<\mu_2<\dots<\mu_k$.
        \STATE Initialize $k$ weights $\hat w_j = w_j$ and means $\hat\mu_j=(0,\dots,0)$.
        \FOR{$i=1$ \TO $d$}
        \STATE Let $r'=r+\tau b_i$.
        \STATE Let $\{\mu_j'\}$ be the estimated means (weights are ignored) from \prettyref{algo:known} using $n$ projected samples $\inner{X_1, r'}, \dots, \inner{X_n, r'}$, variance $r'^\top\Sigma r'$, and interval $[-\rho-\tau,\rho+\tau]$. \label{line:muj-p}
        \STATE Reordering the indices such that $\mu_1'<\mu_2'<\dots<\mu_k'$.\label{line:muj-p-order}
        \STATE Let $\hat\mu_j:=\hat\mu_j+b_i\frac{\mu_j'-\mu_j}{\tau}$ for $j=1,\dots,k$.
        \ENDFOR
    \end{algorithmic}
\end{algorithm}
\begin{theorem}
    \label{thm:d-dim}
    Suppose in a $d$-dimensional Gaussian mixture $\sum_{j=1}^k w_jN(\mu_j,\Sigma)$, 
    \[
        \Norm{\mu_j}_2\le M,\quad \Norm{\mu_i-\mu_j}_2\le \epsilon, \quad w_j\ge \epsilon',\quad \forall~i\ne j.
    \]
    Then \prettyref{algo:d-dim} with $n>(\Omega_k(\frac{M}{\tilde\epsilon\epsilon'}))^{4k-2}\log\frac{d}{\delta}$ samples, $\tau=\frac{\tilde\epsilon}{2M}$, and $\rho=M$, where $\tilde\epsilon=\frac{\delta\epsilon}{k^2\sqrt{d}}$, yields $\hat\pi$ such that, with  probability $1-2\delta$,
    \[
        W_1(\pi,\hat\pi)<O_k\pth{\sqrt{d}\frac{M\epsilon_n}{\tau\epsilon'}},
    \]
    where $\pi=\sum_j w_j\delta_{\mu_j}$ and $\epsilon_n=\min\{(\frac{n}{\log(d/\delta)})^{-\frac{1}{4k-2}},\tilde\epsilon^{2-2k}\sqrt{\frac{\log(d/\delta)}{n}}\}$.
\end{theorem}
\begin{proof}
    By the distribution of random direction $r$ on the unit sphere (see \prettyref{lmm:proj}) and the union bound, we obtain that, with probability $1-\delta$,
    \[
        \abs{\inner{\mu_i-\mu_j,r}} >  2\tilde\epsilon,\quad \forall~i\ne j.
    \]
    Without loss of generality, assume $\inner{\mu_1,r}<\dots<\inner{\mu_k,r}$. Applying \prettyref{thm:main-W1} yields that, with probability $1-\frac{\delta}{d+1}$, 
    \[
        W_1(\pi_r,\hat\pi_r)\le O_k\pth{M\pth{\frac{n}{\log(d/\delta)}}^{-\frac{1}{4k-2}}},
    \]
    where $\pi_r$ denotes the Gaussian mixture projected on $r$ and $\hat\pi_r$ is its estimate. The right-hand side of the above inequality is less than $c\epsilon_r\epsilon'$ for some constant $c<0.5$ when $n>(\Omega_k(\frac{M}{\tilde\epsilon\epsilon'}))^{4k-2}\log\frac{d}{\delta}$. Applying \prettyref{thm:main-W1-adaptive} yields that 
    \[
        W_1(\pi_r,\hat\pi_r)\le O_k\pth{M\tilde\epsilon^{2-2k}\sqrt{\frac{\log(d/\delta)}{n}}}. 
    \]
    Hence, we obtained $W_1(\pi_r,\hat\pi_r)\le O_k(M\epsilon_n)$. It follows from \prettyref{lmm:w1-parameters} that, after reordering indices,
    \begin{equation}
        \label{eq:d-dim-parameters}
        |\inner{\mu_j,r}-\tilde\mu_{j}|< O_k(M\epsilon_n/\epsilon'),\quad |w_j-\hat w_j|<O_k(M\epsilon_n/\tilde\epsilon).
    \end{equation}
    On each direction $r_\ell = r+\tau b_\ell$, the means are separated by $|\inner{\mu_i-\mu_j, r_\ell}|>2\tilde\epsilon-2M\tau>\tilde\epsilon$ and the ordering of the means remains the same as on direction $r$. Therefore the accuracy similar to \prettyref{eq:d-dim-parameters} continues to hold for the estimated means $\tilde\mu_{\ell,j}$ ($\mu_j'$ in lines \ref{line:muj-p} and \ref{line:muj-p-order}). Note that $\mu_j=\sum_{\ell=1}^d b_\ell \frac{\inner{\mu_j,r_\ell}-\inner{\mu_j,r}}{\tau}$ and $\hat\mu_j=\sum_{\ell=1}^d b_\ell \frac{\tilde\mu_{\ell,j}-\tilde\mu_{j}}{\tau}$. Therefore,
    \[
        \Norm{\hat\mu_j-\mu_j}_2^2\le\sum_{\ell=1}^d\pth{\frac{O_k(M\epsilon_n/\epsilon')}{\tau}}^2.
    \]
    Applying the triangle equality yields that
    \[
        W_1(\pi,\hat\pi)<\sqrt{d}O_k(M\epsilon_n/\epsilon')/\tau+MO_k\pth{M\epsilon_n/\tilde\epsilon}
        <O_k\pth{\sqrt{d}\frac{M\epsilon_n}{\tau\epsilon'}}.\qedhere
    \]
\end{proof}

It is interesting to directly extend the DMM methodology to multiple dimensions, which is challenging both theoretically and algorithmically:
\begin{itemize}
	\item To extend the proof technique in multiple dimensions, the challenge is to obtain a multi-dimensional moment comparison theorem analogous to \prettyref{prop:stable1} or \ref{prop:stable2}, the key step leading to the optimal rate.
These results are proved in \prettyref{sec:pf-compare}
by the primal formulation of the Wasserstein distance and its simple formula \prettyref{eq:W1-CDF} in one dimension \cite{villani.topics}. Alternatively, they can be proved via the dual formula \prettyref{eq:W1-dual} which holds in any dimension; however, the proof relies on the Newton's interpolation formula, which is again difficult to generalize or analyze in multiple dimensions.
\item 
To obtain a computationally efficient algorithm, we rely on the semidefinite characterization of the moment space in one dimension to denoise the noisy estimates of moments. In multiple dimensions, however, it remains open how to efficiently describe the moment space \cite{Lasserre2009} as well as how to extend the Gauss quadrature rule to multivariate distributions.

\end{itemize}

\subsection{General finite mixtures}
\label{sec:general-mix}
Though this paper focuses on Gaussian location mixture models, the moments comparison theorems in \prettyref{sec:w1-moments} are independent of properties of Gaussian. As long as moments of the mixing distribution are estimated accurately, similar theory and algorithms can be obtained. Unbiased estimate of moments exists in many useful mixture models, including exponential mixtures \cite{Jewell1982}, Poisson mixtures \cite{KX2005}, and more generally the quadratic variance exponential family (QVEF) whose variance is at most a quadratic function of the mean \cite[(8.8)]{morris1982natural}.


As a closely related topic of this paper, we discuss the Gaussian scale mixture model in details, which has been extensively studied in the statistics literature \cite{AM1974} and is widely used in image and video processing \cite{WS2000,PSWS2003}. In a Gaussian scale mixture, a sample is distributed as
\[
    X\sim \sum_{i=1}^{k}w_iN(0,\sigma_i^2)=\int N(0,\sigma^2)\diff \nu(\sigma^2),
\]
where $\nu=\sum_{i=1}^kw_i\delta_{\sigma_i^2}$ is a $k$-atomic mixing distribution. Equivalently, a sample can be represented as $X=\sqrt{V}Z$, where $V\sim \nu$ and $Z$ is standard normal independent of $V$. In this model, samples from different components significantly overlap, so clustering-based algorithms will fail. Nevertheless, moments of $\nu$ can be easily estimated, for instance, using $\Expect_n[X^{2r}]/\Expect[Z^{2r}]$ for $m_r(\nu)$ with accuracy $O_r(1/\sqrt{n})$. Applying a similar algorithm to DMM in \prettyref{sec:known}, we obtain an estimate $\hat \nu$ such that
\[
    W_1(\nu,\hat \nu)\le O_k(n^{-\frac{1}{4k-2}}),
\]
with high probability.

Moreover, a matching minimax lower bound can be established using similar techniques to \prettyref{sec:lb}. Analogous to \prettyref{eq:optimal-lb}, let $\nu$ and $\nu'$ be a pair of $k$-atomic distributions supported on $[0,\epsilon]$ such that they match the first $2k-2$ moments, and let
\[
    \pi=\int N(0,\sigma^2)\diff \nu(\sigma^2),\quad 
    \pi'=\int N(0,\sigma^2)\diff \nu'(\sigma^2),
\]
which match their first $4k-3$ moments and are $\sqrt{\epsilon}$-subgaussian. Applying \prettyref{lmm:d-matching} with $\pi*N(0,0.5)$, $\pi'*N(0,0.5)$, and $\epsilon=O_k(n^{-\frac{1}{4k-2}})$ yields a minimax lower bound 
\[
    \inf_{\hat\nu}\sup_{P\in\calG_k}\Expect_P W_1(\nu,\hat \nu)\ge \Omega_k\pth{n^{-\frac{1}{4k-2}}},
\]
where the estimator $\hat \nu$  is measurable with respect to $X_1,\dots,X_n\sim P$, and the space of $k$ Gaussian scale mixtures is defined as
\[
    \calG_k=\sth{\int N(0,\sigma^2)\diff \nu(\sigma^2): \nu \text{ is $k$-atomic supported on }[0,1]}.
\]

\section{Proofs}
\label{sec:proof}

We begin by briefly reviewing some background on polynomial interpolation, which plays a key role in the proofs.

\subsection{Polynomial interpolation, majorization, and the Neville diagram}
\label{sec:poly}
Given a function $f$ and a set of distinct points (commonly referred to as \emph{nodes}) $\{x_0,\dots,x_{k}\}$, there exists a unique polynomial $P$ of degree $k$ that coincides with $f$ on every node. The interpolating polynomial $P$ can be expressed in the \emph{Lagrange form} as
\begin{equation}
    \label{eq:Lagrange}
    P(x)=\sum_{i=0}^k f(x_i)\frac{\prod_{j\ne i}(x-x_j)}{\prod_{j\ne i}(x_i-x_j)},
\end{equation}
and, alternatively, in the \emph{Newton form} as
\begin{equation}
    \label{eq:interpolation-Newton}
    P(x)=a_0+a_1(x-x_0)+\dots+a_{k}(x-x_0)\cdots(x-x_{k-1}).
\end{equation}
Let us pause to emphasize that, in numerical analysis, typically the Newton form is introduced for computational considerations so that one does not need to recompute all coefficients when an extra node is introduce \cite{stoer.2002}. Here for our theoretical analysis the Newton form turns out to be crucial, which offers better bound on the coefficients of the interpolating polynomials.

The coefficients in \prettyref{eq:interpolation-Newton} can be successively calculated using $a_0=f(x_0)$, $a_0+a_1(x_1-x_0)=f(x_1)$, etc. In general, they coincide with the divided differences $a_r=f[x_0,\dots,x_{r}]$ that are recursively defined as 
\begin{equation}
    \label{eq:div-recursion}
    f[x_i]=f(x_i)\quad f[x_i,\dots,x_{i+r}]=\frac{f[x_{i+1},\dots,x_{i+r}]-f[x_i,\dots,x_{i+r-1}]}{x_{i+r}-x_i}.
\end{equation}
The above recursion can be calculated by the following \emph{Neville's diagram} (cf.~\cite[Section 2.1.2]{stoer.2002}):
\begin{center}
    \begin{tikzpicture}[every node/.style={draw,shape=circle,fill,inner sep=1.5pt,align=left},scale=0.7,transform shape]
        \def\num{3}
        \def\shift{0.5}
        \def\lshift{1.5}
        \foreach \n in {0,...,\num} {
            \foreach \k in {0,...,\n} {
                \node at (-2*\n,-2*\k+\n){};
            }
        }

        \pgfmathtruncatemacro\numd{\num-1}
        \foreach \n in {1,...,\numd} {
            \foreach \k in {1,...,\n} {
                \draw (-2*\n-2,-2*\k+2+\n-1) -- (-2*\n,-2*\k+2+\n) -- (-2*\n-2,-2*\k+2+\n+1);
            }
        }
        
        \foreach \k in {1,...,\num} {
            \pgfmathtruncatemacro\kd{\k-1}
            \node[draw=none, fill=none] at (-2*\num-\lshift,-2*\k+2+\num){$x_\kd$};
            \node[draw=none, fill=none] at (-2*\num,-2*\k+2+\num+\shift){$f[x_\kd]$};
            \node[draw=none, fill=none] at (-2*\k,-\k+2-1){$\vdots$};
        }
        
        \foreach \k in {1,...,\numd} {
            \pgfmathtruncatemacro\kd{\k-1}
            \node[draw=none, fill=none] at (-2*\num+2,-2*\k+2+\num-1+\shift){$f[x_\kd,x_\k]$};
            \draw (-2*\k,-\k) -- (-2*\k-2,-\k-1);
        }

        \node[draw=none, fill=none] at (-2,1+\shift){$f[x_0,x_1,x_2]$};
        \node[draw=none, fill=none] at (-0,0+\shift){$f[x_0,\dots,x_k]$};
        \node[draw=none, fill=none] at (-6-\lshift,-3){$x_{k}$};
        \node[draw=none, fill=none] at (-6-\lshift,-2){$\vdots$};
        \node[draw=none, fill=none] at (-6,-3-\shift){$f[x_k]$};
        \draw[dotted] (-2*.6,-1*.6) -- (-2*.4,-1*.4);
        \draw[dotted] (-2*.6,1*.6) -- (-2*.4,1*.4);
        \node[draw=none, fill=none] at (-6,4){$r=0$};
        \node[draw=none, fill=none] at (-4,4){$1$};
        \node[draw=none, fill=none] at (-2,4){$2$};
        \node[draw=none, fill=none] at (-1,4){$\dots$};
        \node[draw=none, fill=none] at (-0,4){$k$};
        \draw (-8,3.7) -- (1,3.7);
        \draw (-6.8,4.5) -- (-6.8,-4);
    \end{tikzpicture}
\end{center}
In Neville's diagram, the $r\Th$ order divided differences are computed in the $r\Th$ column, and are determined by the previous column and the nodes. The coefficients in \prettyref{eq:interpolation-Newton} are found in the top diagonal. In this paper Neville's diagram will be used to bound the coefficients in Newton formula \prettyref{eq:interpolation-Newton}; cf.~\prettyref{lmm:divided}.

Interpolating polynomials are the main tool to prove moment comparison theorems in \prettyref{sec:w1-moments}. Specifically, we will interpolate step functions by polynomials in order to bound the difference of two CDFs via their moment difference. Therefore, it is crucial to have a good control over the coefficients of the interpolating polynomial. To this end, it turns out the Newton form is more convenient to use than the Lagrange form because the former takes into account the cancellation between each term in the polynomial. Indeed, in the Lagrange form \prettyref{eq:Lagrange}, if two nodes are very close, then the individual terms can be arbitrarily large, even if $f$ itself is a smooth function. In contrast, each term of \prettyref{eq:interpolation-Newton} is stable when $f$ is smooth since divided differences are closely related to derivatives. The following example illustrates this point:

\begin{example}[Lagrange versus Newton form]
    \label{ex:Lagrange-Newton}
    Given three points $x_1=0, x_2=\epsilon, x_3=1$ with $f(x_1)=1, f(x_2)=1+\epsilon, f(x_3)=2$, the interpolating polynomial is $P(x)=x+1$. The next equation gives the interpolating polynomial in Lagrange's and Newton's form respectively.
    \begin{align*}
      \text{Lagrange: }&P(x)=\frac{(x-\epsilon)(x-1)}{\epsilon}+(1+\epsilon)\frac{x(x-1)}{\epsilon(\epsilon-1)}+2\frac{x(x-\epsilon)}{1-\epsilon};\\
      \text{Newton: } &P(x)=1+x+0.
    \end{align*}
        The coefficients in the Newton form are bounded, while those in the Lagrange form blow up as $\epsilon\to0$.
\end{example}

Polynomial interpolation can be generalized to interpolate the value of derivatives, known as the \emph{Hermite interpolation}. Formally, given a function $f$ and distinct nodes $x_0<x_1<\ldots<x_m$, there exists a unique polynomial $P$ of degree $k$ satisfying $P^{(j)}(x_i)=f^{(j)}(x_i)$ for $i=0,\dots,m$ and $j=0,\dots,k_i-1$, where $k+1=\sum_{i=0}^m k_i$. Analogous to the Lagrange formula \prettyref{eq:Lagrange}, $P$ can be explicitly constructed with the help of the generalized Lagrange polynomials, and an explicit formula is given in \cite[pp.~52--53]{stoer.2002}. The Newton form \prettyref{eq:interpolation-Newton} can also be extended by using generalized divided differences, which, for repeated nodes, is defined as the value of the derivative:
\begin{equation}
    \label{eq:div-repeated}
    f[x_i,\dots,x_{i+r}] \triangleq \frac{f^{(r)}(x_0)}{r!},\quad x_i=x_{i+1}=\ldots=x_{i+r},
\end{equation}
To this end, we define an expanded set of nodes by repeating each $x_i$ for $k_i$ times:
\begin{equation}
    \label{eq:new-sequence}
    \underbrace{x_0=\ldots =x_0}_{k_0}<\underbrace{x_1=\ldots =x_1}_{k_1}<\ldots<\underbrace{x_m=\ldots=x_m}_{k_m}.
\end{equation}
The Hermite interpolating polynomial is obtained by \prettyref{eq:interpolation-Newton} using this new set of nodes and generalized divided differences, which can also be calculated from the Neville's diagram verbatim by replacing divided differences by derivatives whenever encountering repeated nodes. 
Below we give an example using Hermite interpolation to construct polynomial majorant, which will be used to prove moment comparison theorems in \prettyref{sec:w1-moments}.

\begin{example}[Hermite interpolation and polynomial majorization]
    \label{ex:Hermite}
    Let $f(x)=\indc{x\le 0}$. We want to find a polynomial majorant $P\geq f$ such that $P(x)=f(x)$ on $x=\pm 1$. To this end we interpolate the values of $f$ on $\{-1,0,1\}$ with the following constraints:
    \begin{center}
        \begin{tabular}{ c | c  c  c }
            \toprule
            $x$ & $-1$ & $0$ & $1$ \\
            \hline
            $P(x)$  & 1 & 1 & 0 \\
            $P'(x)$ & 0 & \text{any} & 0 \\
            \bottomrule
        \end{tabular}
    \end{center}
    The resulting polynomial $P$ has degree four and majorizes $f$ \cite[p.~65]{Akhiezer1965}. To see this, we note that $P'(\xi)=0$ for some $\xi\in (-1,0)$ by Rolle's theorem. Since $P'(-1)=P'(1)=0$, $P$ has no other stationary point than $-1,\xi,1$, and thus decreases monotonically in $(\xi,1)$. Hence, $-1,1$ are the only local minimum points of $P$, and thus $P\ge f$ everywhere. The polynomial $P$ is shown in \prettyref{fig:step-approx}.

    To explicitly construct the polynomial, we expand the set of nodes to $-1,-1,0,1,1$ according to \prettyref{eq:new-sequence}. Applying Newton formula \prettyref{eq:interpolation-Newton} with generalized divided differences from the Neville's diagram \prettyref{fig:Neville}, we obtain that $P(x)=1-\frac{1}{4}x(x+1)^2+\frac{1}{2}x(x+1)^2(x-1)$.

    \begin{figure}[ht]
        \centering
        \subfigure[Neville's diagram.]
        {\label{fig:Neville}
            \begin{tikzpicture}[every node/.style={draw,shape=circle,fill,inner sep=1.5pt,align=left},scale=0.5,transform shape]
                \def\labels{{1,1,1,0,0},{0,0,$-1$,0},{0,$-1/2$,1},{$-1/4$,3/4},{1/2}}
                \foreach[count=\n] \x in \labels{
                    \foreach[count=\k] \y in \x{
                        \node[label={[yshift=0cm]\y}] at (2*\n,-2*\k-\n) {};
                        \ifthenelse{\n > 1}{
                            \draw ({2*(\n-1)},{-2*\k-(\n-1)}) -- (2*\n,-2*\k-\n) -- ({2*(\n-1)},{-2*(\k+1)-(\n-1)});
                        }{}
                    }
                }

                \def\labels{$t_0=-1$,$t_1=-1$,$t_2=0$,$t_3=1$,$t_4=1$}
                \foreach[count=\k] \y in \labels{
                    \node[draw=none, fill=none, anchor=west] at (0,-2*\k-1) {\y};
                }

                \foreach \n in {1,2,3} {
                    \draw[red,very thick]
                    (2*\n,-2-\n) node[black,thin] {} -- ({2*(\n+1)},{-2-(\n+1)}) node[black,thin] {}
                    (2*\n,{-2*(5-\n)-\n}) node[black,thin] {} -- ({2*(\n+1)},{-2*(4-\n)-(\n+1)}) node[black,thin] {};
                }



            \end{tikzpicture}
        }
        \qquad
        \subfigure[Hermite interpolation.]
        {\label{fig:step-approx}
            \includegraphics[width=.4\textwidth]{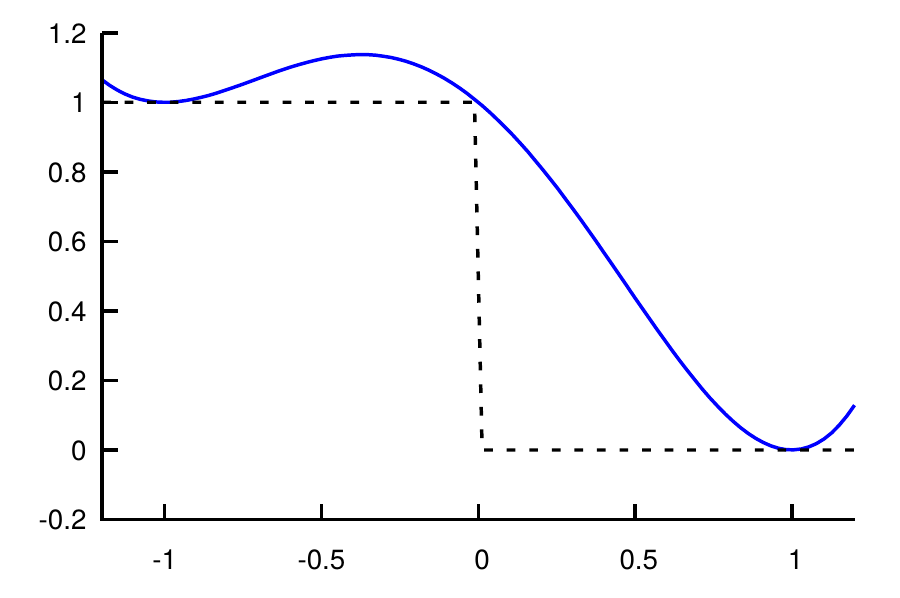}}
        \caption{Neville's diagram and Hermite interpolation.
            In \subref{fig:Neville}, values are recursively calculated from left to right. For example, the red thick line shows that $f[-1,-1,0,1]$ is obtained by $\frac{-1/2-0}{1-(-1)}=-1/4$.}
    \end{figure}
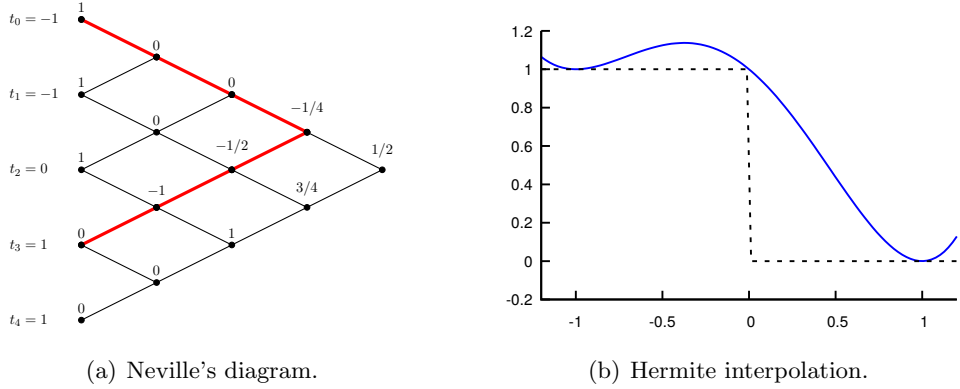
\end{example}

\subsection{Proofs of moments comparison theorems}
\label{sec:pf-compare}
In this subsection we prove Propositions \ref{prop:stable1} and \ref{prop:stable2}. As a warm-up, we start by proving \prettyref{lmm:identify}, with the purpose of introducing the apparatus of interpolating polynomials. Throughout this section, we use
\[
    F_\pi(x) \triangleq \pi((-\infty,x]).
\]
to denote the CDF of a distribution $\pi$.
\begin{proof}[Proof of \prettyref{lmm:identify}]
    We only need to prove the ``if'' part. 
    \begin{enumerate}
        \item Denote the union of the support sets of $\nu$ and $\nu'$ by $S$. Here $S$ is of size at most $2k$. For any $t\in\reals$, there exists a polynomial $P$ of degree at most $2k-1$ to interpolate $x\mapsto \indc{x\le t}$ on $S$. Since $m_i(\nu)=m_i(\nu')$ for $i=1,...,2k-1$, we have 
        \[
            F_{\nu}(t)=\Expect_\nu[\indc{X\le t}]=\Expect_{\nu}[P(X)]=\Expect_{\nu'}[P(X)]=\Expect_{\nu'}[\indc{X\le t}]=F_{\nu'}(t).
        \]
        \item Denote the support set of $\nu$ by $S'=\{x_1,\dots,x_k\}$. Let $Q(x)=\prod_i(x-x_i)^2$, a non-negative polynomial of degree $2k$. Since $m_i(\nu)=m_i(\nu')$ for $i=1,...,2k$, we have 
        \[
            \Expect_{\nu'}[Q(X)]=\Expect_{\nu}[Q(X)]=0.
        \]
        Therefore, $\nu'$ is also supported on $S'$ and thus is $k$-atomic. The conclusion follows from the first case of \prettyref{lmm:identify}. \qedhere
    \end{enumerate}
\end{proof}

Next we prove \prettyref{prop:stable1-union}, which is slightly stronger than \prettyref{prop:stable1}. We provide three proofs: the first two are based on the primal (coupling) formulation of $W_1$ distance \prettyref{eq:W1-CDF}, and the third proof uses the dual formulation \prettyref{eq:W1-dual}. 
Specifically,
\begin{itemize}
    \item The first proof uses polynomials to interpolate step functions, whose expected values are the CDFs. The closeness of moments imply the closeness of distribution functions and thus, by \prettyref{eq:W1-CDF}, a small Wasserstein distance. Similar idea applies to the proof of \prettyref{prop:stable2} later.
    \item The second proof finds a polynomial that preserves the sign of the difference between two CDFs, and then relate the Wasserstein distance to the integral of that polynomial. Related idea has been used in \cite[Lemma 20]{MV2010} which finds a polynomial that preserves the sign of the difference between two Gaussian mixture densities. 
    \item The third proof uses polynomials to approximate 1-Lipschitz functions, whose expected values are related to the Wasserstein distance via the dual formulation \prettyref{eq:W1-dual}. 
\end{itemize}
\begin{prop}
    \label{prop:stable1-union}    
    Let $\nu$ and $\nu'$ be discrete distributions supported on a total of $\ell$ atoms in $[-1,1]$. If
    \begin{equation}
        \label{eq:momentdiff}
        |m_i(\nu)-m_i(\nu')|\le \delta,\quad i=1,\ldots,\ell-1,
    \end{equation}
    then
    \[
        W_1(\nu,\nu')\le O\pth{\ell \delta^{\frac{1}{\ell-1}}}.
    \]
\end{prop}
\begin{proof}[First proof of \prettyref{prop:stable1-union}]
    Suppose $\nu$ and $\nu'$ are supported on
    \begin{equation}
        \label{eq:supports-1}
        S=\{t_1,\dots,t_\ell\},\quad t_1<t_2<\dots<t_{\ell}.    
    \end{equation}
    Then, using the integral representation \prettyref{eq:W1-CDF}, the $W_1$ distance reduces to 
    \begin{equation}
        \label{eq:W1-CDF-discrete}
        W_1(\nu,\nu')=\sum_{r=1}^{\ell-1}|F_{\nu}(t_r)-F_{\nu'}(t_r)|\cdot|t_{r+1}-t_r|.
    \end{equation}
    For each $r$, let $f_r(x)=\indc{x\le t_r}$, and $P_r$ be the unique polynomial of degree $\ell-1$ to interpolate $f_r$ on $S$. In this way we have $f_r=P_r$ almost surely under both $\nu$ and $\nu'$, and thus
    \begin{equation}
        \label{eq:CDF-poly}    
        |F_{\nu}(t_r)-F_{\nu'}(t_r)|=|\Expect_{\nu}f_r-\Expect_{\nu'}f_r|
        =|\Expect_{\nu}P_r-\Expect_{\nu'}P_r|.
    \end{equation}
    $P_r$ can expressed using Newton formula \prettyref{eq:interpolation-Newton} as
    \begin{equation}
        \label{eq:Pr-1}
        P_r(x)=1+\sum_{i=r+1}^{\ell}f_r[t_1,\dots,t_i]g_{i-1}(x),
    \end{equation}
    where $g_{r}(x)=\prod_{j=1}^{r}(x-t_j)$ and we used $f_r[t_1,\dots,t_i]=0$ for $i=1,\ldots,r$. In \prettyref{eq:Pr-1}, the absolute values of divided differences are obtained in \prettyref{lmm:divided}:
    \begin{equation}
        \label{eq:fr-div}
        |f_r[t_1,\dots,t_i]|\le \frac{\binom{i-2}{r-1}}{(t_{r+1}-t_{r})^{i-1}}.
    \end{equation}
    In the summation of \prettyref{eq:Pr-1}, let $g_{i-1}(x)=\sum_{j=0}^{i-1} a_jx^j$. Since $|t_j|\le 1$ for every $j$, we have $\sum_{j=0}^{i-1}|a_j|\le 2^{i-1}$ (see \prettyref{lmm:Px-expand}). Applying \prettyref{eq:momentdiff} yields that
    \begin{equation}
        \label{eq:polydiff}
        |\Expect_\nu[g_{i-1}]-\Expect_{\nu'}[g_{i-1}]|\le \sum_{j=1}^{i-1} |a_j|\delta\le 2^{i-1}\delta.
    \end{equation}
    Then we obtain from \prettyref{eq:CDF-poly} and \prettyref{eq:Pr-1} that
    \begin{equation}
        \label{eq:CDFdiff}
        |F_{\nu}(t_r)-F_{\nu'}(t_r)|\le \sum_{i=r+1}^{\ell}\frac{\binom{i-2}{r-1}2^{i-1}\delta}{(t_{r+1}-t_{r})^{i-1}}
        \le \frac{\ell 4^{\ell-1}\delta}{(t_{r+1}-t_r)^{\ell-1}}.
    \end{equation}
    Also, $|F_{\nu}(t_r)-F_{\nu'}(t_r)|\le 1$ trivially. Therefore,
    \begin{equation}
        \label{eq:W1-example}
        W_1(\nu,\nu')\le \sum_{r=1}^{\ell-1}\pth{\frac{\ell 4^{\ell-1}\delta}{(t_{r+1}-t_r)^{\ell-1}}\wedge 1}\cdot|t_{r+1}-t_r|\le 4e(\ell-1)\delta^{\frac{1}{\ell-1}},
    \end{equation}
    where we used $\max\{\frac{\alpha}{x^{\ell-2}}\wedge x:x>0\}=\alpha^{\frac{1}{\ell-1}}$ and $x^{\frac{1}{x-1}}\le e$ for $x\ge 1$.
\end{proof}

\begin{proof}[Second proof of \prettyref{prop:stable1-union}]
    Suppose on the contrary that
    \begin{equation}
        W_1(\nu,\nu')\ge C\ell\delta^{\frac{1}{\ell-1}},
    \end{equation}
    for some absolute constant $C$. We will show that $\max_{i\in[\ell-1]}|m_i(\nu)-m_i(\nu')|\ge \delta$. We continue to use $S$ in \prettyref{eq:supports-1} to denote the support of $\nu$ and $\nu'$. Let $\Delta F(t)=F_{\nu}(t)-F_{\nu'}(t)$ denote the difference between two CDFs. Using \prettyref{eq:W1-CDF-discrete}, there exists $r\in [\ell-1]$ such that
    \begin{equation}
        \label{eq:DeltaF-2}
        |\Delta F(t_r)|\cdot |t_{r+1}-t_r|\ge C\delta^{\frac{1}{\ell-1}}.
    \end{equation}
    We first construct a polynomial $L$ that preserves the sign of $\Delta F$. To this end, let $S'=\{s_1,\dots,s_m\}\subseteq S$ such that $t_1=s_1<s_2<\dots<s_m=t_{\ell}$ be the set of points where $\Delta F$ changes sign, \ie, $\Delta F(x)\Delta F(y)\le 0$ for every $x\in(s_i,s_{i+1})$, $y\in(s_{i+1},s_{i+2})$, for every $i$. Let $L(x)\in \pm \prod_{i=2}^{m-1}(x-s_i)$ be a polynomial of degree at most $\ell-2$ that also changes sign on $S'$ such that
    \[
        \Delta F(x)L(x)\ge 0,\quad t_1\le x\le t_\ell.
    \]
    Consider the integral of the above positive function. Applying integral by parts, and using $\Delta F(t_{\ell})=\Delta F(t_{1})=0$ yields that
    \begin{equation}
        \label{eq:int-EP}
        \int_{t_1}^{t_\ell} \Delta F(x)L(x)\diff x= -\int_{t_1}^{t_\ell}P(x)\diff \Delta F(x)
        =\Expect_{\nu'}[P(X)]-\Expect_\nu [P(X)],
    \end{equation}
    where $P(x)$ is a polynomial of degree at most $\ell-1$ such that $P'(x)=L(x)$. If we write $L(x)=\sum_{j=0}^{\ell-2}a_jx^j$, then $P(x)=\sum_{j=0}^{\ell-2}\frac{a_j}{j+1}x^{j+1}$. Since $|s_j|\le 1$ for every $j$, we have $\sum_{j=0}^{\ell-2}|a_j|\le 2^{\ell-2}$ (see \prettyref{lmm:Px-expand}), and thus $\sum_{j=0}^{\ell-2}\frac{|a_j|}{j+1}\le 2^{\ell-2}$. Hence, 
    \begin{equation}
        \label{eq:EP-ub}
        |\Expect_{\nu'}[P(X)]-\Expect_\nu [P(X)]|\le 2^{\ell-2}\max_{i\in[\ell-1]}|m_i(\nu)-m_i(\nu')|.
    \end{equation}

    Since $\Delta F(x)L(x)$ is always non-negative, applying \prettyref{eq:DeltaF-2} to \prettyref{eq:int-EP} yields that
    \begin{equation}
        \label{eq:diff-EP}
        |\Expect_{\nu'}[P(X)]-\Expect_\nu [P(X)]|\ge \int_{t_r}^{t_{r+1}} |\Delta F(x)L(x)|\diff x
        \ge \frac{C\delta^{\frac{1}{\ell-1}}}{|t_{r+1}-t_r|}\int_{t_r}^{t_{r+1}}|L(x)|\diff x.
    \end{equation}
    Recall that $|L(x)|=\prod_{i=2}^{m-1}|x-s_i|$. Then for $x\in (t_r,t_{r+1})$, we have $|x-s_i|\ge x-t_r$ if $s_i\le t_r$, and $|x-s_i|\ge t_{r+1}-x$ if $s_i\ge t_{r+1}$. Hence,
    \[
        |L(x)|\ge (t_{r+1}-x)^a(x-t_r)^b,
    \]
    for some $a,b\in\naturals$ such that $a,b\ge 1$ and $a+b\le \ell-2$. The integral of the right-hand side of the above inequality can be expressed as (see \cite[6.2.1]{AS64})
    \[
        \int_{t_r}^{t_{r+1}}(t_{r+1}-x)^a(x-t_{r})^b\diff x=\frac{(t_{r+1}-t_r)^{a+b+1}}{(a+1)\binom{a+b+1}{b}}.
    \]
    Since $|t_{r+1}-t_r|\ge |\Delta F(t_r)|\cdot |t_{r+1}-t_r|\ge C\delta^{\frac{1}{\ell-1}}$ and $\binom{a+b+1}{b}\le 2^{a+b+1}$, and $a+b+1\le \ell-1$, we obtain from \prettyref{eq:diff-EP} that
    \begin{equation}
        \label{eq:EP-lb}
        |\Expect_{\nu'}[P(X)]-\Expect_\nu [P(X)]|\ge \delta\frac{(C/2)^{\ell-1}}{\ell}.
    \end{equation}
    We obtain from \prettyref{eq:EP-ub} and \prettyref{eq:EP-lb} that 
    \[
        \max_{i\in[\ell-1]}|m_i(\nu)-m_i(\nu')|\ge \delta\frac{(C/4)^{\ell-1}}{\ell}.\qedhere
    \]
\end{proof}

\begin{proof}[Third proof of \prettyref{prop:stable1-union}]
    We continue to use $S$ in \prettyref{eq:supports-1} to denote the support of $\nu$ and $\nu'$. For any 1-Lipschitz function $f$, $\Expect_\nu f$ and $\Expect_{\nu'}f$ only pertain to function values $f(t_1),\dots,f(t_{\ell})$, which can be interpolated by a polynomial of degree $\ell-1$. 
    However, the coefficients of the interpolating polynomial can be arbitrarily large.\footnote{
        For example, the polynomial to interpolate $f(-\epsilon)=f(\epsilon)=\epsilon, f(\epsilon)=0$ is $P(x)=x^2/\epsilon$.
    }
    To fix this issue, we slightly modify the function $f$ on $S$ to $\tilde f$, and then interpolate $\tilde f$ with bounded coefficients. In this way we have 
    \begin{equation*}
        |\Expect_\nu f-\Expect_{\nu'}f|
        \le 2 \max_{x\in \{t_1,\dots,t_{\ell}\}}|\tilde f(x)-f(x)| + |\Expect_\nu P-\Expect_{\nu'}P|.
    \end{equation*}
    To this end, we define the values of $\tilde f$ recursively by
    \begin{equation}
        \label{eq:tilde-f-def}
        \tilde f(t_1)=f(t_1),\quad \tilde f(t_i)=\tilde f(t_{i-1})+(f(t_i)-f(t_{i-1}))\indc{t_i-t_{i-1}>\tau},
    \end{equation}
    where $\tau\le 2$ is a parameter we will optimize later. From the above definition $|\tilde f(x)-f(x)|\le \tau\ell$ for $x\in S$. The interpolating polynomial $P$ can be expressed using Newton formula \prettyref{eq:interpolation-Newton} as
    \[
        P(x)=\sum_{i=1}^\ell \tilde f[t_1,\dots,t_i] g_{i-1}(x),
    \]
    where $g_{r}(x)=\prod_{j=1}^{r}(x-t_j)$ such that $|\Expect_\nu[g_{r}]-\Expect_{\nu'}[g_{r}]|\le 2^{r}\delta$ by \prettyref{eq:polydiff} for $r\le \ell-1$. Since $f$ is 1-Lipschitz, we have $|\tilde f[t_i,t_{i+1}]|\le 1$ for every $i$. Higher order divided differences are recursively evaluated by \prettyref{eq:div-recursion}. We now prove
    \begin{equation}
        \label{eq:tilde-f-div}
        \tilde f[t_i,\ldots,t_{i+j}]\le (2/\tau)^{j-1},~\forall~i,j.
    \end{equation}
    by induction on $j$. Assume \prettyref{eq:tilde-f-div} holds for every $i$ and some fixed $j$. The recursion \prettyref{eq:div-recursion} gives 
    \[
        \tilde f[t_i,\ldots,t_{i+j+1}]=\frac{\tilde f[t_{i+1},\ldots,t_{i+j+1}]-\tilde f[t_{i},\ldots,t_{i+j}]}{t_{i+j+1}-t_{i}}.
    \]
    If $t_{i+j+1}-t_{i}<\tau$, then $\tilde f[t_i,\ldots,t_{i+j+1}]=0$ by \prettyref{eq:tilde-f-def}; otherwise, $\tilde f[t_i,\ldots,t_{i+j+1}]\le (\frac{2}{\tau})^{j}$ by triangle inequality. Using \prettyref{eq:tilde-f-div}, we obtain that
    \begin{equation*}
        |\Expect_\nu f-\Expect_{\nu'}f|
        \le 2\tau\ell+\sum_{i=2}^\ell \pth{\frac{2}{\tau}}^{i-2}2^{i-1}\delta
        \le 2\ell\pth{\tau+\frac{4^{\ell-2}}{\tau^{\ell-2}}\delta}.
    \end{equation*}
    The conclusion follows by letting $\tau=4\delta^{\frac{1}{\ell-1}}$.
\end{proof}

The proof of \prettyref{prop:stable2} uses a similar idea as the first proof of \prettyref{prop:stable1-union} to approximate step functions for all values of $\nu$ and $\nu'$; however, this is clearly impossible for non-discrete $\nu'$. For this reason, we turn from interpolation to majorization. A classical method to bound a distribution function by moments is to construct two polynomials that majorizes and minorizes a step function, respectively. Then the expectations of these two polynomials provide a sandwich bound for the distribution function. This idea is used, for example, in the proof of Chebyshev-Markov-Stieltjes inequality (cf.~\cite[Theorem 2.5.4]{Akhiezer1965}).
\begin{proof}[Proof of \prettyref{prop:stable2}]
    Suppose $\nu$ is supported on $x_1<x_2<\ldots<x_k$. Fix $t\in\reals$ and let $f_t(x)=\indc{x\le t}$. Suppose $x_m<t<x_{m+1}$. Similar to \prettyref{ex:Hermite}, we construct polynomial majorant and minorant using Hermite interpolation. To this end, let $P_t$ and $Q_t$ be the unique degree-$2k$ polynomials to interpolate $f_t$ with the following:
    \begin{center}
        \begin{tabular}{ c | c  c  c | c | c c c }
        \toprule
                &$x_1$  &\dots  &$x_m$  &$t$      &$x_{m+1}$  &\dots  &$x_k$ \\
        \hline
        $P$     &1     &\dots   &1      &1      &0          &\dots  &0 \\
        $P'$    &0     &\dots   &0      &any    &0          &\dots  &0 \\
        \hline
        $Q$     &1     &\dots   &1      &0      &0          &\dots  &0 \\
        $Q'$    &0     &\dots   &0      &any    &0          &\dots  &0 \\
        \bottomrule
        \end{tabular}
    \end{center}
    \vspace{0.5em}
    As a consequence of Rolle's theorem, $P_t\ge f_t\ge Q_t$ (cf.~\cite[p.~65]{Akhiezer1965}, and an illustration in \prettyref{fig:major-minor}):
    \begin{figure}[ht]
        \centering
        \includegraphics[width=.5\linewidth]{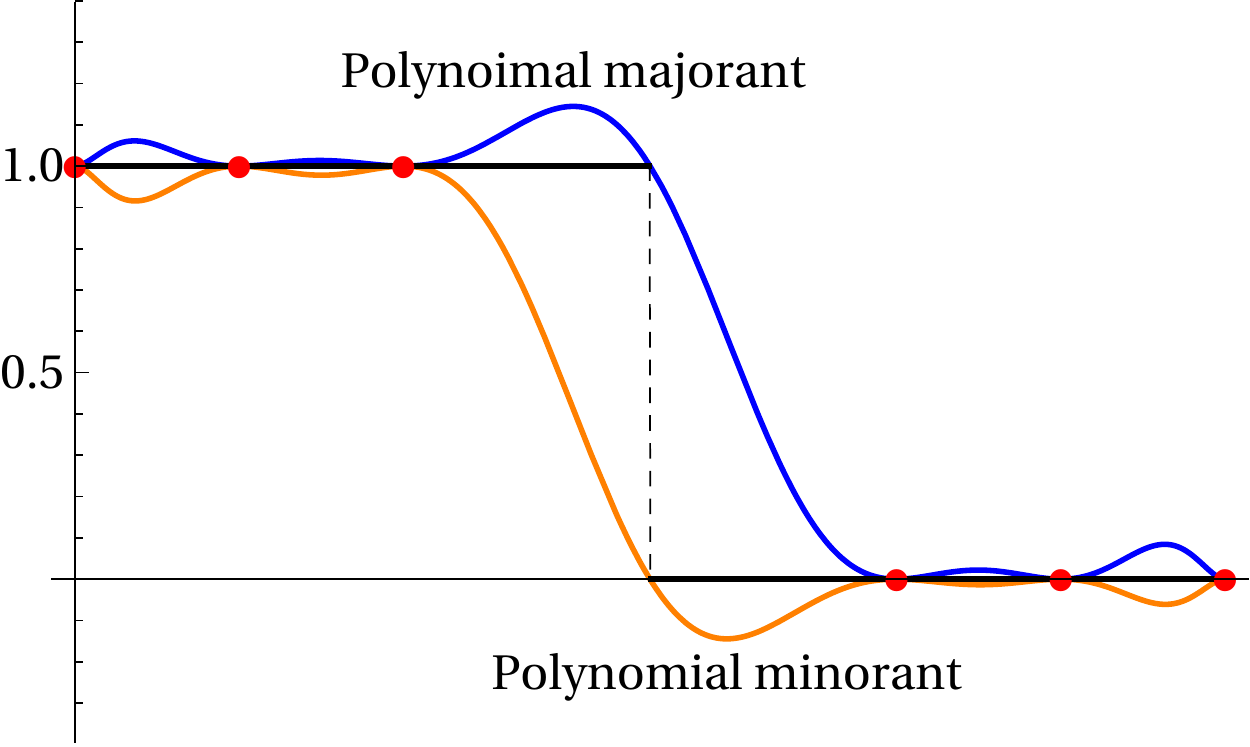}
        \caption{Polynomial majorant $P_t$ and minorant $Q_t$ that coincide with the step function on 6 red points.
            The polynomials are of degree 12, obtained by Hermite interpolation in \prettyref{sec:poly}.
        }
        \label{fig:major-minor}
    \end{figure}
    Using Lagrange formula of Hermite interpolation \cite[pp.~52--53]{stoer.2002}, $P_t$ and $Q_t$ differ by
    \[
        P_t(x)-Q_t(x)=R_t(x)\triangleq\prod_i\pth{\frac{x-x_i}{t-x_i}}^2.
    \]
    The sandwich bound for $f_t$ yields a sandwich bound for the CDFs:
    \begin{align*}
        & \Expect_{\nu'}[Q_t]\le F_{\nu'}(t) \le \Expect_{\nu'}[P_t]=\Expect_{\nu'}[Q_t]+\Expect_{\nu'}[R_t],\\
        & \Expect_{\nu}[Q_t]\le F_{\nu}(t) \le \Expect_{\nu}[P_t]=\Expect_{\nu}[Q_t].
    \end{align*}
    Then the CDFs differ by 
    \begin{gather}
        |F_{\nu}(t)-F_{\nu'}(t)|\le (f(t)+g(t))\wedge 1\le f(t)\wedge 1+g(t)\wedge 1,\label{eq:CDF-diff-fg}\\
        f(t)\triangleq |\Expect_{\nu'}[Q_t]-\Expect_{\nu}[Q_t]|, \quad g(t)\triangleq\Expect_{\nu'}[R_t].\nonumber
    \end{gather}
    The conclusion will be obtained from the integral of CDF difference using \prettyref{eq:W1-CDF}. Since $R_t$ is almost surely zero under $\nu$, we also have $g(t)=|\Expect_{\nu'}[R_t]-\Expect_{\nu}[R_t]|$. Similar to \prettyref{eq:polydiff}, we obtain that 
    \[
        g(t)=|\Expect_{\nu'}[R_t]-\Expect_{\nu}[R_t]|\le \frac{2^{2k}\delta}{\prod_{i=1}^k (t-x_i)^2}.
    \]
    Hence, 
    \begin{equation}
        \label{eq:int-g}
        \int (g(t)\wedge 1) \diff t \le \int \pth{\frac{2^{2k}\delta}{\prod_{i=1}^k (t-x_i)^2}\wedge 1} \diff t \le 16k\delta^{\frac{1}{2k}},
    \end{equation}
    where the last inequality is proved in \prettyref{lmm:wedge1-integral}.

    Next we analyze $f(t)$. The polynomial $Q_t$ (and also $P_t$) can be expressed using Newton formula \prettyref{eq:interpolation-Newton} as
    \begin{equation}
        \label{eq:Qt}
        Q_t(x)=1+\sum_{i=2m+1}^{2k+1}f_t[t_1,\dots,t_i]g_{i-1}(x),
    \end{equation}
    where $t_1,\ldots,t_{2k+1}$ denotes the expanded sequence
    \[
        x_1,x_1,\ldots,x_m,x_m,t,x_{m+1},x_{m+1},\ldots,x_{k},x_{k}
    \]
    obtained by \prettyref{eq:new-sequence}, $g_{r}(x)=\prod_{j=1}^{r}(x-t_j)$, and we used $f_t[t_1,\dots,t_i]=0$ for $i=1,\ldots, 2m$. In \prettyref{eq:Qt}, the absolute values of divided differences are obtained in \prettyref{lmm:divided}: 
    \[
        f_t[t_1,\dots,t_i]\le \frac{\binom{i-2}{2m-1}}{(t-x_m)^{i-1}}.
    \]
    Using \prettyref{eq:Qt}, and applying the upper bound for $|\Expect_{\nu}[g_{i-1}]-\Expect_{\nu'}[g_{i-1}]|$ in \prettyref{eq:polydiff}, we obtain that, 
    \[
        f(t)=|\Expect_{\nu'}[Q_t]-\Expect_{\nu}[Q_t]|
        \le \sum_{i=2m+1}^{2k+1}\frac{\binom{i-2}{2m-1}2^{i-1}\delta}{(t-x_m)^{i-1}}
        \le \frac{k4^{2k}\delta}{(t-x_m)^{2k}},\quad x_m<t<x_{m+1},~m\ge 1.
    \]
    If $t<x_1$, then $Q_t=0$ and thus $f(t)=0$. Then, analogous to \prettyref{eq:int-g}, we obtain that
    \begin{equation}
        \label{eq:int-f}
        \int (f(t)\wedge 1) \diff t \le 16k\delta^{\frac{1}{2k}}.
    \end{equation}
    Using \prettyref{eq:int-g} and \prettyref{eq:int-f}, the conclusion follows by applying \prettyref{eq:CDF-diff-fg} to the integral representation of Wasserstein distance \prettyref{eq:W1-CDF}.
\end{proof}

\subsection{Proofs of density estimation}
\begin{lemma}[Bound $\chi^2$-divergence using moments difference]
    \label{lmm:chi2-moments}
    Suppose all moments of $\nu$ and $\nu'$ exist, and $\nu'$ is centered with variance $\sigma^2$.
    Then,
    \[
        \chi^2(\nu*N(0,1)\|\nu'*N(0,1))\le e^{\frac{\sigma^2}{2}}\sum_{j\ge 1}\frac{(\Delta m_j)^2}{j!},
    \]
    where $\Delta m_j =  m_j(\nu)-m_j(\nu')$ denotes the $j^{\rm th}$ moment difference.
\end{lemma}
\begin{proof}
    The densities of two mixture distributions $\nu*N(0,1)$ and $\nu'*N(0,1)$ are 
    \begin{gather*}
        f(x)=\int \phi(x-u)\diff \nu(u)=\phi(x)\sum_{j\ge 1}H_j(x)\frac{m_j(\nu)}{j!},\\
        g(x)=\int \phi(x-u)\diff \nu'(u)= \phi(x)\sum_{j\ge 1}H_j(x)\frac{m_j(\nu')}{j!},
    \end{gather*}
    respectively, where $\phi$ denotes the density of $N(0,1)$, and we used $\phi(x-u)=\phi(x)\sum_{j\ge 0}H_j(x)\frac{x^j}{j!}$ (see the exponential generating function of Hermite polynomials \cite[22.9.17]{AS64}). Since $x\mapsto e^x$ is convex, applying Jensen's inequality yields that
    \[
        g(x)=\phi(x)\Expect[\exp(U'x-U'^2/2)]
        \ge \phi(x)\exp(-\sigma^2/2).
    \]
    Consequently, 
    \begin{align*}
        &\phantom{{}={}}\chi^2(\nu*N(0,1)\|\nu'*N(0,1)) = \int \frac{(f(x)-g(x))^2}{g(x)}\diff x\\
        &\leq e^{\frac{\sigma^2}{2}} \expect{\pth{\sum_{j\ge 1}H_j(Z)\frac{\Delta m_j}{j!}}^2}       
        =  e^{\frac{\sigma^2}{2}} \sum_{j\ge 1}\frac{(\Delta m_j)^2}{j!},
    \end{align*}
    where $Z\sim N(0,1)$ and the last step follows from the orthogonality of Hermite polynomials:
    \begin{equation}
        \label{eq:Hermite-ortho}
        \Expect[H_i(Z)H_j(Z)] = j! \indc{i=j}.\qedhere
    \end{equation}
\end{proof}


\begin{lemma}
    \label{lmm:high-moments}
    If $U$ and $U'$ each takes at most $k$ values in $[-1,1]$, and $|\Expect[U^j]-\Expect[U'^j]|\le \epsilon$ for $j=1,\dots,2k-1$, then, for any $\ell\ge 2k$,
    \[
        |\Expect[U^\ell]-\Expect[U'^\ell]|\le 3^\ell \epsilon.
    \]
\end{lemma}
\begin{proof}
    Let $f(x)=x^\ell$ and denote the atoms of $U$ and $U'$ by $x_1<\dots<x_{k'}$ for some $k'\le 2k$. The function $f$ can be interpolated on $x_1,\dots,x_{k'}$ using a polynomial $P$ of degree at most $2k-1$, which, in the Newton form \prettyref{eq:interpolation-Newton}, is
    \[
        P(x)=\sum_{i=1}^{k'}f[x_1,\dots,x_i]g_{i-1}(x)
        =\sum_{i=1}^{k'}\frac{f^{(i-1)}(\xi_i)}{(i-1)!}g_{i-1}(x),
    \]
    for some $\xi_i\in[x_1,x_i]$, where $g_{r}(x)=\prod_{j=1}^{r}(x-x_j)$ and we used the intermediate value theorem for the divided differences (see \cite[(2.1.4.3)]{stoer.2002}).  Note that for any $\xi\in[-1,1]$, $|f^{(i-1)}(\xi)|\le \frac{\ell!}{(\ell-1+i)!}$. Similar to \prettyref{eq:polydiff}, we obtain that
    \[
        |\Expect[U^\ell]-\Expect[U'^\ell]|
        =|\Expect[P(U)]-\Expect[P(U')]|
        \le \sum_{i=1}^{k'}\binom{\ell}{i-1}2^{i-1}\epsilon
        \le 3^\ell \epsilon.\qedhere
    \]
\end{proof}

\begin{proof}[Proof of \prettyref{thm:main-density}]
    Here we prove a stronger result that 
    \[
        \chi^2(\hat f\|f) + \chi^2(f\|\hat f)\le O_k(\log(1/\delta)/n).
    \]
    By scaling it suffices to consider $M=1$. Similar to \prettyref{eq:bfm-example} and \prettyref{eq:Hoeffding-example}, we obtain an estimated mixing distribution $\hat \nu$ supported on $k$ atoms in $[-1,1]$ such that, with probability $1-\delta$,
    \[
        \Norm{{\bf m}_{2k-1}(\hat \nu)-{\bf m}_{2k-1}(\nu)}_2\le \sqrt{c_k\log(1/\delta)/n},
    \]
    for some constant $c_k$ that depends on $k$. The conclusion follows from Lemmas \ref{lmm:chi2-moments} and \ref{lmm:high-moments}.
\end{proof}

\begin{proof}[Proof of \prettyref{thm:oracle}]
    Recall that $f$ is $1$-subgaussian and $\sigma$ is a fixed constant. Similar to \prettyref{eq:Hoeffding-example}, we obtain an estimate $\tilde m_r$ for $\Expect_f[\gamma_r(X,\sigma)]$ (see the definition of $\gamma_r(\cdot,\sigma)$ in \prettyref{eq:Hermite}) for $r=1,\ldots,2k-1$ such that, with probability $1-\delta$,
    \[
        |\tilde m_r-\Expect_f [\gamma_r(X,\sigma)]|\le \sqrt{c_k\log(1/\delta)/n},
    \]
    for some constant $c_k$ that depends on $k$. By assumption, $\TV(f,g)\le \epsilon$ where both $f$ and $g$ are 1-subgaussian. Let $g=\nu*N(0,\sigma^2)$. Then, using \prettyref{lmm:subg-moments} below and the triangle inequality, we have
    \[
        |\tilde m_r-m_r(\nu)|\le O_k\pth{\epsilon\sqrt{\log(1/\epsilon)}+\sqrt{\log(1/\delta)/n}},\quad r=1,\dots,2k-1.
    \]
    Let $\hat\nu$ be obtained from the projection \prettyref{eq:project}. Similar to \prettyref{eq:bfm-example}, we have such that
    \[
        \Norm{{\bf m}_{2k-1}(\hat\nu)-{\bf m}_{2k-1}(\nu)}_2\le O_k\pth{\epsilon\sqrt{\log(1/\epsilon)}+\sqrt{\log(1/\delta)/n}}.
    \]
    Let $\hat f= \hat \nu*N(0,\sigma^2)$. Using the moment comparison in Lemmas \ref{lmm:chi2-moments} and \ref{lmm:high-moments}, and applying the upper bound $\TV(\hat f,g)\le \sqrt{\chi^2(\hat f\|g)/2}$, we obtain that
    \[
        \TV(\hat f,g)\le O_k\pth{\epsilon\sqrt{\log(1/\epsilon)}+\sqrt{\log(1/\delta)/n}}.
    \]
    The conclusion follows from the triangle inequality.
\end{proof}
\begin{lemma}
    \label{lmm:subg-moments}
    Let $\sigma$ be a constant. If $f$ and $g$ are 1-subgaussian, and $\TV(f,g)\le \epsilon$, then, 
    \[
        |\Expect_f[\gamma_r(X,\sigma)]-\Expect_g[\gamma_r(X,\sigma)]|\le O_r(\epsilon\sqrt{\log(1/\epsilon)}).
    \]
\end{lemma}
\begin{proof}
    The total variation distance has the following variational representation:
    \begin{equation}
        \label{eq:TV-sup}
        \TV(f,g)=\frac{1}{2}\sup_{\|h\|_\infty \le 1}|\Expect_fh-\Expect_gh|.
    \end{equation}
    Here the function $\gamma_r(\cdot,\sigma)$ is a polynomial and unbounded, so the above representation cannot be directly applied. Instead, we apply a truncation argument, thanks to the subgaussianity of $f$ and $g$, and obtain that, for both $X\sim f$ and $g$ (see Lemmas \ref{lmm:hermite-ub} and \ref{lmm:normal-tail-moments}),
    \[
        \Expect[\gamma_r(X,\sigma)\indc{|X|\ge \alpha}]
        \le (O(\sqrt{r}))^r\Expect|X^r\indc{|X|\ge \alpha}|
        \le (O(\alpha\sqrt{r}))^re^{-\alpha^2/2}.
    \]
    Note that by definition \prettyref{eq:Hermite}, $\gamma_r(x,\sigma)$ on $|x|\le \alpha$ is at most $(O(\alpha\sqrt{r}))^r$.  Applying \prettyref{eq:TV-sup} yields that, for $h(x)=\gamma_r(x,\sigma)\indc{|x|\le \alpha}$,
    \[
        |\Expect_fh-\Expect_gh|\le \epsilon(O(\alpha\sqrt{r}))^r.
    \]
    The conclusion follows by choosing $\alpha=O_r(\sqrt{\log(1/\epsilon)})$ and using the triangle inequality.
\end{proof}

\subsection{Proofs for \prettyref{sec:known}}
\label{sec:pf-known}

\begin{proof}[Proof of \prettyref{lmm:var-tildem}]
    Note that $\tilde m_r=\frac{1}{n}\sum_{i=1}^{n}\gamma_r(X_i,\sigma)$. Then we have 
    \[
        \var[\tilde m_r]=\frac{1}{n}\var[\gamma_r(X,\sigma)],
    \]
    where $X\sim \nu*N(0,\sigma^2)$. Since the standard deviation of a summation is at most the sum of individual standard deviations, using \prettyref{eq:Hermite}, we have 
    \[
        \sqrt{\var[\gamma_r(X,\sigma)]}
        \le r!\sum_{j=0}^{\floor{r/2}} \frac{(1/2)^j }{j!(r-2j)!}\sigma^{2j}\sqrt{\var[X^{r-2j}]}.
    \]
    $X$ can be viewed as $U+\sigma Z$ where $U\sim\nu$ and $Z\sim N(0,1)$ independent of $U$. Since $\nu$ is supported on $[-M,M]$, for any $\ell\in\naturals$, we have 
    \[
        \var[X^\ell]\le \Expect[X^{2\ell}]\le 2^{2\ell-1}(M^{2\ell}+\Expect|\sigma Z|^{2\ell})
        \le ((2M)^{\ell}+\Expect|3\sigma Z|^\ell)^2,
    \]
    where in the last step we used the inequality $\Expect|Z|^{2\ell}\le 2^{\ell}(\Expect|Z|^\ell)^2$ (see \prettyref{lmm:m-normal} below). Therefore,
    \begin{align*}
        \sqrt{\var[\gamma_r(X,\sigma)]}
        &\le r!\sum_{j=0}^{\floor{r/2}} \frac{(1/2)^j }{j!(r-2j)!}\sigma^{2j}((2M)^{r-2j}+\Expect|3\sigma Z|^{r-2j})\\
        &=\Expect(2M+\sigma Z')^r+\Expect(3\sigma |Z|+\sigma Z')^r,
    \end{align*}
    where $Z'\sim N(0,1)$ independent of $Z$. The conclusion follows by the moments of the standard normal distribution (see \cite{subgaussian}).
\end{proof}

\begin{lemma}
    \label{lmm:m-normal}
    Let $Z\sim N(0,1)$. For $\ell\in\naturals$, we have
    \[
        \sqrt{\frac{\pi}{8}}\le \frac{\Expect|Z|^{2\ell}}{2^\ell(\Expect|Z|^{\ell})^2}\le \sqrt{\frac{2}{\pi}}.
    \]
\end{lemma}
\begin{proof}
    Direct calculations lead to (see \cite[3.461.2--3]{GR}):
    \[
        \frac{\Expect|Z|^{2\ell}}{2^\ell (\Expect|Z|^{\ell})^2}
        = \begin{cases}
            \frac{\binom{2\ell}{\ell}}{\binom{\ell}{\ell/2}2^\ell},&\ell~\mathrm{even},\\
            \frac{\pi\ell}{8^\ell}\binom{2\ell}{\ell}\binom{\ell-1}{\frac{\ell-1}{2}},&\ell~\mathrm{odd}.
        \end{cases}
    \]
    Using $\frac{2^n}{\sqrt{2n}}\le \binom{n}{n/2}\le 2^n\sqrt{\frac{2}{\pi n}}$ \cite[Lemma 4.7.1]{ash-itbook}, we obtain that
    \begin{gather*}
        \sqrt{\frac{\pi}{8}}\le \frac{\binom{2\ell}{\ell}}{\binom{\ell}{\ell/2}2^\ell}\le \sqrt{\frac{2}{\pi}},\\
        \frac{\pi}{4}\sqrt{\frac{\ell}{2(\ell-1)}}\le \frac{\pi\ell}{8^\ell}\binom{2\ell}{\ell}\binom{\ell-1}{\frac{\ell-1}{2}}\le \sqrt{\frac{\ell}{2(\ell-1)}},
    \end{gather*}
    which prove this lemma for $\ell\ge 5$. For $\ell\le 4$ the lemma follows from the above equalities. 
\end{proof}

\subsection{Proofs for \prettyref{sec:unknown}}
\label{sec:pf-unknown}
\begin{proof}[Proof of \prettyref{prop:accuracy2}]
    By scaling it suffices to consider $M=1$. Without loss of generality assume $ \sigma\ge \hat\sigma $ and otherwise we can interchange $\pi$ and $\hat\pi$. Let $\tau^2=\sigma^2-\hat\sigma^2$ and $\nu'=\nu*N(0,\tau^2)$. Similar to \prettyref{eq:mr-deconv-example}, we obtain that
    \begin{equation}
        \label{eq:mr-deconv}
        |m_r(\nu')-m_r(\hat \nu)|\le (c\sqrt{k})^{2k}\epsilon,\quad r=1,\dots,2k,
    \end{equation}
    for some absolute constant $c$. Using \prettyref{lmm:tau} below yields that $\tau\le O(\epsilon^{\frac{1}{2k}})$. It follows from \prettyref{prop:stable2} that
    \[
        W_1(\nu',\hat\nu)\le O\pth{k^{1.5}\epsilon^{\frac{1}{2k}}}.
    \]
    The conclusion follows from $W_1(\nu',\nu)\le O(\tau)$ and the triangle inequality.
\end{proof}
\begin{lemma}
    \label{lmm:tau}
    Suppose $\pi=\nu*N(0,\tau^2)$ and $\pi'$ is $k$-atomic supported on $[-1,1]$. Let $\epsilon=\max_{i\in [2k]}|m_i(\pi)-m_i(\pi')|$. Then,
    \[
        \tau\le 2\pth{\epsilon/k!}^{\frac{1}{2k}}.
    \]
\end{lemma}
\begin{proof}
    Denote the support of $\pi'$ by $\{x_1',\dots,x_k'\}$. Consider the polynomial $P(x)=\prod_{i=1}^k (x-x_i')^2=\sum_{i=0}^{2k} a_i x^i$ which is almost surely zero under $\pi'$. Since every $|x_i'|\le 1$, similar to \prettyref{eq:polydiff}, we obtain that 
    \[
        \Expect_\pi[P]=|\Expect_\pi[P]-\Expect_{\pi'}[P]|\le 2^{2k} \epsilon.
    \]
    Since $\pi=\nu*N(0,\tau^2)$, we have
    \[
        \Expect_\pi[P]\ge \min_{x}\Expect[P(x+\tau Z)]\ge \tau^{2k}\min_{y_1,\dots,y_{k}}\Expect\qth{\prod_i(Z+y_i)^2}=k!\tau^{2k},
    \]
    where $Z\sim N(0,1)$, and in the last step we used \prettyref{lmm:monic} below.
\end{proof}
\begin{lemma}
    Let $Z\sim N(0,1)$. Then,
    \label{lmm:monic}       
    \[
        \min\{\Expect[p^2(Z)]:   \deg(p) \leq k,~p \text{ is  monic} \} = k!
    \]
    achieved by $p=H_k$.
\end{lemma}
\begin{proof}
    Since $p$ is monic, it can be written as $p=H_k + \sum_{j=0}^{k-1} \alpha_j H_j$, where $H_j$ is the Hermite polynomial \prettyref{eq:Hermite1}. By the orthogonality \prettyref{eq:Hermite-ortho}, we have $\Expect[p^2(Z)]=k!+\sum_{j=0}^{k-1} \alpha_j^2 j!$ and the conclusion follows.
\end{proof}

\begin{proof}[Proof of \prettyref{lmm:sigma-exist}]
    The proof is similar to that of \cite[Theorem 5B]{Lindsay1989}. Let $ \hat{\bfM}_r(\sigma) $ denote the moment matrix associated with the empirical moments of $ \gamma_i(X,\sigma) $ for $ i\le 2r $; in other words, $(\hat{\bfM}_r(\sigma))_{ij} = \Expect_n[\gamma_{i+j}(X,\sigma)]$, $i,j =0,\ldots,r$.		Let
    \begin{equation}
        \label{eq:hat-sigma}
        \hat{\sigma}_r=\inf\{\sigma>0:\det(\hat\bfM_r(\sigma))=0\}.
    \end{equation}
    The smallest positive zero of $\hat d_k$ is given by $\hat\sigma_k$. Direct calculation shows that $\hat{\sigma}_1=s$. Since the mixture distribution has a density, then almost surely, the empirical distribution has $ n $ points of support. By \prettyref{thm:supp-detM}, the matrix $ \hat{\bfM}_r(0) $ is positive definite and thus $ \hat{\sigma}_r>0 $ for any $ r<n $. For any $ q<r $, if $ \hat{\bfM}_r(\sigma) $ is positive definite, then $ \hat{\bfM}_q(\sigma) $ as a leading principal submatrix is also positive definite. Since eigenvalues of $ \hat{\bfM}_r(\sigma) $ are continuous functions of $ \sigma $, we have $\hat\sigma_r>\sigma\Rightarrow\hat\sigma_q>\sigma$, and thus
    \begin{equation}
        \label{eq:hat-sigma-cmp}
        \hat{\sigma}_q\ge \hat{\sigma}_r,\quad \forall~q<r.
    \end{equation}
    In particular, $ \hat{\sigma}_k\le \hat{\sigma}_1 $.
\end{proof}
\begin{proof}[Proof of \prettyref{lmm:U-exist}]
    We continue to use the notation in \prettyref{eq:hat-sigma}. Applying \prettyref{eq:hat-sigma-cmp} and \prettyref{lmm:sigma-exist} yields that
    \[
        0<\hat{\sigma}=\hat{\sigma}_k\le \hat{\sigma}_{k-1}\le ...\le \hat{\sigma}_1=s,
    \]
    and for any $ \sigma<\hat{\sigma}_j $, the matrix $ \hat{\bfM}_j(\sigma) $ is positive definite. Since $ \det(\hat{\bfM}_k(\hat{\sigma}))=0 $, then, for some $ r\in \{1,\dots,k\} $, we have $ \det(\hat{\bfM}_j(\hat{\sigma}))=0 $ for $ j=r,\dots,k $, and $ \det(\hat{\bfM}_j(\hat{\sigma}))>0 $ for $ j=0,\dots,r-1 $. By \prettyref{thm:supp-detM}, there exist a $ r $-atomic distribution whose $ j^{\rm th} $ moment coincides with $ \hat{\gamma}_j(\hat{\sigma}) $ for $ j\le 2r $. It suffices to show that $ r=k $ almost surely.

    Since the mixture distribution has a density, in the following we condition on the event that all samples $ X_1,\dots,X_n $ are distinct, which happens almost surely, without loss of generality. We first show that the empirical moments $ (\hat{\gamma}_1,\dots,\hat{\gamma}_n) $, where $ \hat{\gamma}_j=\frac{1}{n}\sum_i X_i^j $, have a joint density in $ \reals^n $. The Jacobian matrix of this transformation is
    \begin{equation*}
        \frac{1}{n}
        \begin{bmatrix}
            1 & &  &\\
            & 2 &  &\\
            &  &\ddots&   \\
            &  &  &n 
        \end{bmatrix}
        \begin{bmatrix}
            1 & \cdots & 1 \\
            X_1 & \cdots & X_{n} \\
            \vdots  & \ddots & \vdots  \\
            X_1^{n-1} & \cdots & X_{n}^{n-1} 
        \end{bmatrix},
    \end{equation*}
    which is invertible. Since those $ n $ samples $ (X_1,\dots,X_n) $ have a joint density, then the empirical moments $ (\hat{\gamma}_1,\dots,\hat{\gamma}_n) $ also have a joint density.

    Suppose, for the sake of contradiction, that $ r\le k-1 $. Then $ \det(\hat{\bfM}_{r-1}(\hat{\sigma}))>0 $ and $ \det(\hat{\bfM}_{r}(\hat{\sigma}))=\det(\hat{\bfM}_{r+1}(\hat{\sigma}))=0 $. In this case, $ \hat{m}_{2r+1}(\hat{\sigma}) $ is a deterministic function of $ \hat{m}_{1}(\hat{\sigma}),\dots,\hat{m}_{2r}(\hat{\sigma})$ (see \prettyref{lmm:hankel-det}). Since $ \hat{\sigma} $ is the smallest positive root of $ \hat{d}_r(\sigma)=0 $, it is uniquely determined by $ (\hat{\gamma}_1,\dots,\hat{\gamma}_{2r}) $.
    Therefore, $ \hat{m}_{2r+1}(\hat{\sigma}) $, and thus $ \hat{\gamma}_{2r+1} $, are both deterministic functions of $ (\hat{\gamma}_1,\dots,\hat{\gamma}_{2r}) $, which happens with probability zero, since the sequence $ (\hat{\gamma}_1,\dots,\hat{\gamma}_{2r+1}) $ has a joint density.
    Consequently, $ r\le k-1 $ with probability zero.
\end{proof}

The proof of \prettyref{eq:hatnu-tail} relies on the following result, which obtains a tail probability bound by comparing moments.
\begin{lemma}
    \label{lmm:tail}
    Let $\epsilon=\max_{i\in [2k]}|m_i(\nu)-m_i(\nu')|$. If either $\nu$ or $\nu'$ is $k$-atomic, and $\nu$ is supported on $[-1,1]$, then, for any $t>1$,
    \[
        \Prob[|Y|\ge t]\le 2^{2k+1}\epsilon (t-1)^{-2k},\quad Y\sim\nu'.
    \]
\end{lemma}
\begin{proof}
    We only show the upper tail bound $ \Prob[Y\ge t] $. The lower tail bound of $ Y $ is equal to the upper tail bound of $ -Y $. 
    \begin{itemize}
        \item Suppose $\nu$ is $k$-atomic supported on $\{x_1,\ldots,x_k\}$. Consider a polynomial $ P(x)=\prod_i (x-x_i)^2 $ of degree $2k$ that is almost surely zero under $\nu$. Since every $|x_i|\le 1$, similar to \prettyref{eq:polydiff}, we obtain that 
        \[
            \Expect_{\nu'}[P]=|\Expect_\nu[P]-\Expect_{\nu'}[P]|\le 2^{2k} \epsilon.
        \]
        Using Markov inequality, for any $t>1$, we have
        \[
            \Prob[Y\ge t]\le \Prob[P(Y)\ge P(t)] \le \frac{\Expect[P(Y)]}{P(t)} \le \frac{2^{2k} \epsilon}{(t-1)^{2k}}.
        \]
        \item Suppose $\nu'$ is $k$-atomic supported on $\{x_1,\ldots,x_k\}$. If those values are all within $[-1,1]$, then we are done. If there are at most $ k-1 $ values, denoted by $ \{x_1,\dots,x_{k-1}\} $, are within $ [-1,1] $, then we consider a polynomial $ P(x)=(x^2-1)\prod_i (x-x_i)^2 $ of degree $2k$ that is almost surely non-positive under $\nu$. Similar to \prettyref{eq:polydiff}, we obtain that
        \[
            \Expect_{\nu'}[P]\le \Expect_{\nu'}[P]-\Expect_{\nu}[P]\le 2^{2k} \epsilon.
        \]
        Since $P\ge 0$ almost surely under $\nu'$, the conclusion follows follows analogously using Markov inequality. \qedhere
    \end{itemize}
\end{proof}
\begin{lemma}
    \label{lmm:tail-hatU}
    Let
    \[
        \pi=\nu*N(0,\tau^2),\quad \hat\pi= \hat\nu,
    \]
    where $\nu$ and $\nu$ are both $k$-atomic, $\nu$ is supported on $[-1,1]$, and $\tau\le 1$. If $|m_i(\pi)-m_i(\hat{\pi})|\le \epsilon$ for $i\le 2k$, then, for any $t\ge \sqrt{18k}$,
    \begin{equation*}
        \Prob[|\hat U|\ge t]\le 
				2^{2k+1}\epsilon
				\pth{\frac{t}{\sqrt{18k}}-1}^{-2k},\quad \hat U\sim\hat \nu.
    \end{equation*}
\end{lemma}
\begin{proof}
    Let $g$ be the $(k+1)$-point Gauss quadrature of the standard normal distribution. Furthermore, $g$ is supported on $[-\sqrt{4k+6},\sqrt{4k+6}]$ for some absolute constant $c$ (see the bound on the zeros of Hermite polynomials in \cite[p.~129]{orthogonal.poly}). Let $G\sim g$, $U\sim \nu$, and $\hat U\sim\hat\nu$. Denote the maximum absolute value of $U+\tau G$ by $M$ which is at most $1+\sqrt{4k+6}\le \sqrt{18k}$ for $k\ge 1$. Applying \prettyref{lmm:tail} to the distributions of $\frac{U+\tau G}{\sqrt{18k}}$ and $\frac{\hat U}{\sqrt{18k}}$ yields the desired conclusion.
\end{proof}

\subsection{Proofs for \prettyref{sec:adaptive}}
\label{sec:pf-adaptive}


\begin{proof}[Proof of \prettyref{prop:stable1-separation}]
    The proof is analogous to the first proof of \prettyref{prop:stable1-union}, apart from a more careful analysis of polynomial coefficients. When each atom is at least $\gamma$ away from all but at most $\ell'$ other atoms, the left-hand side of \prettyref{eq:CDFdiff} is upper bounded by
    \[
        |F_{\nu}(t_r)-F_{\nu'}(t_r)|\le \frac{\ell 4^{\ell-1}\delta}{(t_{r+1}-t_r)^{\ell'}\gamma^{\ell-\ell'-1}},
    \]
    The remaining proof is similar. 
\end{proof}
\begin{proof}[Proof of \prettyref{prop:stable2-separation}]
    Similar to the proof of \prettyref{prop:stable1-separation}, this proof is analogous to \prettyref{prop:stable2} apart from a more careful analysis of polynomial coefficients. When every $t\in\reals$ is at least $\gamma$ away from all but $k'$ atoms, the left-hand sides of \prettyref{eq:int-g} and \prettyref{eq:int-f} are upper bounded by
    \begin{gather*}
        \int (g(t)\wedge 1) \diff t \le 4k\pth{\frac{2^{2k} \delta}{\gamma^{2(k-k')}}}^{1/(2k')},\\
        \int (f(t)\wedge 1) \diff t \le 4k\pth{\frac{k 4^{2k} \delta}{\gamma^{2(k-k')}}}^{1/(2k')}.
    \end{gather*}
    The remaining proof is similar. 
\end{proof}
\begin{proof}[Proof of \prettyref{prop:accuracy2-separation}]
    The proof is similar to \prettyref{prop:accuracy2}, except that moment comparison theorem \prettyref{prop:stable2} is replaced by its adaptive version \prettyref{prop:stable2-separation}. Recall \prettyref{eq:mr-deconv}:
    \[
        |m_r(\nu')-m_r(\hat \nu)|\le (c\sqrt{k})^{2k}\epsilon,\quad r=1,\dots,2k,
    \]
    where $\nu'=\nu*N(0,\tau^2)$ and $\tau^2=|\sigma^2-\hat\sigma^2|$. Since $\hat\nu*N(0,1)$ has $k_0$ $\gamma$-separated clusters, any $t\in \reals$ can be $\gamma/2$ close to at most $k-k_0+1$ atoms of $\hat\nu$. Applying \prettyref{prop:stable2-separation} yields that 
    \[
        W_1(\nu',\hat\nu)\le 8k\pth{\frac{k  (4c\sqrt{k})^{2k}\epsilon}{(\gamma/2)^{2(k_0-1)}}}^{\frac{1}{2(k-k_0+1)}}.
    \]
    Using \prettyref{lmm:tau-W1} below yields that $\tau\le O_k(W_1(\nu',\hat\nu))$. The conclusion follows from $W_1(\nu',\nu)\le O(\tau)$ and the triangle inequality.
\end{proof}
\begin{lemma}
    \label{lmm:tau-W1}
    Suppose $\pi=\nu*N(0,\tau^2)$ and $\pi'$ is $k$-atomic. Then
    \[
        \tau \le O_k(W_1(\pi,\pi')).
    \]
\end{lemma}
\begin{proof}
    In this proof we write $W_1(X,Y)=W_1(P_X,P_Y)$. Let $Z\sim N(0,1)$, $U\sim\nu$, and $U'\sim\pi'$. For any $x\in \reals$, we have 
    \[
        W_1(x+\tau Z,U')= \tau W_1(Z, (U'-x)/\tau)\ge c_k\tau,
    \]
    where $c_k=\inf\{W_1(Z,Y): Y\text{ is $k$-atomic}\}$\footnote{
        We can prove that $c_k\ge \Omega(1/k)$ using the dual formula \prettyref{eq:W1-dual}. 
    }. 
    For any couping between $U+\tau Z$ and $U'$, 
    \[
        \Expect|U+\tau Z-U'|=\Expect[\Expect[|U+\tau Z-U'| | U]] \ge c_k\tau. \qedhere
    \] 
\end{proof}

\begin{proof}[Proof of \prettyref{thm:separation-W1}]
    By scaling it suffices to consider $M=1$. Recall that the Gaussian mixture is assumed to have $k_0$ $\gamma$-separated clusters in the sense of \prettyref{def:sep}, that is, there exists a partition $S_1,\dots,S_{k_0}$ of $[k]$ such that $|\mu_i - \mu_{i'}| \ge \gamma$ for any $i\in S_\ell$ and $i'\in S_{\ell'}$ such that $\ell\ne \ell'$. Denote the union of the support sets of $\nu$ and $\hat \nu$ by $\calS$. Each atom is $\calS$ is at least $\gamma/2$ away from at least $k_0-1$ other atoms. Then \prettyref{eq:known-variance-W1-separation} follows from \prettyref{prop:stable1-separation} with $\ell=2k$ and $\ell'=(2k-1)-(k_0-1)$. 
\end{proof}

\subsection{Proofs for \prettyref{sec:unbounded}}
\label{sec:pf-unbounded}
\begin{lemma}
    \label{lmm:cluster}
    Assume in the Gaussian mixture \prettyref{eq:model} $w_i \geq \epsilon$, $\sigma=1$. Suppose $L=\sqrt{c\log n}$ in \prettyref{algo:unbounded}. Then, with probability at least $1-ke^{-n'\epsilon}-n^{-(\frac{c}{8}-1)}$, the following holds:
    \begin{itemize}
        \item $\ell_j\le 3kL$ for every $j$. 
        \item Let $X_i= U_i + Z_i$ for $i\in[n]$, where $U_i \sim \nu$ is the latent variable and $Z_i\sim N(0,1)$. Then, $|Z_i|\le 0.5L$ for every $i\in [n]$; $X_i\in I_j$ if and only if $U_i\in I_j$.
    \end{itemize}
\end{lemma}
\begin{proof}
    By the union bound, with probability $1-ke^{-n'\epsilon}-n^{-(\frac{c}{8}-1)}$, the following holds:
    \begin{itemize}
        \item $|Z_i|\le 0.5L$ for every $i\in [n]$.
        \item For every $j\in [k]$, there exists $i\le n'$ such that $U_i=\mu_j$. 
    \end{itemize} 
    Recall the disjoint intervals $I_1\cup \ldots\cup I_s=\cup_{i=1}^{n'}[X_i\pm L]$. Then, we obtain that 
    \[
        \bigcup_{j=1}^{k}[\mu_j\pm 0.5L]~\subseteq~ I_1\cup \cdots\cup I_s
        ~\subseteq~ \bigcup_{j=1}^{k}[\mu_j\pm 1.5L].
    \]
    The total length of all intervals is at most $3kL$. Since $|Z_i|\le 0.5L$, $X_i=U_i+Z_i$ is in the same interval as $U_i$.
\end{proof}

\begin{proof}[Proof of \prettyref{thm:unbounded}]
    Since $n'\ge \Omega(\frac{\log(k/\delta)}{\epsilon})$, applying \prettyref{lmm:cluster} yields that, with probability at least $1-\frac{\delta}{3}-n^{-\Omega(1)}$, the following holds:
    \begin{itemize}
        \item $\ell_j\le O(kL)$ for every $j$. 
        \item Let $X_i= U_i + \sigma Z_i$ for $i\in[n]$ as in \prettyref{lmm:cluster}. Then, $|Z_i|\le 0.5L$ for every $i\in [n]$; $X_i\in I_j$ if and only if $U_i\in I_j$.
    \end{itemize}
    The intervals $I_1,\ldots,I_s$ are independent of every $C_j$ and are treated as deterministic in the remaining proof. We first evaluate the expected moments of samples in $C_j$, conditioned on $|Z_i|\le L'\triangleq 0.5L$. Let $X=U+\sigma Z$ where $U\sim \nu$ and $Z\sim N(0,1)$. Then,
    \[
        \Expect[(X-c_j)^r|X\in I_j, |Z|\le L']
        =\Expect[(X-c_j)^r|U\in I_j, |Z|\le L']
        =\Expect[(U_j'+\sigma Z)^r| |Z|\le L'],
    \]
    where $U_j'=U_j-c_j$, and $U_j\sim P_{U|U\in I_j}$. Since $|U_j'|\le O(kL)$ and $L'=\Theta(\sqrt{\log n})$, the right-hand side differs from the unconditional moment by (see \prettyref{lmm:conditional-moments-normal} in \prettyref{app:aux})
    \[
        |\Expect[(U_j'+\sigma Z)^r| |Z|\le L']-\Expect[(U_j'+\sigma Z)^r]|\le (kL\sigma\sqrt{r})^rn^{-\Omega(1)},\quad r=1,\ldots,2k-1,
    \]
    which is less than $n^{-1}$ when $k\le O(\frac{\log n}{\log\log n})$. Therefore, the accuracy of empirical moments in \prettyref{eq:tilde-mr-rate}, \prettyref{eq:tildemr-rate2} and thus \prettyref{thm:main-W1} are all applicable. Since $w_i\ge \epsilon$, with probability at least $1-\frac{\delta}{3}$, each $C_j$ contains $\Omega(n\epsilon)$ samples, and applying \prettyref{thm:main-W1} yields that, with probability $1-\frac{\delta}{3}$,
    \[
        W_1(\hat\nu_j,\nu_j)\le 
        \begin{cases}
            O(L k^{2.5} (\frac{n\epsilon}{\log (3k/\delta)})^{-\frac{1}{4k-2}}),\quad \sigma\text{ known},\\
            O(L k^3 (\frac{n\epsilon}{\log (3k/\delta)})^{-\frac{1}{4k}}),\quad \sigma\text{ unknown},
        \end{cases}
    \]
    for every $j$, where $\nu_j$ denotes the distribution of $U_j'$ and $\hat\nu_j$ is the estimate in \prettyref{thm:main-W1}. Using the weights threshold $\tau=\epsilon/(2k)$, and applying \prettyref{lmm:w1-hausdorff-2} below, we obtain that
    \[
        d_H(\supp(\hat \nu_j),\supp(\nu_j))\le \frac{W_1(\hat\nu_j,\nu_j)}{\epsilon/(2k)}.
    \]
    The conclusion follows. 
\end{proof}

\begin{lemma}
    \label{lmm:w1-hausdorff-2}
    Let $\nu$ be a discrete distribution whose atom has at least $\epsilon$ probability. Let $S_\nu$ and $S_{\hat\nu}$ denote the support sets of $\nu$ and $\hat\nu$, respectively.
    For $\hat S\subseteq S_{\hat\nu}$, 
    \begin{equation*}
        d_H(S_\nu,\hat S)\le \frac{W_1(\nu,\hat\nu)}{(\min_{y\in \hat S}\hat\nu(y))\wedge (\epsilon-\hat\nu(\hat S^c))_+}.
    \end{equation*}
\end{lemma}
\begin{proof}
    This is a generalization of \prettyref{lmm:w1-hausdorff} in the sense that the minimum weight of $\hat\nu$ is unknown. For any coupling $P_{XY}$ such that $X\sim\nu$ and $Y\sim\hat\nu$, for any $y\in \hat S$,
    \[
        \Expect|X-Y|\ge \hat\nu(y)\Expect[|X-Y| | Y=y]\ge \epsilon_1\min_{x\in S_\nu} |x-y|,
    \]
    where $\epsilon_1=\min_{y\in \hat S}\hat\nu(y)$. Note that $\Prob[Y\not\in \hat S,X=x]\le \hat\nu(\hat S^c)$ and $\nu(x)\ge \epsilon$ for any $x\in S_{\nu}$. Then we have $\Prob[Y\in \hat S,X=x]\ge (\epsilon-\hat\nu(\hat S^c))_+\triangleq \epsilon_2$, and thus
    \[
        \Expect|X-Y|\ge \epsilon_2\Expect[|X-Y| | X=x, Y\in \hat S]\ge \epsilon_2\min_{y\in\hat S}|x-y|.
    \]
    Using the definition of $d_H$ in \prettyref{eq:hausdorff}, the proof is complete.
\end{proof}

\subsection{Proofs for \prettyref{sec:lb}}
\begin{proof}[Proof of \prettyref{lmm:d-matching}]
    Let $U\sim\nu$ and $U'\sim\nu'$. If $\nu$ and $\nu'$ are $\epsilon$-subgaussian, then $\var[U']\le \epsilon^2$, and $\Expect|U|^p, \Expect|U'|^p\le 2(\epsilon\sqrt{p/e})^p$ \cite{subgaussian}. Applying the $\chi^2$ upper bound from moment difference in \prettyref{lmm:chi2-moments} yields that 
    \[
        \chi^2(\nu*N(0,1)\|\nu'*N(0,1))
        \le e^{\epsilon^2/2}\sum_{j\ge \ell+1}\frac{16\epsilon^{2j}}{\sqrt{2\pi j}},
    \]
    where we used Stirling's approximation $n!>\sqrt{2\pi n}(n/e)^n$. If $\nu$ and $\nu'$ are supported on $[-\epsilon,\epsilon]$, the conclusion is obtained similarly by using $\Expect|U|^p, \Expect|U'|^p\le \epsilon^p$.
\end{proof}

\begin{proof}[Proof of \prettyref{prop:lb-known-sigma}]
    Let $\nu$ and $\nu'$ be the optimal pair of distributions for \prettyref{eq:optimal-lb}. Applying \prettyref{lmm:d-matching} yields that 
    \[
        \chi^2(\nu*N(0,1)\| \nu'*N(0,1))\le c\pth{\frac{e\epsilon^2}{2k-1}}^{2k-1},
    \]
    for some absolute constant $c$. The two mixing distributions satisfy (see \prettyref{lmm:max-W1} below)
    \[
        W_1(\nu,\nu')\ge \Omega(\epsilon/\sqrt{k}).
    \]
    The conclusion follows by choosing $\epsilon=c'\sqrt{k}n^{-\frac{1}{4k-2}}$ for some absolute constant $c'$ and applying Le Cam's method \cite{Lecam86}.
\end{proof}
\begin{lemma}
    \label{lmm:max-W1}
    \[
        \sup\{W_1(\nu,\nu'):{\bf m}_\ell(\nu) = {\bf m}_\ell(\nu'),~\nu,\nu'\text{ on }[-1,1]\}=\Theta(\beta/(\ell+1)).
    \]
    Furthermore, the supremum is $\frac{\beta(\pi-o(1))}{\ell+1}$ as $\ell\rightarrow\infty$, and is achieved by two distributions whose support sizes differ by at most one and sum up to $\ell+2$. 
\end{lemma}
\begin{proof}
    It suffices to prove for $\beta=1$. Using the dual characterization of the $W_1$ distance in \prettyref{sec:wass}, the supremum is equal to
    \[
        \sup_{f:1-\mathrm{Lipschitz}}\sup\sth{\Expect_\nu f-\Expect_{\nu'}f:{\bf m}_\ell(\nu) = {\bf m}_\ell(\nu'),~\nu,\nu'\text{ on }[-\beta,\beta]}.
    \]
    Using the duality between moment matching and best polynomial approximation (see \cite[Appendix E]{WY14}), the optimal value is further equal to
    \[
        2\sup_{f:1-\mathrm{Lipschitz}}\inf_{P:\text{degree }\le \ell}\sup_{|x|\le 1}|f(x)-P(x)|.
    \]
    The above value is the best uniform approximation error over 1-Lipschitz functions, a well-studied quantity in the approximation theory (see, \eg, \cite[section 4.1]{Jorge2011}), and thus the optimal values in the lemma are obtained. A pair of optimal distributions are supported on the maxima and the minima of $P^*-f^*$, respectively, where $f^*$ is the optimal 1-Lipschitz function and $P^*$ is the best polynomial approximation for $f^*$. The numbers of maxima and minima differ by at most one by Chebyshev's alternating theorem (see, \eg, \cite[p.~54]{timan63}).
\end{proof}

\begin{proof}[Proof of \prettyref{prop:lb-unknown-sigma}]
    Let $\nu=N(0,\epsilon^2)$ and $\nu'$ be its $k$-point Gauss quadrature. Then ${\bf m}_{2k-1}(\nu)={\bf m}_{2k-1}(\nu')$ and $\nu$ and $\nu'$ are both $\epsilon$-subgaussian (see \prettyref{lmm:quadrature-moments-error}). Applying \prettyref{lmm:d-matching} yields that
    \[
        \chi^2(\nu*N(0,1)\|\nu'*N(0,1))\le O(\epsilon^{4k}).
    \]
    Note that $\nu*N(0,1)=N(0,1+\epsilon^2)$ is a valid Gaussian mixture distribution (with single zero mean component). Between the above two mixture models, the variance parameters differ by $\epsilon^2$; the mean parameters satisfy $W_1(g_k,\delta_0)\ge \Omega(\epsilon/\sqrt{k})$ (see \prettyref{lmm:quadrature-l1}). The conclusion follows by choosing $\epsilon=cn^{-\frac{1}{4k}}$ for some absolute constant $c$ applying applying Le Cam's method \cite{Lecam86}.
\end{proof}
\begin{lemma}
    \label{lmm:quadrature-moments-error}
    Let $g_k$ be the $k$-point Gauss quadrature of $N(0,\sigma^2)$. For $j\ge 2k$, we have $m_j(g_k)\le m_j(N(0,\sigma^2))$ when $j$ is even, and $m_j(g_k)= m_j(N(0,\sigma^2))=0$ otherwise. In particular, $g_k$ is $\sigma$-subgaussian.
\end{lemma}
\begin{proof}
    By scaling it suffices to consider $\sigma=1$. Let $\nu=N(0,1)$. If $j$ is odd, $m_j(g_k)= m_j(\nu)=0$ by symmetry. If $j\ge 2k$ and $j$ is even, the conclusion follows from the integral representation of the error term of Gauss quadrature (see, \eg,~\cite[Theorem 3.6.24]{stoer.2002}):
    \[
        m_j(\nu)- m_j(g_k) =\frac{f^{(2k)}(\xi)}{(2k)!}\int \pi_k^2(x)\diff \nu(x),
    \]
    for some $\xi\in\reals$; here $f(x)=x^{j}$, $\{x_1,\ldots,x_k\}$ is the support of $g_k$, and $\pi_k(x) \triangleq \prod_i (x-x_i)$. Consequently, $g_k$ is $1$-subgaussian \cite[Lemma 2]{subgaussian}.
\end{proof}
\begin{lemma}
    \label{lmm:quadrature-l1}
    Let $g_k$ be the $k$-point Gauss quadrature of $N(0,1)$. Then
    \[
        \Expect_{g_k}|X|\ge (4k+2)^{-1/2},\quad k\ge 2.
    \]
\end{lemma}
\begin{proof}
    Let $G_k\sim g_k$. Note that $|G_k|\le \sqrt{4k+2}$ using the bound on the zeros of Hermite polynomials \cite[p.~129]{orthogonal.poly}. The conclusion follows from $1=\Expect[G_k^2]\le \Expect|G_k|\sqrt{4k+2}$.
\end{proof}
\begin{lemma}
    \label{lmm:quadrature-moments-Hermite}
    Let $g_k$ be the $k$-point Gauss quadrature of $N(0,1)$.
    Then $\Expect_{g_k}[H_{j}]=0$ for $j=1,\dots, 2k-1$, and $\Expect_{g_k}[H_{2k}]=-k!$, where $H_j$ is the Hermite polynomial of degree $j$ (see \prettyref{eq:Hermite1}).
\end{lemma}
\begin{proof}
    Let $Z\sim N(0,1)$ and $G_k\sim g_k$. By orthogonality of Hermite polynomials \prettyref{eq:Hermite-ortho} we have $\Expect[H_j(Z)]=0$ for all $j\ge 1$ and thus $\Expect[H_j(G_k)]=0$ for $j=1,\dots,2k-1$. Expand $H_k^2(x)$ as
    \[
        H_k^2(x)=H_{2k}(x)+a_{2k-1}H_{2k-1}(x)+\dots+a_1H_1(x)+a_0.
    \]
    Since $G_k$ is supported on the zeros of $H_k$, we have $0=\Expect[H_k^2(G_k)]=\Expect[H_{2k}(G_k)]+a_0$. The conclusion follows from $k!=\Expect[H_k^2(Z)]=a_0$ (see \prettyref{eq:Hermite-ortho}). 
\end{proof}

\subsection{Proofs for higher-order mixtures}
\label{sec:large-k}

\begin{proof}[Proof of \prettyref{thm:large-k}]
It suffices to consider $M=1$. Following the proof of \prettyref{thm:main-W1}, we obtain that, with probability $1-\delta$, 
\[
    \Norm{{\bf m}_{2k-1}(\hat \nu)-{\bf m}_{2k-1}(\nu)}_2\le \sqrt{\frac{\log(2k/\delta)}{n}}(\sqrt{ck})^{2k+1}.
\]
The conclusion follows by applying \prettyref{lmm:w1-compare} with $L=2k-1$ and $k=O(\frac{\log n}{\log\log n})$.
\end{proof}

\prettyref{lmm:w1-compare} is a sharpened version of \cite[Proposition 1]{KV17}.
\begin{lemma}
\label{lmm:w1-compare}
Let $\mu$ and $\nu$ be two probability distributions supported on $[-1,1]$. Then
\[
W_1(\mu,\nu)\le \frac{\pi}{L+1}+2(1+\sqrt{2})^L\Norm{\bfm_L(\mu)-\bfm_L(\nu)}_2.
\]
\end{lemma}
\begin{proof}
    Fix any 1-Lipschitz function $f$. 
    Let $ P_L^* $ be the best polynomial of degree $ L $ to uniformly approximate $ f $ over $ [-1,1] $, and denote its coefficients by $ a=(a_1,\dots,a_L) $.
    \begin{align*}
      |\Expect_\mu f-\Expect_{\nu} f|
      & \le |\Expect_\mu (f-P_L^*)|+|\Expect_{\nu}(f-P_L^*)|+|\Expect_\mu P_L^*-\Expect_{\nu} P_L^*|\\
      & \le 2\sup_{x\in [-1,1]}|f(x)-P_L^*(x)|+\sum_{i=1}^{L}|a_i||m_i(\mu)-m_i(\nu)|\\
      & \le \frac{\pi}{L+1}+\Norm{a}_2\Norm{\bfm_L(\mu)-\bfm_L(\nu)}_2,
    \end{align*}
    where in the second inequality we applied the upper bound on the uniform approximation error of 1-Lipschitz functions \cite[Theorem 4.1.1]{Jorge2011}.
    Since $f$ is 1-Lipschitz, it has variation no more than $2$ over $ [-1,1] $ then by the optimality of $P_L^*$ we have $ |P_L^*(x)-a_0|\le 2 $ over $[-1,1]$.
    Applying \prettyref{cor:poly-coeffs} yields that
    \[
    |\Expect_\mu f-\Expect_{\nu} f|
    \le \frac{\pi}{L+1}+2(1+\sqrt{2})^L\Norm{\bfm_L(\mu)-\bfm_L(\nu)}_2.    
    \]
    The conclusion follows by applying \prettyref{eq:W1-dual}.
\end{proof}

\begin{proof}[Proof of \prettyref{prop:large-k-lb}]
Let $\nu$ and $\nu'$ be two discrete distributions obtained from \prettyref{lmm:max-W1} that are supported on at most $\ell+1$ atoms on $[-1,1]$ such that
\[
|\Expect_\nu f-\Expect_\nu' f|\gtrsim \frac{1}{\ell}.
\]
Applying \prettyref{eq:chi2-bdd} with $\ell\asymp \frac{\log n}{\log\log n}$ yields that
\[
\chi^2(\nu*N(0,1)\|\nu'*N(0,1)) \lesssim \frac{1}{n}.
\]
The conclusion follows by Le Cam's method.
\end{proof}

\appendix

\section{Standard form of the semidefinite programming \prettyref{eq:project}}
\label{app:sdp}
Given an arbitrary vector $\tilde{m}=(\tilde{m}_1,\dots,\tilde{m}_r)$, \prettyref{eq:project} computes its projection onto the moment space $\calM_r([a,b])$.
By introducing an auxiliary scalar variable $t$ satisfying $ t\ge \|x\|_2^2$, \prettyref{eq:project} is equivalent to 
\begin{align*}
  \min ~&t-2\Iprod{\tilde{m}}{x}+\|\tilde{m}\|_2^2,\\
  \mathrm{s.t.}~&t\ge \|x\|_2^2,~x~\text{satisfies \prettyref{eq:moment-psd}}.
\end{align*}
This is a semidefinite programming with decision variable $(x,t)$, since the constraint $ t\ge \|x\|_2^2$ is equivalent to
$ \begin{bmatrix}
    t & x^\top\\
    x & I  
\end{bmatrix}\succeq 0
$ using the Schur complement (see, \eg, \cite{VB1996}).

\section{Auxiliary lemmas}
\label{app:aux}

\begin{lemma}
    \label{lmm:divided}
    Let $t_1\le t_2\le \dots$ be an ordered sequence (not necessarily distinct) and $t_r<t<t_{r+1}$. Let $f(x)=\indc{x\le t}$. Then 
    \begin{equation}
        \label{eq:divided}
        f[t_i,\dots,t_j]=(-1)^{i-r}\sum_{L\in\calL(i,j)}\prod_{(x,y)\in L}\frac{1}{t_x-t_y}, \quad i\le r < r+1 \le j,
    \end{equation}
    where $ \calL(i,j) $ is the set of lattice paths from $ (r,r+1) $ to $ (i,j) $ using steps $ (0,1) $ and $ (-1,0) $\footnote{
        Formally, for $a,b\in \naturals^2$, a lattice path from $ a $ to $ b $ using a set of steps $ S $ is a sequence $a=x_1,x_2,\dots,x_n=b$ with all increments $x_{j+1}-x_j\in S$. In the matrix representation shown in the proof, this corresponds to a path from $a_{r,r+1}$ to $a_{i,j}$ going up and right. This path consists of entries $(i,j)$ such that $i\le r<r+1\le j$, and thus in \prettyref{eq:divided} we always have $t_x\le t_r<t_{r+1}\le t_y$.
    }. 
    Furthermore,
    \begin{equation}
        \label{eq:divided-ub}
        |f[t_1,\dots,t_i]|\le \frac{\binom{i-2}{r-1}}{(t_{r+1}-t_{r})^{i-1}},\quad i\ge r+1.
    \end{equation}
\end{lemma}
\begin{proof}
    Denote by $a_{i,j}=f[t_i,\dots,t_j]$ when $i\le j$. It is obvious that $a_{i,i}=1$ for $i\le r$; $a_{i,i}=0$ for $i\ge r+1$; $a_{i,j}=0$ for both $i<j\le r$ and $j>i\ge r+1$. For $i\le r<r+1\le j$, the values can be obtained recursively by 
    \begin{equation}
        \label{eq:recursion}
        a_{i,j}=\frac{a_{i,j-1}-a_{i+1,j}}{t_i-t_j}.
    \end{equation}
    The above recursion can be represented in Neville's diagram as in \prettyref{sec:poly}. In this proof, it is equivalently represented in a upper triangular matrix as follows:
    \begin{equation*}
        \begin{bmatrix}
            1 & 0 & \cdots      & 0      & a_{1,r+1}    &  \cdots  &      &      \\
            & 1   & \ddots      & \vdots & \vdots      &          &      &      \\
            &     & 1           & 0      & a_{r-1,r+1}  &  \cdots  &      &       \\
            &     &             & 1      & a_{r,r+1}    &  \cdots  &      &      \\
            &     &             &        & 0 & \cdots   &  0   &      \\
            &     &\text{\huge0}&        &   & \ddots   &\vdots&      \\
            &     &             &        &   &          &   0  &      
        \end{bmatrix}.
    \end{equation*}
    In the matrix, every $a_{i,j}$ is calculated using the two values left to it and below it. The values on any path from $a_{r,r+1}$ to $a_{i,j}$ going up and right will contribute to the formula of $a_{i,j}$ in \prettyref{eq:divided}. The paths consist of two types: first go to $a_{i,j-1}$ and then go right; first go to $a_{i+1,j}$ and then go up. Formally, $\{L,(i,j):L\in\calL_{i,j-1}\}\cup \{L,(i,j):L\in\calL_{i+1,j}\}=\calL_{i,j}$. This will be used in the proof of \prettyref{eq:divided} by induction present next. The base cases ($r\Th$ row and $(r+1)\Th$ column) can be directly computed:
    \[
        a_{r,j}=\prod_{v=r+1}^j\frac{1}{t_{r}-t_v},\quad a_{i,r+1}=(-1)^{i-r}\prod_{v=i}^{r}\frac{1}{t_v-t_{r+1}}.
    \]
    Suppose \prettyref{eq:divided} holds for both $a_{i,j-1}$ and $a_{i+1,j}$. Then $a_{i,j}$ can be evaluated by
    \begin{align*}
        a_{i,j}&=\frac{(-1)^{i-r}}{t_i-t_j}\pth{\sum_{L\in\calL(i,j-1)}\prod_{(x,y)\in L}\frac{1}{t_x-t_y}+\sum_{L\in\calL(i+1,j)}\prod_{(x,y)\in L}\frac{1}{t_x-t_y}}\\
        &=(-1)^{i-r}\pth{\sum_{L\in\calL(i,j)}\prod_{(x,y)\in L}\frac{1}{t_x-t_y}}.
    \end{align*}
    For the upper bound in \prettyref{eq:divided-ub}, we note that $|\calL(i,j)|\le\binom{(r-1)+(i-(r+1))}{r-1}$ in \prettyref{eq:divided}, and each summand is at most $\frac{1}{(t_{r+1}-t_{r})^{i-1}}$ in magnitude.
\end{proof}

\begin{lemma}
    \label{lmm:Px-expand}
    Let
    \[
        P(x)=\prod_{i=1}^\ell(x-x_i)=\sum_{j=0}^{\ell}a_jx^j.
    \]
    If $|x_i|\le \beta$ for every $i$, then
    \[
        |a_j|\le \binom{\ell}{j}\beta^{\ell-j}.
    \]
\end{lemma}
\begin{proof}
    $P$ can be explicitly expanded and we obtain that
    \[
        a_{\ell-j}=(-1)^j\sum_{\{i_1, i_2,\dots, i_j\}\subseteq [\ell]}x_{i_1}\cdot x_{i_2}\cdot \ldots\cdot x_{i_j}.
    \]
    The summation consists of $\binom{\ell}{j}$ terms, and each term is at most $\beta^j$ in magnitude. 
\end{proof}

\begin{lemma}
    \label{lmm:Hermite-moments}
    If $|\Expect[X^\ell]-\Expect[X'^\ell]|\le (C\sqrt{\ell})^\ell\epsilon$ for $\ell=1,\dots,r$, then, for $\gamma_r$ in \prettyref{eq:Hermite},
    \[
        |\Expect[\gamma_r(X,\sigma)]-\Expect[\gamma_r(X',\sigma)]|\le \epsilon\pth{(2\sigma \sqrt{r/e})^r+(2C\sqrt{r})^r}.
    \]
\end{lemma}
\begin{proof}
    Note that $|\Expect[X^\ell]-\Expect[X'^\ell]|\le \Expect|C\sqrt{e}Z'|^r\epsilon$ by \prettyref{lmm:normal-moments} below, where $Z'\sim N(0,1)$. Then, 
    \begin{equation*}
        |\Expect[\gamma_r(X,\sigma)]-\Expect[\gamma_r(X',\sigma)]|
        \le \sum_{i=0}^{\Floor{r/2}}\frac{r!\sigma^{2i}}{i!(r-2i)!2^i}\Expect[|C\sqrt{e}Z'|^r]\epsilon
        =\epsilon\cdot\Expect[(\sigma Z+|C\sqrt{e}Z'|)^r],
    \end{equation*}
    where $Z\sim N(0,1)$ independent of $Z'$. Applying $(a+b)^r\le 2^{r-1}(|a|^r+|b|^r)$ and \prettyref{lmm:normal-moments} completes the proof.
\end{proof}

\begin{lemma}
    \label{lmm:normal-moments}
    \[
        (p/e)^{p/2}\le \Expect|Z|^p\le \sqrt{2}(p/e)^{p/2},\quad p\ge 0.
    \]
\end{lemma}
\begin{proof}
    Note that
    \[
        \frac{\Expect|Z|^p}{(p/e)^{p/2}}=\frac{2^{p/2}\Gamma(\frac{p+1}{2})}{\sqrt{\pi}(p/e)^{p/2}}\triangleq f(p),\quad \forall~p\ge 0.
    \]
    Since $f(0)=1$ and $f(\infty)=\sqrt{2}$, it suffices to show that $f$ is increasing in $[0,\infty)$.
    Equivalently, $\frac{x}{2}\log\frac{2e}{x}+\log\Gamma(\frac{x+1}{2})$ is increasing, which is equivalent to $\psi(\frac{x+1}{2})\ge \log\frac{x}{2}$ by the derivative, where $\psi(x)\triangleq \frac{\diff}{\diff x}\log\Gamma(x)$.
    The last inequality holds for any $x>0$ (see, \eg, \cite[(3)]{DS2016}).
\end{proof}

\begin{lemma}
    \label{lmm:wedge1-integral}
    Let $r\ge 2$.
    Then,
    \begin{equation*}
        \int \pth{\frac{\delta}{\prod_{i=1}^r|t-x_i|}\wedge 1}\diff t \le 4r\delta^{\frac{1}{r}}.
    \end{equation*}
\end{lemma}
\begin{proof}
    Without loss of generality, let $x_1\le x_2\le \dots \le x_r$.
    Note that
    \begin{equation*}
        \int \pth{\frac{\delta}{\prod_{i=1}^r|t-x_i|}\wedge 1} \diff t
        =\int_{-\infty}^{x_1}+\int_{x_1}^{\frac{x_1+x_2}{2}}+\int_{\frac{x_1+x_2}{2}}^{x_2}+\dots + \int_{x_r}^{\infty}.
    \end{equation*}
    There are $2r$ terms in the summation and each term can be upper bounded by
    \begin{equation*}
        \int_{x_i}^{\infty}\pth{\frac{\delta}{|t-x_i|^{r}}\wedge 1} \diff t 
        =\int_{0}^{\infty}\pth{\frac{\delta}{t^{r}}\wedge 1} \diff t 
        =\frac{r}{r-1}\delta^{\frac{1}{r}}.
    \end{equation*}
    The conclusion follows.
\end{proof}

\begin{lemma}
    \label{lmm:match2k-2}
    Given any $2k$ distinct points $x_1<x_2<\dots<x_{2k}$, there exist two distributions $\nu$ and $\nu'$ supported on $\{x_1,x_3,\dots,x_{2k-1}\}$ and $\{x_2,x_4,\dots,x_{2k}\}$, respectively, such that ${\bf m}_{2k-2}(\nu)={\bf m}_{2k-2}(\nu')$.
\end{lemma}
\begin{proof}
    Consider the following linear equation
    \begin{equation*}
        \begin{pmatrix}
            1 & 1 & \cdots & 1 \\
            x_1 & x_2 & \cdots & x_{2k} \\
            \vdots  & \vdots  & \ddots & \vdots  \\
            x_1^{2k-2} & x_{2}^{2k-2} & \cdots & x_{2k}^{2k-2} 
        \end{pmatrix}
        \begin{pmatrix}
            w_1 \\
            w_2 \\
            \vdots  \\
            w_{2k}
        \end{pmatrix}
        =0,
    \end{equation*}
    This underdetermined system has a non-zero solution. Let $w$ be a solution with $\|w\|_1=2$.
    Since all weights sum up to zero, then positive weights in $w$ sum up to $1$ and negative weights sum up to $-1$.
    Let one distribution be supported on $x_i$ with weight $w_i$ for $w_i>0$, and the other one be supported on the remaining $x_i$'s with the corresponding weights $|w_i|$.
    Then these two distribution match the first $2k-2$ moments.

    It remains to show that the weights in any non-zero solution have alternating signs.
    Note that all weights are non-zero: if one $w_i$ is zero, then the solution must be all zero since the Vandermonde matrix is of full row rank. To verify the signs of the solution, without loss generality, assume that $w_{2k}=-1$ and then
    \begin{equation*}
        \begin{pmatrix}
            1 & \cdots & 1 \\
            x_1 & \cdots & x_{2k-1} \\
            \vdots & \ddots & \vdots  \\
            x_1^{2k-2} & \cdots & x_{2k-1}^{2k-2} 
        \end{pmatrix}
        \begin{pmatrix}
            w_1 \\
            w_2 \\
            \vdots  \\
            w_{2k-1}
        \end{pmatrix}
        =
        \begin{pmatrix}
            1 \\
            x_{2k} \\
            \vdots  \\
            x_{2k}^{2k-2} 
        \end{pmatrix}.
    \end{equation*}
    The solution has an explicit formula that $w_i=P_i(x_{2k})$ where $P_i$ is an interpolating polynomial of degree $2k-2$ satisfying $P_i(x_j)=1$ for $j=i$ and $P_i(x_j)=0$ for all other $j\le 2k-1$.
    Specifically, we have $w_i=\frac{\prod_{j\ne i,j\le 2k-1}(x_{2k}-x_j)}{\prod_{j\ne i,j\le 2k-1}(x_{i}-x_j)}$, which satisfies $w_i>0$ for odd $i$ and $w_i<0$ for even $i$.
    The proof is complete.
\end{proof}

\begin{lemma}[Non-existence of an unbiased estimator]
    \label{lmm:unbiased-not-exist}
    Let $X_1,\dots,X_m\iiddistr pN(s,\sigma^2)+(1-p)N(t,\sigma^2) = \nu * N(0,\sigma^2)$, where $\nu = p \delta_s + (1-p) \delta_t$ and $p,s,t,\sigma$ are the unknown parameters. 
                For any $r\ge 2$, unbiased estimator for the $r\Th$ moments of $\nu$, namely, $ps^r+(1-p)t^r$, does not exist.
\end{lemma}
\begin{proof}
    We will derive a few necessary conditions for an unbiased estimator, denoted by $g(x_1,\dots,x_m)$, and then arrive at a contradiction.
    Expand the function under the Hermite basis
    \[
        g(x_1,\dots,x_m)=\sum_{n_1,\dots,n_m\ge 0} \alpha_{n_1,\dots,n_m} \prod_i H_{n_i}(x_i),
    \]
    and denote by $T_n(\mu,\sigma^2) $ the expected value of the Hermite polynomial $\Expect H_n(X)$ under Gaussian model $X\sim N(\mu,\sigma^2)$.
    Without loss of generality we may assume that the function $g$ and the coefficients $\alpha$ are symmetric (permutation invariant).
    Then, the expected value of the function $g$ under $\sigma^2=1$ is
    \begin{equation}
        \label{eq:g-expand}
        \Expect[g(X_1,\dots,X_m)]
        =\sum_{n_1,\dots,n_m\ge 0} \alpha_{n_1,\dots,n_m}\prod_i (ps^{n_i}+(1-p)t^{n_i}),
    \end{equation}
    which can be viewed as a polynomial in $p$, whereas the target is $ps^r+(1-p)t^r$, a linear function in $p$.
    Matching polynomial coefficients yields that
    \begin{align}
      & \sum_{n_1+\dots+n_m\ge 0} \alpha_{n_1,\dots,n_m} t^{n_1+\dots+n_m}=t^r,\label{eq:g-coeff1}\\
      & \sum_{n_1+\dots+n_m\ge 0} \alpha_{n_1,\dots,n_m} (s^{n_1}-t^{n_1})t^{n_2+\dots+n_m}\cdot m=s^r-t^r,\label{eq:g-coeff2},\\
      & \sum_{n_1+\dots+n_m\ge 0} \alpha_{n_1,\dots,n_m}\prod_{i=1}^{j}(s^{n_i}-t^{n_i})t^{n_{j+1}+\dots+n_m}=0,~\forall~j=2,\dots,m,\label{eq:g-coeff3}
    \end{align}
    where we used the symmetry of the coefficients $\alpha$.
    The equality \prettyref{eq:g-coeff3} with $j=m$ yields that $\alpha_{n_1,\dots,n_m}\ne 0$ only if at least one $n_i$ is zero;
    then \prettyref{eq:g-coeff3} with $j=m-1$ yields that $\alpha_{n_1,\dots,n_m}\ne 0$ only if at least two $n_i$ are zero;
    repeating this for $j=m,m-1,\dots,2$, we obtain that $\alpha_{n_1,\dots,n_m}$ is nonzero only if at most one $n_i$ is nonzero.
    Then the equality \prettyref{eq:g-coeff2} implies that $\alpha_{n_1,\dots,n_m}$ is nonzero only if exactly one $n_i=r$ and the coefficient is necessarily $\frac{1}{m}$.
    Therefore, it is necessary that the symmetric function is $g(x_1,\dots,x_m)=\frac{1}{m}\sum_{i=1}^m H_r(x_i)$.
    However, this function is biased when $\sigma^2\ne 1$.
\end{proof}

\begin{lemma}
    \label{lmm:hankel-det}
    Given a sequence $ \gamma_1,\gamma_2,\dots $, let $ \bfH_j $ denote the Hankel matrix of order $ j+1 $ using $ 1,\gamma_1,\dots,\gamma_{2j} $.
    Suppose $ \det(\bfH_{r-1})\ne 0 $, and $ \det(\bfH_{r})=\det(\bfH_{r+1})=0 $.
    Then,
    \begin{equation*}
        \gamma_{2r+1}=(\gamma_{r+1},\dots,\gamma_{2r})(\bfH_{r-1})^{-1}(\gamma_{r},\dots,\gamma_{2r-1})^\top.
    \end{equation*}
\end{lemma}
\begin{proof}
    The matrices $ \bfH_{r-1} $ and $ \bfH_{r} $ are both of rank $ r $ by their determinants.
    We first show that the rank of $ [\bfH_{r},v] $, which is the first $ r+1 $ rows of $ \bfH_{r+1} $ and is of dimension $ (r+1)\times (r+2) $, is also $ r $, where $ v\triangleq (\gamma_{r+1},\dots,\gamma_{2r+1})^\top $.
    Suppose the rank is $ r+1 $.
    Then $ v $ cannot be in the image of $ \bfH_{r} $.
    By symmetry of the Hankel matrix, the transpose of $ [\bfH_{r},v] $ is the first $ r+1 $ columns of $ \bfH_{r+1} $.
    Those $ r+1 $ columns are linearly independent when its rank is $ r+1 $.
    Since $ \det(\bfH_{r+1})=0 $, then the last column of $ \bfH_{r+1} $ must be in the image of the first $ r+1 $ columns, which is a contradiction.

    Since first $ r $ columns of $ \bfH_{r+1} $ are linearly independent, and the first $ r+1 $ columns of $ \bfH_{r+1} $ are of rank $ r $.
    Then the $ (r+1)^{\rm th} $ column of $ \bfH_{r+1} $ is in the image of the first $ r $ columns, and thus $ \gamma_{2r+1} $ is a linear combination of $ \gamma_{r+1},\dots,\gamma_{2r} $.
    Since $ \bfH_{r-1} $ is of full rank, the coefficients can be uniquely determined by $ (\bfH_{r-1})^{-1}(\gamma_{r},\dots,\gamma_{2r-1})^\top $.
\end{proof}

\begin{lemma}
    \label{lmm:hermite-ub}
    If $|x|>1$, then
    \[
        |H_r(x)|\le (\sqrt{cr}|x|)^r,
    \]
    for some absolute constant $c$.
\end{lemma}
\begin{proof}
    For $|x|>1$,
    \begin{gather*}
        |H_r(x)|\le r!\sum_{j=0}^{\floor{r/2}} \frac{(1/2)^j }{j!(r-2j)!}|x|^{r}
        = |x|^r|H_n({\bf i})|
        = |x|^r|\Expect({\bf i}+{\bf i}Z)^r|\\
        = |x|^r|\Expect(1+Z)^r|
        \le (\sqrt{cr}|x|)^r,
    \end{gather*}
    for some absolute constant $c$, where ${\bf i}=\sqrt{-1}$ and $Z\sim N(0,1)$.
\end{proof}
\begin{lemma}
    \label{lmm:normal-tail}
    Let $Z\sim N(0,1)$.
    \begin{equation*}
        \Prob[Z> M]\le e^{-\frac{M^2}{2}}.
    \end{equation*}
\end{lemma}
\begin{proof}
    Applying Chernoff bound yields that
    \begin{equation*}
        \Prob[Z> M]\le \exp(-\sup_t(tM-t^2/2))=\exp(-M^2/2).\qedhere
    \end{equation*}
\end{proof}
\begin{lemma}
    \label{lmm:normal-tail-moments}
    For $r$ even, and $M\ge 1$,
    \begin{equation*}
        \Expect[Z^r\indc{|Z|> M}]\le r(O(\sqrt{r}))^r\pth{M^{r-1}e^{-\frac{M^2}{2}}}.
    \end{equation*}
\end{lemma}
\begin{proof}
    Applying integral by parts yields that
    \begin{equation*}
        \int_M^\infty x^re^{-\frac{x^2}{2}}\diff x
        =M^{r-1}e^{-\frac{M^2}{2}}+(r-1)M^{r-3}e^{-\frac{M^2}{2}}+(r-1)(r-3)M^{r-5}e^{-\frac{M^2}{2}}+\dots+(r-1)!!\int_M^\infty e^{-\frac{x^2}{2}}\diff x.
    \end{equation*}
    Applying \prettyref{lmm:normal-tail} and $(r-1)!!\le (O(\sqrt{r}))^r$, the conclusion follows.
\end{proof}

\begin{lemma}
    \label{lmm:conditional-moments-normal}
    For $M\ge 1$,
    \begin{equation*}
        0\le \Expect[Z^r]-\Expect[Z^r||Z|\le M]\le r(O(\sqrt{r}))^r\pth{M^{r-1}e^{-\frac{M^2}{2}}}.
    \end{equation*}
\end{lemma}
\begin{proof}
    For $r$ odd, we have $\Expect[Z^r]-\Expect[Z^r||Z|\le M]=0$.
    For $r$ even, the left inequality is immediate since $x\mapsto x^r$ is increasing.
    For the right inequality,
    \begin{equation*}
        \Expect[Z^r]-\Expect[Z^r||Z|\le M]
        =\Expect[Z^r]-\frac{\Expect[Z^r\indc{|Z|\le M}]}{\Prob[|Z|\le M]}
        \le \frac{\Expect[Z^r]-\Expect[Z^r\indc{|Z|\le M}]}{\Prob[|Z|\le M]},
    \end{equation*}
    and the conclusion follows from \prettyref{lmm:normal-tail-moments}.
\end{proof}

\begin{lemma}[Distribution of random projection]
    \label{lmm:proj}
	Let $X$ be uniformly distributed over the unit sphere $S^{d-1}$. For any $a \in S^{d-1}$ and $r>0$,
    \[
        \Prob[|\inner{a, X}|<r] < r\sqrt{d}.
    \]
\end{lemma}
\begin{proof}
    Denote the surface area of the $d$-dimensional unit sphere by $S_{d-1}=\frac{2\pi^{d/2}}{\Gamma(d/2)}$.
    By symmetry,
    \[
        \Prob[|\inner{a,X}|<r] = \Prob[|X_1|<r]
        =\frac{\int_{-r}^r (\sqrt{1-x^2})^{d-2}S_{d-2}\sqrt{1-x^2}\diff x}{S_{d-1}}
        = \frac{2S_{d-2}}{S_{d-1}}\int_{0}^r(1-x^2)^{\frac{d-3}{2}}\diff x < r\sqrt{d},
    \]
    where $X_1$ is the first coordinate of $X$.
\end{proof}

\begin{lemma}[Accuracy of the spectral method]
    \label{lmm:pca}
    Let $X_1,\ldots,X_n\iiddistr \frac{1}{2}N(-\theta,I_d)+\frac{1}{2}N(\theta,I_d)$, where $\theta\in\reals^d$. Let $\hat{\lambda}$ be the largest eigenvalue of $S-I_d$, where $S=\frac{1}{n}\sum_i X_iX_i^\top$ denotes the sample covariance matrix, and $\hat v$ the corresponding normalized eigenvector, where we decree that $\theta^\top \hat v \ge 0$. Let $\hat s=\sqrt{(\hat{\lambda})_+}$ and $\hat\theta=\hat s\hat v$. If $n>d$, then, with probability $1-\exp(-c_0d)$ for some constant $c_0$,
    \[
        \Norm{\theta-\hat \theta}_2\le O(d/n)^{1/4}.
    \]
\end{lemma}
\begin{proof}
    The samples can be represented in a matrix form $X=\theta\varepsilon^\top +Z\in\reals^{d\times n}$, where $\varepsilon\in\reals^n$ is a vector of independent Rademacher random variables, and $Z$ has independent standard normal entries. Using $\varepsilon^\top \varepsilon =n$, we have
    \[
        S-I_d=\theta\theta^\top +B+C,
    \]
    where $B=\frac{1}{n}ZZ^\top -I_d$ and $C=\frac{1}{n}(\theta\varepsilon^\top Z^\top +Z\varepsilon\theta^\top )$ are both symmetric. When $n>d$, we have $\Norm{B},\Norm{C}\le O(\sqrt{d/n})$ with probability $1-\exp(-c_0d)$ for some constant $c_0$ (see \cite[Theorem II.13]{DS2001}). Then, $|\hat{\lambda}-\Norm{\theta}_2^2|\le O(\sqrt{d/n})$ by Weyl's inequality, and thus $|\hat s-\Norm{\theta}_2|\le O(d/n)^{1/4}$. Since $\hat v$ maximizes $|u^\top (S-I_d)u|$ among all unit vectors $u\in\reals^d$, including the direction of $\theta$, then we obtain that $(\theta^\top \hat v)^2\ge \Norm{\theta}_2^2-O(\sqrt{d/n})$, and consequently,
    \[
        \Norm{\theta-\Norm{\theta}_2\hat v}_2^2\le O(\sqrt{d/n}).
    \]
    The conclusion follows from the triangle inequality.
\end{proof}



\begin{lemma}
    \label{lmm:boundary}
    The boundary of the space of the first $2k-1$ moments of all distributions on $\reals$ corresponds to distributions with fewer than $k$ atoms, while the interior corresponds to exactly $k$ atoms.
\end{lemma}
\begin{proof}
    Given $m=(m_1,\dots,m_{2k-1})$ that corresponds to a distribution of exactly $k$ atoms, by \cite[Theorem 2A]{Lindsay1989}, the moment matrix $\bfM_{k-1}$ is positive definite.
    For any vector $m'$ in a sufficiently small ball around $m$, the corresponding moment matrix $\bfM_{k-1}'$ is still positive definite.
    Consequently, the matrix $\bfM_{k-1}'$ is of full rank, and thus  $m'$ is a legitimate moment vector by \cite[Theorem 3.4]{Lasserre2009} (or \cite[Theorem 3.1]{CF1991}).
    If $m$ corresponds to a distribution with exactly $r<k$ atoms, by \cite[Theorem 2A]{Lindsay1989}, $\bfM_{r-1}$ is positive definite while $\bfM_{r}$ is rank deficient.
    Then, $m$ is no longer in the moment space if $m_{2r}$ is decreased.
\end{proof}

\begin{lemma}
    \label{lmm:poly-unit-circle}
    For polynomial $p$ of degree $L$ such that $|p(x)|\le 1$ on $ [-1,1] $, we have $ |p(z)|\le (1+\sqrt{2})^L $ for any $z$ on the complex unit circle.
\end{lemma}
\begin{proof}
    Let $ f(y)\triangleq p(\frac{y+y^{-1}}{2})/y^L $ which is analytic and bounded on $ |y|\ge 1 $.
    For $ y=e^{i\theta} $, $ |f(y)|=|p(\cos\theta)|\le 1 $.
    By the maximum modulus principle, $ |f(y)|\le 1 $ for any $ |y|>1 $.
    Consider $ |z|=1 $ and let $ \frac{y+y^{-1}}{2}=z $ for some $ |y|\ge 1 $.
    Then $ y=z\pm\sqrt{z^2-1} $ and by triangle inequality $ |y|\le 1+\sqrt{2} $.
    Since $ |f(y)|\le 1 $, then $ |p(z)|\le |y|^L\le (1+\sqrt{2})^L $.
\end{proof}
\begin{corollary}
    \label{cor:poly-coeffs}
    For polynomial $p(x)=\sum_{i=0}^L a_ix^i$ such that $|p(x)|\le 1$ on $ [-1,1] $, we have $\sum_i |a_i|^2\le (1+\sqrt{2})^{2L}$.
\end{corollary}
\begin{proof}
The sum of squares of its coefficients is given by the following compact formula:
\begin{equation*}
    \sum_{i=0}^{L}|a_i|^2=\frac{1}{2\pi}\oint_{|z|=1}|p(z)|^2 \diff z.
\end{equation*}
The conclusion follows from \prettyref{lmm:poly-unit-circle}.
\end{proof}

\section*{Acknowledgment}
        \label{sec:ack}

We are grateful to Philippe Rigollet for bringing~\cite{HK2015} to our attention and Harry Zhou for pointing out~\cite{Chen95}. 
We thank Roger Koenker for discussions on NPMLE and sharing his experimental results.
We also thank Sivaraman Balakrishnan for helpful comments on \cite{balakrishnan2017statistical,Nguyen2013}.

\newcommand{\etalchar}[1]{$^{#1}$}

\end{document}